\documentclass[11pt]{amsart}

\usepackage{setspace} 
\usepackage{fullpage}

\usepackage{mathptmx} 
\usepackage[T1]{fontenc}
\usepackage{amsfonts}

\usepackage{mathtools}
\usepackage{graphicx}

\usepackage{braket} 

\newtheorem{thm2}{Theorem}

\newtheorem{lem}[subsubsection]{Lemma}
\newtheorem{prop}[subsubsection]{Proposition}

\newtheorem{claim*}{Claim}

\newtheorem{rem}[subsubsection]{Remark}

\newcommand{\ve}{\varepsilon}
\newcommand{\wh}{\widehat}
\newcommand{\wt}{\widetilde}
\newcommand{\mc}{\mathcal}
\newcommand{\intr}{\text{int}\,}
\newcommand{\cl}{\text{cl}\,}
\newcommand{\fr}{\text{fr}\,}
\newcommand{\grp}[1]{\langle #1\rangle}

\newcommand{\ov}{\overline}

\newcommand{\bpr}{\begin{proof}}
\newcommand{\epr}{\end{proof}}

\newcommand{\mA}{\mathbb{A}}
\newcommand{\mD}{\mathbb{D}}
\newcommand{\mT}{\mathbb{T}}
\newcommand{\mZ}{\mathbb{Z}}
\newcommand{\mS}{\mathbb{S}}

\newcommand{\ben}{\begin{enumerate}}
\newcommand{\een}{\end{enumerate}}

\newcommand{\Figw}[4]{
\includegraphics[width=#1]{#2}
\caption{ #3 \label{#4} } }

\begin{document}

\title{The Kakimizu complex for genus one hyperbolic knots \linebreak[1] in the 3-sphere}%

\author[L. G. Valdez-S\'anchez]{Luis G. Valdez-S\'anchez}
\address{Department of Mathematical Sciences,
University of Texas at El Paso\\
El Paso, TX 79968, USA}
\email{lvsanchez@utep.edu}%

\subjclass[2020]{Primary 57K10; Secondary 57K30}%
\keywords{Hyperbolic knot, genus one knot, Seifert torus, Kakimizu complex.}%

\date{\today}

\begin{abstract}
The Kakimizu complex $MS(K)$ for a knot $K\subset\mS^3$ is the simplicial complex with vertices the isotopy classes of minimal genus Seifert surfaces in the exterior of $K$ and simplices any set of vertices with mutually disjoint representative surfaces.
 In this paper we determine the structure of the Kakimizu complex $MS(K)$ of genus one hyperbolic knots $K\subset\mS^3$. In contrast with the case of hyperbolic knots of higher genus, it is known that the dimension $d$ of $MS(K)$ is universally bounded by $4$, and we show that $MS(K)$ consists of a single $d$-simplex for $d=0,4$ and otherwise of at most two $d$-simplices which intersect in a common $(d-1)$-face. For the cases $1\leq d\leq 3$ we also construct infinitely many examples of such knots where $MS(K)$ consists of two $d$-simplices.
\end{abstract}

\maketitle

\section{Introduction}

Let $K$ be a knot in the 3-sphere $\mS^3$ with exterior $X_K=\mS^3\setminus \intr\,N(K)$, where $N(K)\subset\mS^3$ is regular neighborhood of $K$. The knot $K$ is the boundary of orientable, compact and connected surfaces embedded in $\mS^3$, called Seifert surfaces for the knot. Equivalently, there is a unique slope $J\subset\partial X_K$, the longitude of $K$, that bounds orientable compact surfaces in $X_K$ which correspond in a natural way to the Seifert surfaces in $\mS^3$ bounded by the knot.
The genus of the knot $K$, a topological invariant of $K$,  is then defined as the smallest genus of the Seifert surfaces bounded by the knot. 

The Kakimizu complex $MS(K)$ was defined in \cite{kakimizu3} for knots (and links) $K\subset\mS^3$ as the simplicial complex with vertices the equivalence isotopy classes of minimal genus Seifert surfaces for $K$ in $X_K$, such that any set of vertices with mutually disjoint representative surfaces comprise a simplex. For instance, it is well known that the figure eight knot bounds a unique Seifert torus and so its Kakimizu complex consists of a single 0-simplex.

For hyperbolic knots $K\subset\mS^3$ much is known about the complex $MS(K)$. It is a consequence of \cite{eisner1} that $MS(K)$ is a finite complex. It is also known that 
$MS(K)$ is a flag simplicial complex \cite{schultens1} which is 
connected \cite{schar6} and  contractible \cite{schultens2}.

For the family of hyperbolic knots $K\subset\mS^3$ of a fixed genus $g\geq 2$ no universal bound on the dimension of $MS(K)$ is known. Moreover, Y.\ Tsutsumi \cite{tsutsumi2} shows that for each genus $g\geq 2$ there are hyperbolic knots $K\subset\mS^3$ of genus $g$ such that the number of vertices of $MS(K)$, and hence of simplices, is arbitrarily large. 

In \cite{sakuma3} M.\ Sakuma and K.\ J.\ Schakleton provide a bound for the diameter of the (1-skeleton of) $MS(K)$, quadratic in the genus of the knot, and show that the diameter of $MS(K)$ is 1 or 2 for genus one knots, with an example that realizes the diameter of 2. These bounds on the diameter of $MS(K)$ do not however bound the number of top-dimensional simplices present in $MS(K)$.

In this paper we determine the structure of the Kakimizu complex for genus one hyperbolic knots $K\subset\mS^3$ and obtain a picture opposite to that of the case of genus $g\geq 2$.
Our main result is the following.

\begin{thm2}\label{main1}
If $K\subset\mS^3$ is a genus one hyperbolic knot and $d=\dim\,MS(K)$
then

\begin{enumerate}
\item
$0\leq d\leq 4$,

\item
$MS(K)$ consists of at most two $d$-dimensional simplices which intersect in a common $(d-1)$-face, and exactly one $d$-simplex if $d=0,4$,

\item
for each integer $1\leq d\leq 3$ 
there are infinitely many genus one hyperbolic knots $K\subset\mS^3$ such that $MS(K)$ consists of two $d$-simplices.
\end{enumerate}
\end{thm2}

A simplex of $MS(K)$ corresponds to a collection of mutually disjoint and nonparallel Seifert tori $\mT=T_1\sqcup\cdots\sqcup T_N\subset X_K$. We refer to $\mT$ as a {\it simplicial collection} of Seifert tori for short and assume that its components are labeled consecutively as they appear around $X_K$, as shown in Fig.~\ref{f21}. The collection $\mT$ is maximal if its number of components $|\mT|=N$ is largest among all possible simplicial collections of Seifert tori in $X_K$, in which case $\dim\,MS(K)=N-1$.
Components $T_i,T_{j}$ of $\mT$ then cobound a region $R_{i,j}$ in $X_K$ with boundary a surface of genus two that contains the longitude slope $J$ of $K$. 

It was proved in \cite{valdez14} that the exterior $X_K$ of a genus one hyperbolic knot in $\mS^3$ contains at most 5 mutually disjoint and nonparallel Seifert tori, which gives the bound $\dim\,MS(K)\leq 4$ in Theorem~\ref{main1}(1). This was achieved
by studying the properties of maximal simplicial collections $\mT$ of Seifert tori in $X_K$ with at least 5 components; all but at most one of the complementary regions $R_{i,i+1}\subset X_K$ of $\mT$ are then genus two handlebodies whose structure was determined in \cite{valdez14}. 

It is the low genus of the complementary handlebody regions of $\mT$ that makes it possible to establish Theorem~\ref{main1}(2) and construct the examples in Theorem~\ref{main1}(3). One of the difficulties here, which need not be handled in detail in \cite{valdez14}, is that some region $R_{i,i+1}\subset X_K$ in a maximal simplicial collection may not be a handlebody.

The paper is organized as follows. Most of Sections~\ref{prelim} and \ref{ann3} contain background material and extension of definitions from \cite{valdez14} adapted to the needs of the present paper. Specifically, in Section~\ref{prelim} we introduce the notation and some basic results that are used throughout the paper. The definition of a pair $(H,J)$ given in \cite{valdez14}, that $H$ is a genus two handlebody and $J\subset\partial H$ a separating circle which is nontrivial in $H$, is updated to include 3-manifolds other than handlebodies. 
Pairs of the form $(R_{i,j},J)$, where $J\subset\partial R_{i,j}$ is the longitudinal slope in $\partial X_K$, are then naturally produced by any maximal simplicial collection of Seifert tori $\mT\subset X_K$.
The structure of several types of handlebody pairs is also described in some detail.

In Section~\ref{stk} (see Lemma~\ref{new01c}) the structure of $MS(K)$ is shown  to depend on the presence of {\it annular pairs} $(R_{i,j},J)$: pairs such that $R_{i,j}$ contains an incompressible spanning annulus with one boundary component in $T_i$ and the other in $T_j$. 

The properties of general annular pairs are established in Section~\ref{ann3}, while the properties of the annular pairs in $X_K\subset\mS^3$ for $K$ a hyperbolic knot are developed further in Section~\ref{min5}. Two major restrictions arise once we restrict our attention to annular pairs $(R_{i,j},J)$ in $\mS^3$. The first one concerns the case where $R_{i,j}$ is not a genus two handlebody. In this case $R_{i,j}$ must be an atoroidal, irreducible and boundary irreducible manifold which is the complement in $\mS^3$ of a genus two handlebody $V\subset\mS^3$. In the context of \cite{ozawa01}, $V$ is a nontrivial genus two handlebody knot in $\mS^3$ and as such it has a restricted structure outlined in Lemma~\ref{koz1-b}. The second one, the content of Proposition~\ref{annmin}, restricts even further the structure of the annular pair $(R_{i,j},J)$ so that $R_{i,j}=R_{i,i+1}$ or $R_{i,j}=R_{i,i+2}$; that is, $R_{i,j}$ may contain at most one Seifert torus not parallel to $T_i$ or $T_j$. It is also established in Proposition~\ref{annmin} that $MS(K)$ has more than one top-dimensional simplex iff for each maximal simplicial collection $\mT\subset X_K$ there is at least one annular pair of the form $(R_{i,i+2},J)$.

We call an annular pair of the form $(R_{i,i+2},J)$ an {\it exchange pair}, given that 
there is a Seifert torus $T'_{i+1}$ in $R_{i,i+2}$ that can be exchanged with $T_{i+1}\subset R_{i,i+2}$ to construct a new simplicial collection $\mT'$ not isotopic to $\mT$ in $X_K$. The properties of this {\it exchange trick} are established in Sections~\ref{52} and \ref{trick}.

In Section~\ref{classif} we show that any maximal simplicial collection $\mT\subset X_K$ of size $5$ does not produce exchange pairs and hence that such a collection $\mT$ is unique up to isotopy. Maximal simplicial collections $\mT\subset X_K$ of size $2\leq |\mT|\leq 4$ are handled in Section~\ref{maxcoll}, where it is shown in Proposition~\ref{mainp} that $\mT$ produces at most one exchange pair and hence that, up to isotopy, there are most two maximal simplicial collections of Seifert tori in $X_K$.
The results so far make possible the proof of parts (1) and (2) of Theorem~\ref{main1} by the end of Section~\ref{maxcoll}.

Section~\ref{examples} is devoted to the construction of examples of genus one hyperbolic knots in $\mS^3$ with maximal simplicial collections satisfying various conditions.
The extreme examples in \S\ref{82} where $|\mT|=4$ and $MS(K)$ consists of two 3-simplices are particularly challenging to construct since each of the 4 complementary regions $R_{i,i+1}$ of $\mT$ must be genus two handlebodies. 
These examples establish Theorem~\ref{main1}(3) in the case $\dim\,MS(K)=3$.

For the cases $\dim\,MS(K)=1,2$ in Theorem~\ref{main1}(3) there are two possible types of examples, depending on whether the exchange region $R_{i,i+2}$ is a handlebody or not. These examples  are constructed in Sections~\ref{hk23} and \ref{hk23-2}.

For the examples in Section~\ref{hk23-2} where the exchange region $R_{i,i+2}$ is not a handlebody we use the construction in \cite{ozawa01} of genus two nontrivial handlebody knots in $\mS^3$ whose exteriors contain essential annuli and of
excellent 1-submanifolds of 3-manifolds in \cite{myers1}. Separating the cases $\dim\,MS(K)=1$ and $\dim\,MS(K)=2$ requires the use of another special type of handlebody pair, the general {\it basic pair}, whose classification is discussed in Section~\ref{FF}.
The proof of Theorem~\ref{main1}(3) is given at the end of
Section~\ref{hk23-2}.

Examples of knots $K\subset\mS^3$ with a maximal simplicial collection of size 4 as above are not easy to render in a regular projection. However, making use of basic hyperbolic pairs, in Section~\ref{ht4} we construct an infinite family of genus one hyperbolic knots $K(p,q)\subset\mS^3$ with $MS(K)$ consisting of a single $3$-simplex, all of which have the simple projections shown in Fig.~\ref{k04-3}. The smallest member $K(2,2)$ of this family is represented in Fig.~\ref{k04-2e} with a crossing number of 141 along with the structure of the pairs in their exteriors $X_K$ (see Section \S\ref{hpairs} for definitions).

\begin{figure}
\Figw{4in}{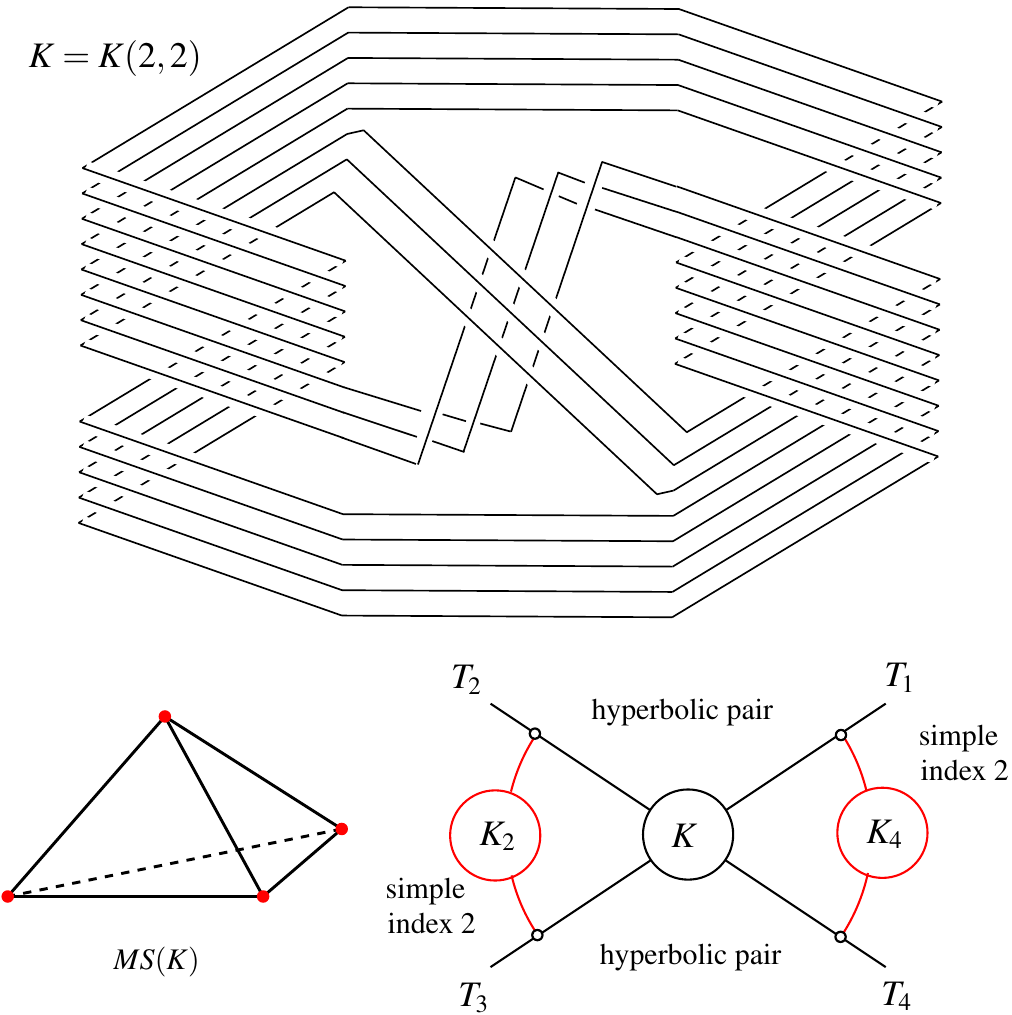}{A genus one hyperbolic knot $K=K(2,2)\subset\mS^3$ with  $MS(K)$ a single 3-simplex.}{k04-2e}
\end{figure}


In Section~\ref{hk23} an infinite family of genus one hyperbolic knots $K=K(-1,n,2)$, $|n|\geq 2$, with at most $14+6|n|$ crossings is constructed such that $MS(K)$ consists of two 2-dimensional simplices. The structure of their exteriors is represented in Fig.~\ref{m19} (where $\boxed{n}$ stands for $n$ full twists).

\begin{figure}
\Figw{5.5in}{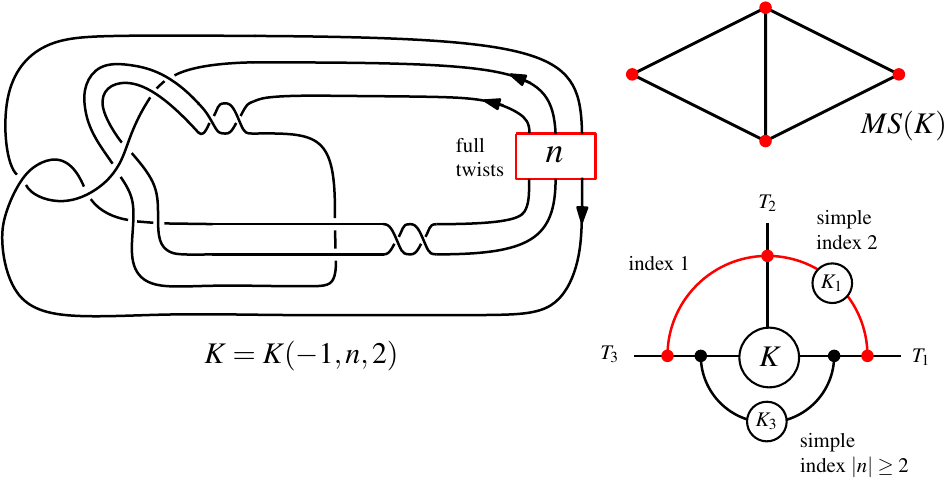}{The genus one hyperbolic knot $K=K(-1,n,2)\subset\mS^3$ with  $MS(K)$ consisting of two 2-simplices.}{m19}
\end{figure}

Examples of genus one satellite knots $K\subset\mS^3$ for which the dimension of $MS(K)$ is arbitrarily large, showing that the restriction to hyperbolic knots is necessary, can be explicitly constructed as follows: Let $A_1,A_2\subset\mS^3$ be trivial unlinked and  untwisted annuli connected by a band $B$ whose core follows the pattern of a connected sum of nontrivial knots $K_1\,\#\cdots\#\,K_n$, $n\geq 2$, with $K$ the boundary component of the resulting pair of pants indicated in Fig.~\ref{k30}. The knot $K$ is then a nontrivial zero winding number satellite of each knot $K_i$.
Attaching an annulus to the free boundary components of $A_1$ and $A_2$ that swallows the factors $K_1\,\#\cdots\#\,K_s$ and follows the factors $K_{s+1}\,\#\cdots\#\,K_n$ produces a Seifert torus $T_s\subset X_K$. It is not hard to see that the Seifert tori $T_1,\dots,T_n\subset X_K$ can be constructed so as to be mutually disjoint and hence nonparallel in $X_K$.

\begin{figure}
\Figw{3.5in}{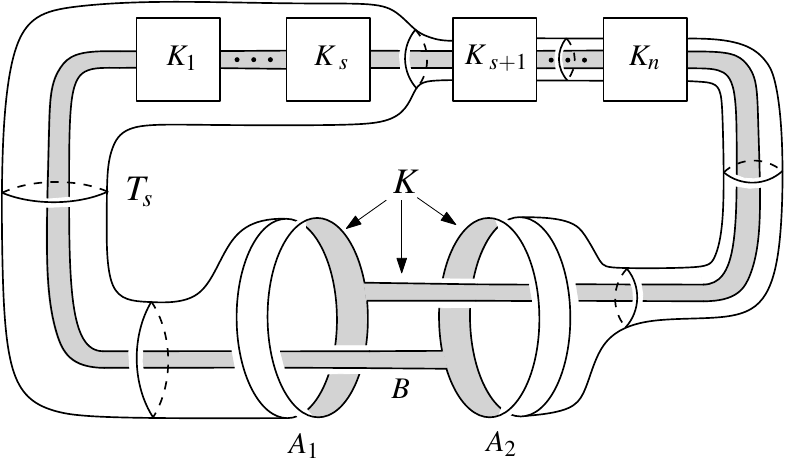}{The satellite knot $K\subset\mS^3$ and the swallow-follow Seifert torus $T_s\subset X_K$.}{k30}
\end{figure}

On the structure of maximal simplicial collections $\mT\subset X_K$ some questions remain unresolved. 
For instance, an infinite family of genus one hyperbolic knots $K\subset\mS^3$ with $|\mT|=5$ was constructed in \cite{valdez14}, all of whose pairs $(R_{i,i+1},J)$ are of a type called {\it simple} (see Section~\ref{hpairs}). 
It is not known if hyperbolic knots with a simplicial collection of size $|\mT|=5$ can be constructed where at least one pair $(R_{i,i+1},J)$ is not simple. 

More specifically, in the case $|\mT|=5$ there are two options for a nonsimple pair $(R_{i,i+1},J)$: a {\it primitive pair} or a {\it hyperbolic pair} (see Lemma~\ref{annp}). 
Realizing the case $|\mT|=5$ where one of the pairs $(R_{i,i+1},J)$ is {\it primitive} could produce an example of a hyperbolic knot in $\mS^3$ with a nonintegral Seifert surgery. One example realizing such a primitive pair is constructed in Proposition~\ref{K4} (see Fig.~\ref{k15b-3}, bottom) but with a maximal simplicial collection of size $|\mT|=3$.

We want to thank J.\ Schultens and M.\ Eudave-Mu\~noz, organizers of the 2021 BIRS-CMO virtual workshop {\it Knots, surfaces and 3-manifolds}, and K. Baker, organizer of the 2023 conference {\it Low Dimensional Topology and Circle-valued Morse Functions} at UMiami-IMSAC, where preliminary versions of some of the results in this paper were presented. This research was partially funded by a generous gift from Microsoft Co.\ (2022-2023).

\section{Preliminaries}\label{prelim}

We work in the PL category. For definitions of basic concepts see \cite{hempel} or \cite{jaco}. 
We will make use of many of the definitions and results in \cite{valdez14}, some of which are reproduced throughout the paper.

Unless otherwise stated, submanifolds are assumed to be properly embedded.  For $A$ a submanifold of $B$, $\cl(A)$, $\intr(A)$ and $\fr(A)=\cl(\partial A\setminus\partial B)$ denote the closure, interior and frontier  of $A$ in $B$, respectively.

If $A$ is a finite set or a manifold then $|A|$ denotes the cardinality or the number of components of $A$, respectively.

The isotopy class of a circle in a surface is the {\it slope} of the circle. The circle is {\it nontrivial} if it does not bound a disk in the surface. 

Any two circles $\alpha,\beta$ in a surface can be isotoped so as to intersect transversely and minimally, in which case $|\alpha\cap\beta|_{\text{min}}$ denotes their minimal number of intersections. 

The algebraic intersection number between 1-submanifolds $\alpha,\beta$ of a surface will be denoted by $\alpha\cdot\beta\in\mZ$.

Let $M$ be a 3 manifold and $F\subset M$ a surface. The components of $\partial F\subset\partial M$ are sometimes indexed as  $\partial F=\partial_1F\sqcup\partial_2F\sqcup\cdots$.

The manifold  obtained by cutting $M$ along the surface $F\subset M$ is denoted by $M|F=\cl[M\setminus N(F)]$.
Two surfaces $F,G$ in $M$ are {\it parallel} if they cobound a product region in $M$ of the form $F\times I$, where $F\times\{0\}$ corresponds to $F$ and $F\times\{1\}$ to $G$. 

Observe that if two properly embedded surfaces $F$ and $G$ in $M$ intersect transversely and $|\partial F\cap\partial G|$ is not as small as possible then it is possible to isotope $F$ and $G$ near $\partial M$ to reduce $|\partial F\cap\partial G|$ without increasing $|F\cap G|$. Hence we will say that $F$ and $G$ intersect {\it minimally} if they intersect transversely so that the pair $(|\partial F\cap\partial G|,|F\cap G|)$ is smallest in the lexicographic order, in which case $|F\cap G|_{\text{min}}$ denotes the number of components in the such minimal intersection.

For a 1-submanifold $\Gamma\subset\partial M$ we let $M(\Gamma)$ denote the 3-manifold obtained by adding 2-handles to $\partial M$ along the components of $\Gamma$ and capping off any resulting 2-sphere boundary components with 3-balls. For a surface $F\subset M$ we denote by $\wh{F}\subset M(\partial F)$ the closed surface obtained by capping off $\partial F$ with disks.

Let $\gamma\subset\partial M$ be a circle which is nontrivial in $M$. An annulus $A\subset M$ is a {\it companion annulus} for $\gamma$ if the circles $\partial A$ cobound an annular regular neighborhood of $\gamma$ in $\partial M$ and $A$ is not parallel to $\partial M$. The following result on the properties of companion annuli is established in \cite[Lemma 5.1]{valdez7}.

\begin{lem}[\cite{valdez7}]\label{compa2}
Let $M$ be an irreducible and atoroidal 3-manifold with boundary. If a circle $\gamma\subset \partial M$ has a companion annulus $A\subset M$ then
\begin{enumerate}
\item
$A$ is unique up to isotopy,

\item
$A$ cobounds with $\partial M$ a companion solid torus $V$ around which $A$ runs $p\geq 2$ times.
\hfill\qed
\end{enumerate}
\end{lem}

We denote by $F(a,b,\dots)$ a Seifert fiber space over the surface $F$ with singular fibers of indices $a,b,\dots\geq 1$. Typically the surface $F$ will be a disk $\mD^2$, an annulus $\mA^2$ or a 2-sphere $\mS^2$.
$L_p\neq\mS^3,\mS^1\times\mS^2$ stands for a lens space with finite fundamental group of order $p\geq 2$

\subsection{Genus one hyperbolic knots}

With very few exceptions, for the rest of this paper we restrict our attention to hyperbolic knots in $\mS^3$.
For notation, let $K\subset\mS^3$ be a genus one hyperbolic knot and $J\subset\partial X_K=\partial N(K)$ the longitudinal slope of $K$. 

Recall that by a {\it simplicial collection of Seifert tori in $X_K$} we mean a collection $\mT\subset X_K$ of mutually disjoint and nonparallel Seifert tori in $X_K$. 
The collection $\mT$ is {\it maximal} if its number of components $|\mT|$ is as large as possible. By \cite{valdez14} we have that $|\mT|\leq 5$, with $|\mT|=5$ being the largest attainable bound.

The components of $\mT$ are labeled consecutively following their order of appearance around $\partial X_K$ as $T_1,T_2,\dots,T_N$, $N=|\mT|$. For $|\mT|>2$ we denote by $R_{i,i+1}$ the region in $X_K$ cobounded by $T_i$ and $T_{i+1}$ which contains no components of $\mT$ other than $T_i,T_{i+1}$; if $|\mT|=2$ then a region cobounded by $T_1\sqcup T_2$ is chosen as $R_{1,2}$, and if $\mT=T_1$ then we define $R_{1,1}$ as the complement of a product region $\cl(X_K\setminus T_1\times[-1,1])$. Here we interpret a label $i$ modulo $N=|\mT|$, so $N+1=1$ etc. 
Notice that $\partial T_i$ and $\partial T_{i+1}$ cobound the annulus $\partial R_{i,i+1}\cap\partial X_K$ whose core has slope $J$ in $\partial X_K$.

More generally we set $R_{i,i}=\cl(X_K\setminus T_i\times [-1,1]\,)$ and $R_{i,j}=R_{i,i+1}\cup\cdots\cup R_{j-1,j}$. A simplicial collection $\mT\subset X_K$ of size $|\mT|=5$ is represented in Fig.~\ref{f21}. 

\begin{figure}
\Figw{2.5in}{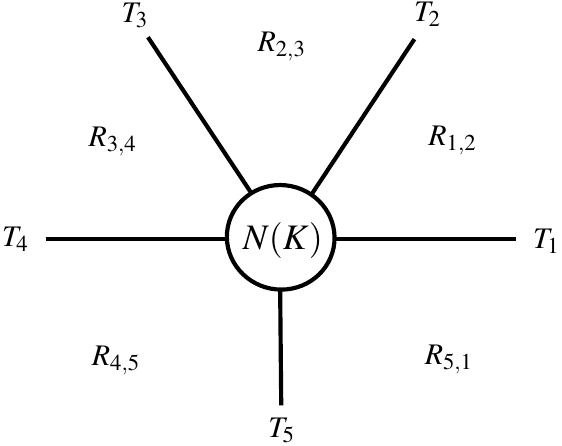}{The knot $K\subset\mS^3$ and a simplicial collection $\mT=T_1\sqcup\cdots\sqcup T_5\subset X_K$.}{f21}
\end{figure}

The next result summarizes the general properties of the regions $R_{i,j}\subset X_K$, which follow from \cite[Lemmas 3.7 and 4.1]{valdez14}.

\begin{lem}[\cite{valdez14}]\label{genprop}
Let $K\subset\mS^3$ be a hyperbolic knot with a simplicial collection of Seifert tori $\mT\subset X_K$.
\begin{enumerate}
\item[(P1):] The manifold $R_{i,j}$ is either a genus two handlebody or an irreducible, boundary irreducible, atoroidal 3-manifold, 

\item[(P2):] if $R_{i,j}$ is not a handlebody then $R_{j,i}$ is a handlebody,

\item[(P3):] if $R_{k,\ell}\subseteq R_{i,j}$ and $R_{i,j}$ is a handlebody then $R_{k,\ell}$ is a handlebody,

\item[(P4):] at most one region $R_{i,i+i}$ is not a handlebody, and if such a region is present then $|\mT|\leq 4$.
\hfill\qed
\end{enumerate}
\end{lem}

In \cite{valdez14} a pair $(H,J)$ consists of  a genus two handlebody $H$ and a separating circle $J\subset\partial H$ which is nontrivial in $H$; they were used to model the handlebody regions $R_{i,j}$ produced by a simplicial collection $\mT\subset X_K$.
By \cite[Lemma 4.3]{valdez14}, if $|\mT|= 5$ then all regions $R_{i,i+1}$ are genus two handlebodies, but in the cases $|\mT|\leq 4$ some region $R_{i,i+1}$ may not be a handlebody. In the next section we update this definition of a pair appropriately so as to be able to deal with nonhandlebody regions $R_{i,j}$.

\subsection{Pairs}
A {\it pair} $(H,J)$ consists of an irreducible, atoroidal, connected 3-manifold $H$ with boundary a genus two surface and $J\subset \partial H$ a separating circle which in $H$ is nontrivial and has no companion annulus.

In the pair $(H,J)$ the circle $J$ separates $\partial H$ into two once-punctured tori $T_1,T_2$ such that $\partial H=T_1\cup_{J}T_2$, each of which is necessarily incompressible in $H$.

For convenience, a once-punctured torus in $H$ with boundary slope $J$ will be called a {\it $J$-torus}.

A pair $(H,J)$ is {\it minimal} if any $J$-torus in $H$ is parallel into $T_1$ or $T_2$.

The next result shows that handlebodies and atoroidal regions $R_{i,j}\subset X_K$ for arbitrary genus one knots $K\subset\mS^3$ satisfy this more general definition of pair.

\begin{lem}\label{Rpair}
\begin{enumerate}
\item
If $(H,J)$ is a pair then
\begin{enumerate}
\item
$H$ is either boundary irreducible or a genus two handlebody.

\item
If $H$ is a handlebody then $(H,J)$ is a pair for any nontrivial separating circle $J\subset\partial H$.
\end{enumerate}

\item
If $K\subset\mS^3$ is a genus one knot and $\mT=T_1\sqcup\cdots\sqcup T_N\subset X_K$ is a simplicial collection of Seifert tori such that the region $R_{i,j}$ is atoroidal, as is the case when $K$ is hyperbolic, then $(R_{i,j},J)$ is a pair.
\end{enumerate}
\end{lem}

\begin{proof}
Part (1)(a) follows as in the proof of \cite[Lemma 4.1]{valdez14} and (1)(b) from \cite[Lemma 3.3]{valdez14}.

For (2), if $R_{i,j}$ is atoroidal and the circle $J\subset\partial R_{i,j}$ has a companion annulus then by Lemma~\ref{compa2} $J$ has a companion solid torus $V\subset R_{i,j}$ around which it runs $p\geq 2$ times, which implies that $R_{i,j}(J)$ has a lens space connected summand $L_p$. However, by \cite[Corollary 8.3]{gabai03} the manifold $X_K(J)$ is irreducible and each torus $\wh{T}_i\subset X_K(J)$ is incompressible, hence the manifold $R_{i,j}(J)\subset X_K(J)$ is irreducible. This contradiction shows that $J$ has no companion annuli in $R_{i,j}$ and hence that $(R_{i,j},J)$ is a pair.
\end{proof}

\subsection{Handlebody pairs}\label{hpairs}

If $H$ is a genus two handlebody then we call $(H,J)$ a {\it handlebody pair}. Handlebody pairs were introduced in \cite{valdez14} and play a prominent role in the structure of a genus one knot exterior. In this section we gather the main examples and properties of handlebody pairs. 

As usual we write $\partial H=T_1\cup_J T_2$. Two nonseparating circles in $T_1$ and $T_2$ are {\it coannular} if they cobound an annulus in $H$.

\subsubsection{Primitive and power circles.}\label{prim-pwr}

The fundamental group of handlebody $H$ (of any genus) is a free group. We say that a circle in $\partial H$ is {\it primitive} or a {\it power} in $H$ if it represents a primitive or a power $p\geq 2$ of a nontrivial element in the fundamental group $\pi_1(H)$ of $H$, in which case the circle must be nonseparating in $\partial H$ (cf \cite{valdez14}). By \cite[Lemma 3.3]{valdez14},
a circle $\omega\subset\partial H$ is primitive or a power in $H$ iff the surface $\partial H\setminus\omega$ compresses in $H$.

\medskip
By \cite{cassongor}, if $\omega\subset\partial H$ is a power circle then $\omega$ is a power of a primitive element in $\pi_1(H)$. Equivalently, the circle $\omega\subset\partial H$ is a power circle iff it has a companion annulus in $H$, in which case $\omega$ is a power of primitive element of $\pi_1(H)$ represented by the core of its companion solid torus.

In the special case that $H$ is a genus two handlebody, if $w(x,y)$ is a cyclically reduced word representing a primitive element of the free group $\pi_1(H)=\grp{x,y \ | \ -}$ other than $x,x^{-1},y,y^{-1}$ then by \cite{cohen} there is $\ve\in\{\pm1\}$ and an integer $n\in\mZ$ such that in $w(x,y)$ all exponents of $x$ ($y$) are $\ve$ and all exponents of $y$ ($x$, respectively) are $n$ or $n+1$. The same conclusion holds when $w(x,y)$ is a power of some primitive word $w'(x,y)$, as $w'(x,y)$ must then be cyclically reduced.

\subsubsection{Circles with companion annuli in general pairs.}\label{pcgp}
By \cite[Lemma 3.1]{valdez14}, if $(H,J)$ is a general pair with $\partial H=T_1\cup_JT_2$ and $i\in\{1,2\}$ then up to isotopy there is at most one circle $\omega_i\subset T_i$ which has a companion annulus and companion solid torus in $H$, and these companion objects are unique in $H$ up to isotopy.

\medskip

Handlebody pairs $(H,J)$ include the following types.
Here we write $\partial H=T_1\cup_J T_2$.

\subsubsection{Trivial pair}
$H=T\times I$ for $T$ a once-punctured torus and $J$ the slope of the core of the annulus $(\partial T)\times I$. By \cite[Lemma 3.7(4)]{valdez14} a handlebody pair $(H,J)$ with $\partial H=T_1\cup_JT_2$ is trivial iff $H(J)\approx\wh{T}_1\times I$.

\subsubsection{Simple pair:}\label{simple}
$H=(T\times I)\cup_B V$ where $V$ is a solid torus, $B=(T\times\{0\})\cap\partial V$ is a closed annular neighborhood of a nonseparating circle $\omega\times\{0\}$ in $T\times\{0\}$, and $B$ runs $p_1\geq 2$ times around $V$. The separating circle $J$ corresponds to $\partial T\times\{1/2\}$. The core of the annulus $\partial V\setminus B\subset T_2$ is then coannular in $H$ to $\omega\times\{0\}\subset T_1$.

Simple pairs are discussed in detail in 
\cite[\S3.2 and \S6.1]{valdez14}.
In this case, for $\partial H=T_1\cup_J T_2$, there are coannular $p$-power circles $\omega_1=\omega\times \{0\}\subset T_1$ and $\omega_2\subset T_2$ which cobound a nonseparating annulus $A$ in $H$ and there is a nonseparating disk $D\subset H\setminus A$, all unique under isotopy. 

The minimal intersection of $D$ and $J$ satisfies $|D\cap J|=2$. In fact, by \cite[Lemma 3.11]{valdez14} a handlebody pair $(H,J)$ is trivial or simple iff there is a disk in $H$ which minimally intersects $J$ in two points.

\begin{figure}
\Figw{5.5in}{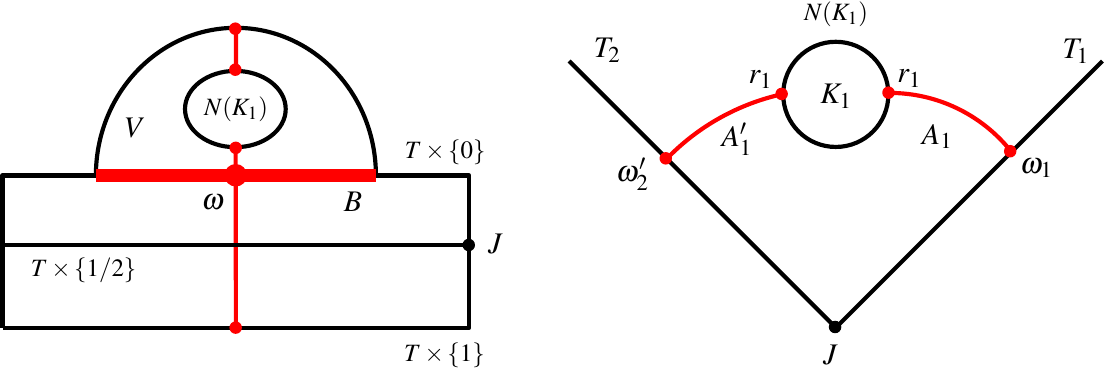}{The core knot $K_1$ and the annuli $A_1,A'_1$ in a simple pair $(H,J)$.}{k26}
\end{figure}

The core $K_1$ (defined up to isotopy) of the solid torus $H|D$ obtained by cutting $H$ along $D$ is called the {\it core of the pair $(H,J)$}. 

Thus $K_1$ is isotopic in $H=(T\times I)\cup_B V$ to the core of the solid torus $V$. It follows that, in the exterior $XH_{K_1}=H\setminus\intr\,N(K_1)$ of $K_1$ in $H$, there are disjoint annuli $A_1,A'_1$ such that 
\[
\partial_1A_1=\omega_1\subset T_1,\quad \partial_1A'_1=\omega'_1\subset T_2,\quad
\partial_2A_1\sqcup\partial_2A'_1\subset\partial N(K_1)
\]
and each circle $\partial_2A_1,\partial_2A'_1$ has nonintegral boundary slope in $\partial N(K_1)$ of the form $r_1=a_1/p_1$, $\gcd(a_1,p_1)=1$. These objects are represented in Fig.~\ref{k26}. Fig.~\ref{m02}, top left, shows an actual simple pair $(H,J)$ with the disk $D\subset J$ that intersects $J$ in two points.

The {\it index of the simple pair $(H,J)$} is defined as the integer $p_1\geq 2$; we also say that its core $K_1$ has index $p_1$.

\subsubsection{Operations with simple pairs}\label{oper}
Let $(H,J)$ be a pair with $\partial H=T_1\cup_JT_2$ and $T\subset H$ a $J$-torus such that $H|T$ consists of two components $H_1,H_2$, with $H=H_1\cup_TH_2$ and $\partial H_i=T\cup_JT_i$. Suppose that $(H_1,J)$ is a simple pair of index $p\geq 2$ and $\omega_1\subset T_1$ and $\omega\subset T$ are the coannular $p$-power circles in $H_1$ in \S\ref{simple}. The following observations will be useful in the analysis of general pairs.

\begin{enumerate}
\item
{\it If $V_1\subset H_1$ is the companion solid torus of $\omega_1\subset T_1$ then $H_2\approx\cl[H\setminus V_1]$.} Equivalently, if $A_1\subset H_1$ is the companion annulus of $\omega_1\subset T_1$ then the components of $H|A_1$ are homeomorphic to $H_2$ and $V_1$. 

\item
{\it If $V\subset H_1$ is the companion solid torus of $\omega\subset T$ then $H\approx H_2\cup V$.}

\item
{\it $H$ is a handlebody iff $H_2$ is a handlebody and $\omega\subset T$ is a primitive circle in $H_2$.}

\item
{\it If $H$ is a handlebody and $T$ is not parallel into $\partial H$ then at least one of the pairs $(H_1,J)$, $(H_2,J)$ is simple, hence there is a circle in $T_1$ or $T_2$ which is a power in $H$.}
\end{enumerate}
(1) and (2) follow by construction of the simple pair $(H_1,J)$ and Fig.~\ref{k26}. (3) is the content of \cite[Lemma 6.3]{valdez14}, and (4) follows from \cite[Lemma 3.7(3)]{valdez14}.

\subsubsection{Primitive pair:} \label{primitive} 
A nontrivial pair $(H,J)$ such that there is a nonseparating annulus $A\subset H$ with each boundary component $\partial_1 A\subset T_1$ and $\partial_2 A\subset T_2$ a primitive circle in $H$. By \cite[Lemma 6.9]{valdez14} the circle 
$\partial_iA\subset T_i$ is the unique circle in $T_i$ which is primitive in $H$, and $A$ is also unique up to isotopy.

\subsubsection{Basic pair:}\label{basic}
We say that a pair of circles $\omega_1\subset T_1$ and $\omega_2\subset T_2$ are basic in $H$ if, relative to some base point $*$, the circles represent a basis for the free group $\pi_1(H,*)$. By \cite[\S3]{valdez14} this is equivalent to saying that in $H$ the circles $\omega_1\subset T_1$ and $\omega_2\subset T_2$ are primitive and separated by a disk.

A pair $(H,J)$ is {\it basic} if it contains a pair of basic circles as above.

\subsubsection{Double pair:} \label{double}
This is a variation of a simple pair, essentially the union of two simple pairs. Let $T$ be a once-punctured torus and $\omega_0,\omega_1\subset T$ two circles which intersect minimally in one point. Then the circles $\omega_0\times\{0\}\subset T\times\{0\}$ and $\omega_1\times\{1\}\subset T\times\{1\}$ are basic in $T\times [0,1]$. Attaching solid tori $V_0,V_1$ to $T\times I$ along annular neighborhoods $B_0\subset T\times\{0\},B_1\subset T\times\{1\}$ of $\omega_0\times\{0\},\omega_1\times\{1\}$, with the circles $\omega_0\times\{0\},\omega_1\times\{1\}$ running $p_0,p_1\geq 2$ times around $V_0,V_1$, respectively, produces a handlebody $H=T\times I\cup_{B_0}V_0\cup_{B_1}V_1$, with $J$ corresponding to the circle $(\partial T)\times\{1/2\}$.

The $J$-torus $T\times\{1/2\}\subset H$ separates $H$ into two handlebodies $H_0\supset T\times\{0\}$ and $H_1\supset T\times\{1\}$ such that $(H_0,J)$ and $(H_1,J)$ are simple pairs of indices $p_0,p_1$, respectively, with the power circles $\omega_0\times\{1/2\}\subset T\times\{1/2\}\subset H_0$ and $\omega_1\times\{1/2\}\subset T\times\{1/2\}\subset H_1$ intersecting minimally in one point in $T\times\{1/2\}$.

By \cite[Lemma 6.8(2)(a)]{valdez14}, any $J$-torus in the double pair $(H,J)$ is parallel to a boundary $J$-torus or to the $J$-torus $T$ that splits it into the simple subpairs $(H_0,J)$ and $(H_1,J)$.

Fig.~\ref{m02}, top right, shows a double pair $(H,J)$ that splits into simple subpairs of index 2.

\subsubsection{Maximal pair:} \label{max}
If $(H,J)$ is a genus two handlebody pair then there are at most two $J$-tori in $H$ that are mutually disjoint and nonparallel, and not parallel to $\partial H$; this follows from Lemma~\ref{compa2} and the fact that any $J$-torus in $H$ is boundary compressible (cf \cite{tsutsumi1}). The pair $(H,J)$ is {\it maximal} if it contains two such $J$-tori.

By \cite[Lemma 6.8]{valdez14}, a maximal pair $(H,J)$ contains two disjoint $J$-tori $T'_1,T'_2\subset H$ such that, in $H$, the $J$-tori $T_i$ and $T'_i$ cobound a simple pair $(H_i,J)$ for $i=1,2$ and $T'_1,T'_2$ cobound a nontrivial basic pair $(H_0,J)$. Specifically, the circles $\omega'_1\subset T'_1$ and $\omega'_2\subset T'_2$ that are power circles in $H_1$ and $H_2$, respectively, are basic circles in $H_0$. The situation is represented in Fig.~\ref{m02}, bottom left.

Moreover no circle in $T_1$ or $T_2$ is primitive in $H$, as otherwise by 
\S\ref{oper}(2),(3) it would be possible to construct a handlebody pair $(H',J)$ that contains 3 $J$-tori that are mutually disjoint and nonparallel, and not parallel to $\partial H'$, which is impossible.

\subsubsection{Induced simple pair and induced $J$-torus:}\label{induced}
Suppose that $(H,J)$ is a general pair and $\omega_1\subset T_1$ a circle with companion annulus $A\subset H$. Let $V\subset H$ be the companion solid torus cobounded by $A$ and an annulus neighborhood $B\subset T_1$ of $\omega_1$ (cf Lemma~\ref{compa2}), such that $\omega_1$ runs $p\geq 2$ times around $V$.

Then a small regular neighborhood $H'=N(T_1\cup V)\subset H$ is a genus two handlebody and $(H',J)$ is a simple pair of index $p$ with 
$\omega_1$ a $p$-power circle in $H'$. 

We say that the pair $(H',J)$ and $J$-torus $T'_1=\fr(H')\subset H$ are the simple pair and $J$-torus {\it induced} by $T_1$, and more specifically by the power circle $\omega_1\subset T_1$. Since by \S\ref{pcgp} the $J$-torus $T_1$ contains at most one circle with a companion annulus and solid torus, all of which are unique up to isotopy, it follows that the $J$-torus $T'_1$ and simple pair $(H',J)$ induced by $T_1$ are unique in $H$ up to isotopy.
The situation is represented in Fig.~\ref{m02}, bottom right.

\begin{figure}
\Figw{5.7in}{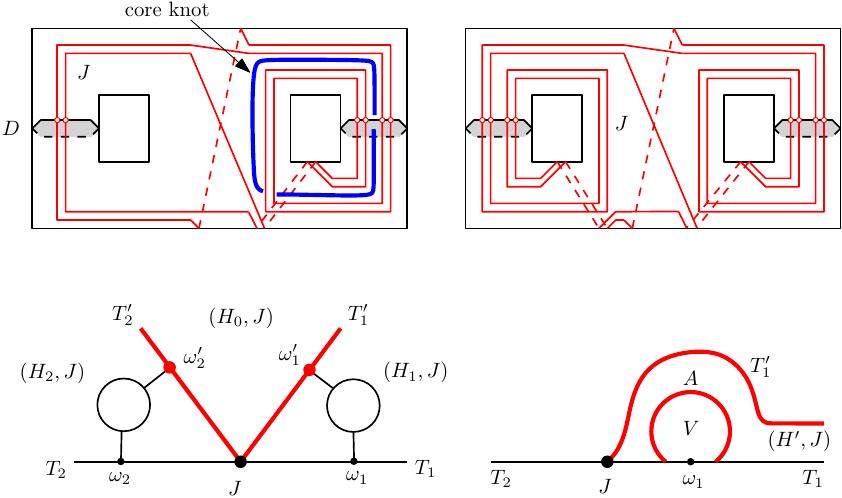}{Examples of simple, double, maximal and induced pairs.}{m02}
\end{figure}

\

Minimal handlebody pairs are characterized as follows.

\begin{lem}\label{ann2}
If $(H,J)$ is a minimal nontrivial handlebody pair then 
\begin{align*}
H(J)=
\begin{cases}
\mA^2(p) & \text{if }(H,J)\text{ is a simple pair of index }p\geq 2\\
\text{toroidal} & \text{if }(H,J)\text{ is a primitive pair,}\\
\text{hyperbolic} & \text{otherwise.}
\end{cases}
\end{align*}
\end{lem}

\begin{proof}
By \cite[Lemma 3.7(1)]{valdez14} the manifold $H(J)$ is irreducible and boundary irreducible, so if $H(J)$ is anannular and atoroidal then it is hyperbolic by \cite{thurs2}.

Since $H$ is a handlebody, by \cite[Theorems 1 and 2]{eudave3} if $H(J)$ contains an incompressible annulus or torus which is not boundary parallel then $H$ contains an incompressible annulus $A$ disjoint from $J$ which is not boundary parallel.

Let $\partial H=T_1\cup_J T_2$. If $\partial A\subset T_i$ 
then by \cite[Lemma 3.3(2)]{valdez14} $A$ is a companion annulus for a power circle in $T_i$ and hence the pair $(H,J)$ is simple by \cite[Lemma 6.2]{valdez14}. In this case the manifold $H(J)$ is a cable space of the form $\mA^2(*)$.

Suppose now that $\partial_1A\subset T_1$ and $\partial_2A\subset T_2$. By \cite[Lemma 3.4]{valdez14} the components of $\partial A$ are both power circles or both primitive circles in $H$. In the first case the pair $(H,J)$ is simple by \cite[Lemma 6.2]{valdez14}. 

In the second case the pair $(H,J)$ is primitive. Let $W$ be a regular neighborhood of $\wh{T}_1\cup\wh{T}_2\cup A$ in $H(J)$. Then $W$ is a composing space of the form $P\times\mS^1$ for some pants $P$ and $\partial W=\wh{T}_1\sqcup\wh{T}_2\sqcup T$ where $T$ is a torus. Since $H(J)$ is irreducible and boundary irreducible, if $T$ compresses in $H(J)$ then it bounds a solid torus $V\subset H(J)$ and hence $H(J)=W\cup_T V$ is a Seifert fiber space of the form $\mA^2(*)$. Thus $H(J)(\beta)$ is an atoroidal Seifert fiber space of the form $\mD^2(*,*)$ for each circle $\beta\subset T_1$ with $\Delta(\beta,\partial_1A)\geq 2$, contradicting \cite[Lemma 6.9]{valdez14} that the manifold  $H(\beta)=H(J)(\beta)$ is toroidal. Therefore $H(J)$ is a toroidal manifold.
\end{proof}

In light of Lemma~\ref{ann2}, we will say that a handlebody pair $(H,J)$ is {\it hyperbolic} if the manifold $H(J)$ is hyperbolic. 

Examples of hyperbolic pairs are provided by some basic pairs as established in the next result.

\begin{lem}\label{prim-2}
\ben
\item  
Primitive, simple, basic and hyperbolic handlebody pairs are minimal,

\item
a primitive pair is neither basic nor simple,

\item
a basic pair is trivial, simple or hyperbolic.
\een
\end{lem}

\bpr
Simple pairs are minimal by \cite[Lemma 3.9(2)]{valdez14}.

If $(H,J)$ is a handlebody pair and $T\subset H$ is a $J$-torus not parallel into $\partial H=T_1\cup_J T_2$ then 
$H_1(J)\neq \wh{T}\times I\neq H_2(J)$ by \cite[Lemma 3.7(4)]{valdez14}
and so $\wh{T}$ is an incompressible torus in $H(J)=H_1(J)\cup_{\wh{T}}H_2(J)$ that is not boundary parallel. It follows that any hyperbolic handlebody pair is minimal. 

We claim that if a handlebody pair $(H,J)$ is simple or nonminimal then there is a circle $\beta\subset\partial H\setminus J$ which is a $p$-power circle in $H$ for some $p\geq 2$. In such case $H(\beta)$ is a reducible manifold of the form 
$H(\beta)=\mS^1\times\mD^2\,\#\,L_p$ for some lens space $L_p$ with finite fundamental group of order $p$.

\medskip

In the case where $(H,J)$ is a simple pair a $p$-power circle $\beta$ exists in each component of $\partial H\setminus J$ by definition.
If there is a $J$-torus $T\subset H$ which is not parallel to $T_1$ or $T_2$ then by \cite[Lemma 3.7(3)]{valdez14} $T$ separates $H$ into two genus two handlebodies $H_1$, $H_2$ with $\partial H_i=T\cup T_i$, where we may assume that $(H_1,J)$ is a simple pair. Thus there is a circle $\beta\subset T_1$ which is a power circle in $H_1$ and hence in $H$.

Now, if $(H,J)$ is a primitive pair then by \S\ref{primitive}
the circles $\alpha_1\subset T_1$ and $\alpha_2\subset T_2$ which are primitive in $H$ are unique and coannular, hence not basic in $H$, so $(H,J)$ is not a basic pair, and by \cite[Lemma 6.9]{valdez14} the manifold 
$H(\beta)$ is irreducible for each circle $\beta\subset T_1$ other than the primitive circle $\alpha_1$, so $(H,J)$ is minimal and not a simple pair by \S\ref{oper}(4).

Suppose now that the pair $(H,J)$ is basic, with basic circles $\alpha_1\subset T_1$ and $\alpha_2\subset T_2$ that are separated by a disk $D\subset H$ (see \S\ref{basic}), and $T\subset H$ is a $J$-torus which is not parallel to $T_1$ or $T_2$. We may assume that $D$ intersects $T$ minimally, so that $D\cap T$ consists of a nonempty collection of arcs. Let $E\subset D$ be a subdisk cut off by an arc in $D\cap T$ which is outermost in $D$, with $(\intr\, E)\cap T=\emptyset$, and suppose that $E\subset H_1$. Then $|E\cap J|=|E\cap\partial T|=2$ and so the pair $(H_1,J)$ is simple by
\cite[Remark 3.8 and Lemma 3.11]{valdez14}. As the circle $\alpha_1\subset T_1$ is primitive in $H$, it must be primitive in $H_1$ and hence it must intersect $E$ minimally in one point by  \cite[Lemma 6.2(5)]{valdez14}. Thus $\alpha_1$ intersects $D$, which is not the case. This contradiction shows that the basic  pair $(H,J)$ is minimal. Therefore (1) and (2) hold, and (3) follows now by Lemma~\ref{ann2}
\end{proof}

\subsection{Construction of basic pairs} \label{baspa}
Recall from Lemma~\ref{prim-2} that any basic pair is trivial, simple or hyperbolic. In this section we construct all basic pairs and give simple conditions to determine their nature.

In preparation for this we set up the following items:
\begin{enumerate}
\item[(i)]
a genus two handlebody $H$,

\item[(ii)]
circles $\omega_1,\omega_2\subset\partial H$ that are basic in $H$ and separated by a disk $D\subset H$,

\item[(iii)]
a decomposition $H=V_1\cup\,(D\times I)\,\cup V_2$ where $V_1,V_2$ are solid tori with meridian disks $D_1\subset V_1$, $D_2\subset V_2$, such that $|D_1\cap\omega_1|_{\text{min}}=1=|D_2\cap\omega_2|_{\text{min}}$,

\item[(iv)]
a decomposition $\partial H=S_1\cup A\cup S_2$ with the once-punctured tori and annulus
$S_1=\partial V_1\cap\partial H$, $S_2=\partial V_2\cap\partial H$, and $A=(\partial D)\times I$.

\item[(v)]
We remark that $D_1$ and $D_2$ are up to isotopy the only disks in $H$ that satisfy the conditions 
\[
|D_i\cap\omega_j|=\delta_{i,j}=\begin{cases}
1 &i=j\\
0 & i\neq j
\end{cases} \quad\text{for all }\{i,j\}=\{1,2\}
\]
\end{enumerate}

If $(H,J)$ is a basic pair with basic circles $\omega_1,\omega_2$ then $J$ can be isotoped in $\partial H\setminus(\omega_1\sqcup\omega_2)$ to intersect the circles $\partial D_1\sqcup\partial D_2\sqcup\partial A$ minimally, in which case

\begin{figure}
\Figw{4.7in}{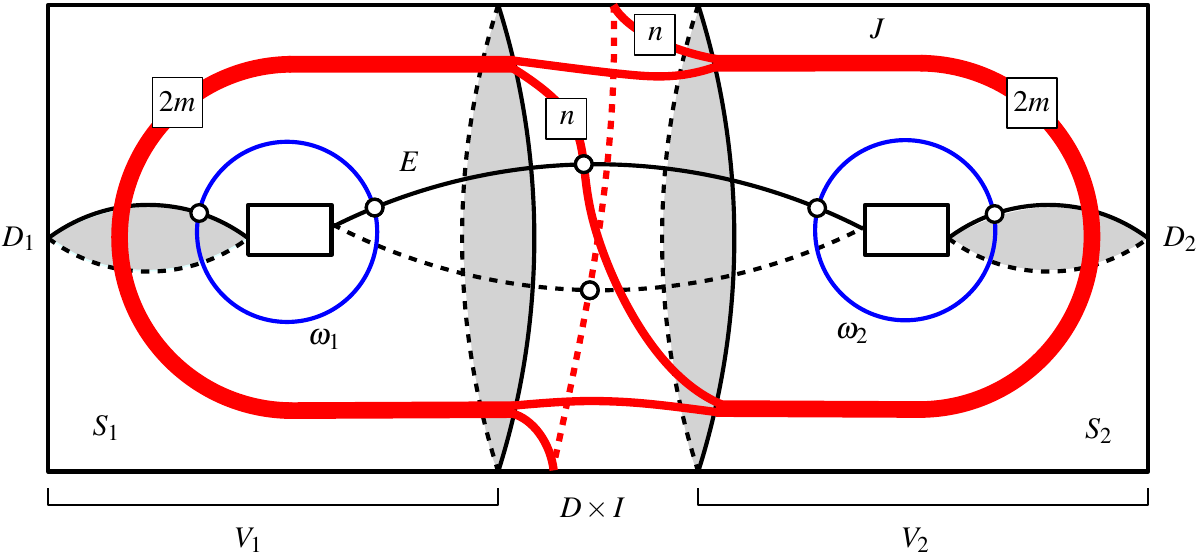}{Construction of the basic pair $(H,J)$.}{k04-6}
\end{figure}

\begin{enumerate}
\item[(vi)]
for $i=1,2$, $J\cap S_i$ is a collection of $2m\geq 2$  parallel arcs disjoint from $\omega_i$,

\item[(vii)]
the arcs $J\cap A$ split into 4 collections of parallel arcs each of size $n$ or $2m-n$, where $n$ is an integer such that $1\leq n\leq m$ and $\gcd(n,2m)=1$.
\end{enumerate}
The situation is represented in Fig.~\ref{k04-6}, where each arc represents one of the collections of parallel arcs in (vi)-(vii) of the size indicated by the number in the box of top of the arc.

It is then always possible to construct a nonseparating disk $E\subset H$ which satisfies the following properties:
\begin{enumerate}
\item[(E1)] $E\cap D_1=\emptyset=E\cap D_2$,

\item[(E2)] $E$ intersects each circle $\omega_1$ and $\omega_2$ minimally in one point,

\item[(E3)] $E$ intersects $J$ minimally in $2n$ points.
\end{enumerate}
The boundary of one such disk $E$ is shown in Fig.~\ref{k04-6}.

\begin{lem}\label{FF}
Any disk in $H$ which satisfies (E2) is isotopic to a disk obtained by performing some number of half-Dehn twists to $E$ along the separating disk $D$ and hence intersects $J$ in at least $2n$ points.
\end{lem}

\begin{proof}
It suffices to show that any disk in $H$ which satisfies (E2) can be isotoped to satisfy (E2) and be disjoint from $D_1\sqcup D_2$.

So let $F\subset H$ be a disk that satisfies (E2). It is then possible to isotope $F$ so that it satisfies (E2) and intersects $D_1\sqcup D_2$ minimally; in particular $F\cap(D_1\sqcup D_2)$ has no circle components.

If $F\cap (D_1\sqcup D_2)\neq\emptyset$, say $F\cap D_1\neq\emptyset$ for definiteness, then there is an outermost arc $c$ of the graph $F\cap D_1\subset D_1$ which cobounds a disk face $D_0\subset D_1$ disjoint from $\omega_1$. The frontier of $N(F\cup D_0)$ then consists of three disks $F_0,F_1,F_2$ properly embedded in $H$, say with $F_0$ parallel to $F$.

If the arc $c\subset F$ separates the points $F\cap\omega_1$ and  $F\cap\omega_2$ then may assume that $|F_i\cap\omega_j|=\delta_{i,j}$ for all $\{i,j\}=\{1,2\}$. By (v) it follows that in $H$ the disks $F_1$ and $F_2$ are isotopic to $D_1$ and $D_2$, respectively, and hence that $F$ can be isotoped to be disjoint from $D_1\sqcup D_2$, contradicting our hypothesis. 

If the arc $c\subset F$ does not separate the points $F\cap\omega_1$ and $F\cap\omega_2$ then one of the disks, say $F_1$, is disjoint from the basic circles $\omega_1\sqcup\omega_2$ and hence must be parallel into $\partial H$. This implies that it is possible to reduce $|F\cap(D_1\sqcup D_2)|$ by an isotopy that replaces a component of $F\setminus c$ with the disk $D_0$, again contradicting our hypothesis. 

Therefore $F$ may be assumed to be disjoint from $D_1\sqcup D_2$ and hence isotopic to a disk obtained by performing some number of half-Dehn twists to the disk $E$ along $D$. Since $1\leq n\leq m$ holds by (vii) and hence $n\leq 2m-n$ it is then not hard to see that $F$ must intersect $J$ in at least $2n$ points.
\end{proof}

The next result classifies the basic pair $(H,J)$ in terms of the numbers $m$ and $n$ in (vi)-(vii).

\begin{lem}\label{baspa2}
The basic pair $(H,J)$ is trivial for $m=1$, simple of index $m\geq 2$ if $n=1$, and otherwise hyperbolic.
\end{lem}

\begin{proof}
Let $\partial H=T_1\cup_J T_2$. For $m=1$ it is not hard to see that $H\approx T_1\times I$ and hence that $(H,J)$ is a trivial pair. 

If $n=1$ and $m\geq 2$ then it is not hard to find a circle in, say, $T_1$, which represents the power $(D_1D_2)^m$ in $\pi_1(H)=\grp{D_1,D_2 \ | \ -}$. Since by Lemma~\ref{prim-2}(1) the pair $(H,J)$ is minimal, by \S\ref{induced} it must be simple.

Conversely, if $(H,J)$ is a simple pair then by \S\ref{simple} and \cite[Lemma 6.2(5)]{valdez14} there  is a disk $F\subset H$ that intersects $J$ minimally in two points and satisfies (E2). Since by Lemma~\ref{FF} the disk $F$ must intersect $J$ in at least $2n$ points it follows that $n=1$.

Finally, if $n\geq 2$ then by the above argument the pair $(H,J)$ is neither primitive nor simple and hence must be hyperbolic by Lemma~\ref{prim-2}(3).
\end{proof}

\begin{rem}
The simplest example of a basic hyperbolic pair $(H,J)$ is constructed in Section~\ref{ht4} and represented in Fig.~\ref{k04}, top. In Proposition~\ref{Khyper} this hyperbolic pair is used to construct an example of a genus one hyperbolic knot $K\subset\mS^3$ with a simplicial collection $\mT\subset X_K$ which decomposes $X_K$ into two simple pairs and two hyperbolic pairs homeomorphic to $(H,J)$.
\end{rem}

\section{Annular pairs} \label{ann3}
Here we generalize the notions of simple and primitive handlebody pairs to arbitrary pairs.

\subsection{Spanning annuli}

Let $(H,J)$ be a pair with $\partial H=T_1\cup_J T_2$.
A {\it spanning annulus} for $(H,J)$ is an annulus $A\subset H$ with $\partial_1A\subset T_1$ and $\partial_2A\subset T_2$ nonseparating circles, in which case we say that the circles $\partial A\subset\partial H$ are {\it coannular} in $H$.

\medskip

Any spanning annulus $A$ for a pair $(H,J)$ is nonseparating and incompressible. If $H$ is a genus two handlebody then by \cite[Lemma 3.4]{valdez14} the boundary circles $\partial_1A=A\cap T_1$ and $\partial_2A=A\cap T_2$ are both primitive or both $p$-power circles in $H$ for some $p\geq 2$ and cobound at most two nonisotopic spanning annuli in $H$. The next result generalizes these facts to arbitrary pairs.

\begin{lem}\label{annp}
Let $(H,J)$ be a pair with $\partial H=T_1\cup_{J}T_2$.
\ben
\item
Any two spanning annuli $A_1,A_2$ for the pair $(H,J)$ which intersect transversely with $\partial A_1\cap\partial A_2\cap T_i=\emptyset$ for some $i=1,2$ can be isotoped so as to be mutually disjoint.

\item
If $(H,J)$ contains two spanning annuli with nonempty minimal intersection then $(H,J)$ is a trivial pair.

\item
A nontrivial pair $(H,J)$ admits at most two isotopy classes of spanning annuli. Specifically, any two nonisotopic spanning annuli $A_1,A_2$ for the pair $(H,J)$ that intersect minimally are mutually disjoint and cobound a solid torus region $V\subset H$ such that $A_1$ and $A_2$ each run $p\geq 2$ times around $V$. 
\een
In particular, for any nontrivial pair $(H,J)$ with a spanning annulus,
\begin{enumerate}
\setcounter{enumi}{3}
\item
the boundary slopes $\omega_1\subset T_1$ and $\omega_2\subset T_2$ of spanning annuli in $H$ are unique up to isotopy,

\item
$(H,J)$ admits two nonisotopic spanning annuli iff
$\omega_1\subset T_1$ or $\omega_2\subset T_2$ has a companion annulus in $H$, in which case
\begin{enumerate}
\item
in $H$ each slope $\omega_1\subset T_1$ and $\omega_2\subset T_2$ has a companion annulus and a companion solid torus around which the slope runs $p\geq 2$ times,

\item
if the pair $(H,J)$ is minimal then it is simple.
\end{enumerate}
\end{enumerate}
\end{lem}

\bpr
For part (1), let $A_1,A_2\subset H$ be any two spanning annuli for the pair $(H,J)$ which intersect transversely. We assume that $\partial_iA_j\subset T_i$ for all $i,j\in\{1,2\}$, and that $\partial_2A_1\cap\partial_2A_2=\emptyset$ for definiteness. It follows that any arc component of $A_1\cap A_2$ is parallel to $\partial_1A_1$ and $\partial_1A_2$ in $A_1$ and $A_2$, respectively, and hence, by a standard outermost arc/innermost circle argument (using the fact that $H$ is irreducible and $T_1$ is incompressible in $H$), that $A_1$ and $A_2$ can be isotoped to intersect minimally so that each component of $A_1\cap A_2$, if any, is a nontrivial circle in $A_1$ and $A_2$.

\begin{figure}
\Figw{3.2in}{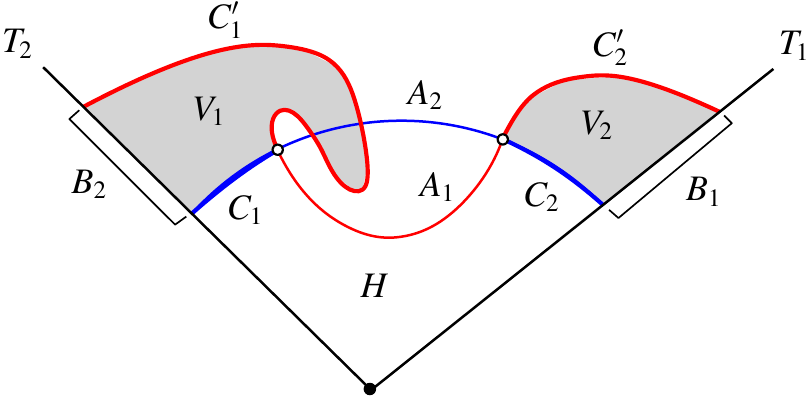}{The spanning annuli $A_1,A_2$ in $H$.}{h04}
\end{figure}

Now, each pair of circles $\partial_1A_1\sqcup\partial_1A_2\subset T_1$ and $\partial_2A_1\sqcup\partial_2A_2\subset T_2$ cobound annuli $B_1\subset T_1$ and $B_2\subset T_2$, respectively. Let $C_1,C_2\subset A_2$ be the annular closures of the components of $A_2\setminus A_1$ that contain the circles $\partial_1A_2,\partial_2A_2$, respectively, and let $C'_1,C'_2\subset A_1$ be the annuli cobounded by the pairs of circles $(C_1\cap A_1)\sqcup\partial_1A_1$ and $(C_2\cap A_1)\sqcup\partial_2A_1$, respectively. By the minimality of $A_1\cap A_2$, for $i=1,2$ the annulus $C_i\cup C'_i$ is a companion annulus in $H$ for the core circle of $B_i$, and hence $B_i$ and $C_i\cup C'_i$ cobound a solid torus $V_i\subset H$, with $B_i$ running at least twice around $V_i$. 
The situation is represented in Fig.~\ref{h04}.

It follows that the manifold $M=N(A_1\cup V_1\cup V_2)\subset H$ is a Seifert fiber space over the disk with two singular fibers, contradicting Lemma~\ref{compa2} applied to the torus $\partial M\subset H$. Therefore we must have $A_1\cap A_2=\emptyset$ and so (1) holds.

\medskip

For (2), let $A_1,A_2\subset H$ be any two spanning annuli for the pair $(H,J)$ which intersect minimally with $A_1\cap A_2\neq\emptyset$. By (1),
for $i=1,2$ the circles $\partial_iA_1,\partial_iA_2\subset T_i$ intersect minimally and, as $T_i$ is a once-punctured torus, coherently in $T_i$, and so $A_1\cap A_2$ consists of a nonempty disjoint collection of mutually parallel spanning arcs in $A_1$ and $A_2$. 
If $A_1\cap A_2$ has at least two arc components then,
from the closure of a rectangular component of $A_2\setminus A_1$, it is possible to construct a spanning annulus $A'_2$ for $(H,J)$ which intersects $A_1$ minimally in one arc. We may therefore assume that $A_1\cap A_2$ is a single spanning arc, in which case $S_1=N(\partial_1A_1\cup\partial_1A_2)\subset T_1$ is a once-punctured torus with $\partial S_1$ parallel in $T_1$ to $J=\partial T_1$, while $N(A_1\cup A_2)\subset H$ is homeomorphic to the genus two handlebody $S_1\times I$, with $A_1\cup A_2$ corresponding to $(\partial_1A_1\cup\partial_1A_2)\times\{1/2\}$.

It follows that the frontier $A$ of $N(A_1\cup A_2)\subset H$ is an annulus with boundary circles $\partial A$ parallel to $J$ in $H$. Since the circle $J$ has no companion annuli in $H$, the annulus $A$ must be parallel to $\partial H$ in $H$, which implies that $H$ is homeomorphic to $S_1\times I$ and hence that the pair $(H,J)$ is trivial, so (2) holds.

\medskip
Suppose now that $(H,J)$ is a nontrivial pair. By (1) any two spanning annuli $A_1$ and $A_2$ for $(H,J)$ can be isotoped so as to be disjoint, whence the circles $\partial_1A_1,\partial_1A_2\subset T_1$ and  $\partial_2A_1,\partial_2A_2\subset T_2$ cobound annuli $B_1\subset T_1,B_2\subset T_2$, respectively. By Lemma~\ref{compa2}, the torus $A_1\cup A_2\cup B_1\cup B_2$ bounds a solid torus $V(A_1,A_2)\subset H$, with each annulus $A_1,A_2\subset\partial V(A_1,A_2)$ running $n(A_1,A_2)\geq 1$ times around $V(A_1,A_2)$. 

If $n(A_1,A_2)=1$ for all spanning annuli $A_1,A_2$ of $(H,J)$ then any two such spanning annuli are mutually isotopic, so $(H,J)$ contains a unique spanning annulus. 

Otherwise, suppose that $p=n(A_1,A_2)\geq 2$ for some mutually disjoint spanning annuli $A_1,A_2\subset H$, and let $A\subset H$ be any spanning annulus for $(H,J)$. Isotope $A$ so as to be disjoint from $A_1$, whence $A$ and $A_1$ have the same boundary slope, and then isotope $A$ so as to intersect $A_2$ minimally subject to $A\cap A_1=\emptyset$. An argument similar to the one used in the proof of part (1) then shows that we must have $A\cap A_2=\emptyset$ too, whence $A$ can be isotoped so as to be disjoint from $A_1\sqcup A_2$. It follows that either $A\subset V(A_1,A_2)$ or $A_i\subset V(A,A_j)$ for some $\{i,j\}=\{1,2\}$, which implies that $A$ is parallel to $A_1$ or $A_2$. 

Therefore $A_1$ and $A_2$ are up to isotopy the only spanning annuli in $H$, hence their boundary slopes are the only slopes in $T_1$ and $T_2$ that cobound a spanning annulus in $H$; and as $p\geq 2$ the solid torus $V(A_1,A_2)$ is a companion solid torus for each slope $\omega_1=\partial_1A_1\subset T_1$ and $\omega_2=\partial_2A_1\subset T_2$.

Conversely, if $A$ is a spanning annulus for $(H,J)$ and $V\subset H$ is a companion solid torus for, say, the circle $\omega_1=A\cap T_1\subset T_1$, so that $\omega_1$ runs $p\geq 2$ times around $V$, then $A$ can be isotoped so that $A\cap\intr\,V=\emptyset$ and $A\cap V=A\cap T_1$, in which case $N(A\cup V)\subset H$ is a solid torus whose frontier consists of two disjoint spanning annuli for $(H,J)$, each of which runs $p$ times around $N(A\cup V)$. Therefore the first part of (5) holds.

Finally, let $V_1$ be the solid torus obtained by pushing $V(A_1,A_2)$ slightly off $T_2$. Then $V_1$ is a companion solid torus for $\omega_1=\partial_1A_1\subset T_1$ which by \S\ref{induced} induces a $J$-torus $T\subset H$ that splits $H$ into genus two handlebodies $H_1=N(T_1\cup V_1)$ and $H_2\subset H$, with $\partial H_1=T_1\cup_JT$ and $\partial H_2=T_2\cup_JT$, such that $(H_1,J)$ is a simple pair of index $p\geq 2$. Thus if the pair $(H,J)$ is minimal then $T$ is parallel to $T_2$ and so the pair $(H,J)$ is simple.
A similar conclusion holds if we push the solid torus $V(A_1,A_2)$ slightly off $T_1$ to obtain a companion solid torus for $\omega_2=\partial_2 A_1\subset T_2$.
Therefore (4) and (5) hold.
\epr

\subsection{The index of an annular pair.}\label{index}

We make the following definitions and observations based on the properties obtained in Lemma~\ref{annp}.
\begin{itemize}
\item[(A1)]
A pair $(H,J)$ is said to be {\it annular} if it is nontrivial and contains a spanning annulus.

\item[(A2)] 
If $(H,J)$ is an annular pair and $A\subset H$ is a spanning annulus then the {\it index} of $(H,J)$ and $A$ is the number $p=n(A_1,A_2)\geq 2$ given in Lemma~\ref{annp}(3) if $A$ is not unique, and otherwise it is 1. 

\item[(A3)] 
By \cite[Lemma 3.4(4)(a)]{valdez14}, a handlebody pair $(H,J)$ is annular of index 1 iff it is a primitive pair.

\item[(A4)]
For an annular pair $(H,J)$ with a spanning annulus $A$ of index $p\geq 2$, the solid torus $V=V(A_1,A_2)\subset H$ in Lemma~\ref{annp}(3) is unique up to isotopy and its core is called the {\it core knot of $(H,J)$}.  

\begin{figure}
\Figw{2.2in}{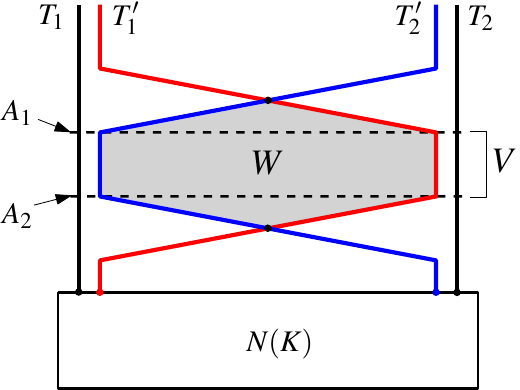}{The $J$-tori $T'_1,T'_2\subset R_{1,2}$ induced along $T_1$ and $T_2$ in the annular pair $(R_{1,2},J)$.}{h11}
\end{figure}

Also, following the notation in the proof of Lemma~\ref{annp}(3), for each $\{i,j\}=\{1,2\}$ the manifold $W_i=N(T_i\cup V)\subset H$ is a handlebody which, after being pushed slightly off from $T_j$, produces a simple pair $(W_i,J)$ of index $p\geq 2$ cobounded by its two frontier $J$-tori $T_i$ and $T'_i\subset H$. As in \S\ref{induced} the pair $(W_i,J)$ and $T'_i$ are the simple pair and $J$-torus {\it induced by the annular pair $(H,J)$ along $T_i$}, unique up to isotopy in $H$. The situation is represented in Fig.~\ref{h11}. 
\end{itemize}

\section{Seifert tori in $X_K$.}\label{stk}
In this section we establish several properties of Seifert tori in the exterior of a hyperbolic knot $K\subset\mS^3$. In particular we determine the structure of the pairs generated by a simplicial collection of Seifert tori in $X_K$ in the presence of a Seifert torus not isotopic to any of those in the collection.

\subsection{General properties}
We use the following notation.
Let $T\subset X_K$ be a $J$-torus and $T\times[-1,1]$ a product neighborhood of $T$ in $X_K$ with $T$ corresponding to $T\times\{0\}$. For a surface $F\subset X_K$, not necessarily properly embedded, such that $T\cap \intr\,F=\emptyset$ and $T\cap\partial F\neq\emptyset$, we say that {\it $F$ locally lies on one side of  $T$} if $F\cap (T\times[-1,0])=\emptyset$ or $F\cap (T\times[0,1])=\emptyset$, and otherwise that {\it $F$ locally lies on both sides of $T$.} For instance, a companion annulus for a slope in $T$ locally lies on one side of $T$, while $\partial X_K$ locally lies on both sides of $T$.

\begin{lem}\label{ann-powers}
Let $T_1,T_2,T_3\subset X_K$ be $J$-tori in the exterior of a hyperbolic knot $K\subset\mS^3$.
\ben
\item
If $F\subset X_K$ is a properly embedded surface which intersects $T_1$ transversely with $(\partial T_1)\cap(\partial F)=\emptyset
$ then the number of circle components of $T_1\cap F$ that are nonseparating in $T_1$ is even. 

\item
If $A$ is an annulus in $X_K$ with $A\cap T_1=(\partial A)\cap T_1$ such that each of the circles $\partial A$ is nonseparating in $T_1$ and $\Delta(\partial_1A,\partial_2A)=0$ then $A$ is a companion annulus that locally lies on one side of $T_1$.

\item
Any two companion annuli for a circle in $T_1$ locally lie on the same side of $T_1$ and are mutually isotopic.

\item
If $T_1,T_2$ are mutually disjoint and $A$ is a spanning annulus for the pair $(R_{1,2},J)$ then the circles $\partial A$ are not coannular in $R_{2,1}$ and not both have companion annuli in $R_{2,1}$.
Moreover, if a component of $\partial A$ has a companion annulus in $R_{2,1}$ then $A\subset R_{1,2}$ has index 1.

\item
If $B\subset R_{1,2}$ is an annulus with $\partial_1 B$ a nonseparating circle in $T_1$ and $\partial_2B$ a circle in $T_i$ for some $i=1,2$  then $\partial_2B$ is also a nonseparating circle in $T_i$.

\item
Suppose that  $T_1,T_2,T_3\subset X_K$ are mutually disjoint and nonparallel Seifert tori with $T_2\subset R_{1,3}$. 
If $A\subset R_{1,3}$ is a spanning annulus which intersects $T_2$ minimally then $A_1=A\cap R_{1,2}$ and $A_2=A\cap R_{2,3}$ are spanning annuli in $R_{1,2}$ and $R_{2,3}$, respectively, and $(R_{1,3},J)$ has index $p\geq 1$ iff one of the pairs $(R_{1,2},J)$, $(R_{2,3},J)$ has index  1 and the other index $p\geq 1$.

\item
If the $J$-tori $T_1,T_2,T_3$ are mutually disjoint and nonparallel in $X_K$, $T_2$ lies in the region $R_{1,3}$, and the pair $(R_{1,2},J)$ is simple with $\omega\subset T_2$ a $p\geq 2$ power circle in $R_{1,2}$, then $R_{1,3}$ is a handlebody iff $R_{2,3}$ is a handlebody and $\omega$ is a primitive circle in $R_{2,3}$.

In particular, if the pair $(R_{2,3},J)$ is primitive then $R_{1,3}$ is a handlebody iff the slopes of the spanning annuli in $R_{1,2}$ and $R_{2,3}$ agree on $T_2$.
\een
\end{lem}

\bpr
For part (1), observe that if $\partial F\neq\emptyset$ then, after suitably capping off with disks any boundary components of $\partial F$ that are trivial in $X_K$, we may assume that each component of $\partial F$ is a nontrivial circle in $X_K$ of slope $J$.

$T_1\cap F$ has no arc components since $(\partial T_1)\cap(\partial F)=\emptyset$ and so each component of $T_1\cap F$ is a circle which either is parallel to $\partial T_1$, bounds a disk in $T_1$, or does not separate $T_1$. Let $\mc{N}\subseteq T_1\cap F$ be the collection of circles that are nonseparating in $T_1$ and assume that $\mc{N}\neq\emptyset$. As the circles in $\mc{N}$ are mutually parallel in $T$, there is a circle $\beta\subset T_1$ which intersects each component of $\mc{N}$ transversely in one point and is disjoint from $T_1\cap F\setminus \mc{N}$. After pushing $\beta$ slightly away from $T_1$, we may assume that $\beta$ is disjoint from $T_1$ and intersects $F$ transversely in $|\mc{N}|$ points. If $F$ separates $X_K$ then $|\mc{N}|$ is even, so we may further assume that $F$ does not separate $X_K$.
Since $H_2(X_K(J);\mathbb{Z}_2)=\mathbb{Z}_2$, the nonseparating closed surfaces $\wh{T}_1,\wh{F}\subset X_K(J)$ belong to the only nontrivial homology class of $H_2(X_K(J);\mathbb{Z}_2)$, hence $\wh{T}_1$ and $\wh{F}$ must have the same homological intersection number mod 2 with $\beta$, ie 
\[
0\equiv \beta\cdot\wh{T}_1 \equiv \beta\cdot\wh{F} \equiv| \mc{N}|\mod 2
\]
and so $|\mc{N}|$ is even.

\medskip

For part (2) suppose that $A$ does not locally lie on one side of $T_1$. Since $\Delta(\partial_1A,\partial_2A)=0$, the circles $\partial A$ have the same slope in $T_1$ and hence $A$ can be isotoped in $X_K$ so that $A\cap T_1=(\partial A)\cap T_1$ and $\partial_1A=\partial_2A$, in which case the resulting closed surface $A$ contradicts the conclusion of part (1). Therefore $A$ locally lies on one side of $T_1$.

\medskip
Part (3) is the content of \cite[Lemmas 3.1 and 5.1]{valdez14}.

\medskip
For part (4), if $B$ is an annulus in $R_{2,1}$ with $\partial B=\partial A$ then $A\cup B\subset X_K$ is a closed surface in $X_K$ that intersects $T_1$ minimally in one circle, contradicting (1); thus the circles $\partial A$ are not coannular in $R_{2,1}$.

Suppose now  that $B_1,B_2\subset R_{2,1}$ are companion annuli for the circles $\partial_1A,\partial_2A$, respectively, isotoped so as to intersect minimally, and let $V_1,V_2\subset R_{2,1}$ be the companion solid tori cobounded by $B_1,T_1$ and $B_2,T_2$, respectively. Let $A_1,A_2$ be mutually disjoint spanning annuli for the pair $(R_{1,2},J)$ that are parallel to $A$ with $\partial A_1\sqcup\partial A_2=\partial B_1\sqcup\partial B_2$. 

If $B_1\cap B_2\neq\emptyset$ then each component of $B_1\cap B_2$ is a core circle of $B_1$ and $B_2$, so it is possible to construct a spanning annulus $B$ for the pair $(R_{2,1},J)$ with $\partial B=\partial A$, contradicting the argument above. And if $B_1\cap B_2=\emptyset$ then $B_1$ and $A_1\cup B_2\cup A_2$ are companion annuli for the circle $T_1\cap \partial A$ that lie on opposite sides of $T_1$, contradicting (3). 

Finally, if $A$ has index $p\geq 2$ then each circle $\partial_iA\subset T_i$ has a companion annulus in $R_{1,2}$ by
Lemma~\ref{annp}(5) and hence by (3) cannot have a companion annulus in $R_{2,1}$. Therefore (4) holds.

\medskip

For part (5) if in $\wh{T}_i\subset X_K(J)$ the circle $\partial_2B\subset T_i$ bounds a disk then the disk $\wh{B}_1$ compresses the nonseparating torus $\wh{T}_1$ in $X_K(J)$ into a nonseparating 2-sphere, contradicting \cite[Corollary 8.3]{gabai03} that the manifold $X_K(J)$ is irreducible. Therefore $\partial_2B$ is nonseparating in $\wh{T}_i$ and hence in $T_i$.

\medskip
For part (6), each component of $A\cap T_2$ is a nontrivial circle in $A$ and in $T_2$ and so each component of $A\cap R_{1,2}$ and $A\cap R_{2,3}$ is an annulus. 

Let $A_1$ be the component of $A\cap R_{1,2}$ with $A\cap T_1\subseteq A_1\cap T_1$. Then necessarily $A_1\cap T_2\neq\emptyset$ and so by (5) the circle $\alpha_1=A_1\cap T_2$ is nonseparating in $T_2$, hence $A_1$ is a spanning annulus in $R_{1,2}$. Similarly the component $A_2$ of $A\cap R_{2,3}$ with $A\cap T_3\subseteq A_2\cap T_3$ is a spanning annulus in $R_{2,3}$ with $\alpha_2=A_2\cap T_2$ a nonseparating circle in $T_2$.
In particular either $\alpha_1=\alpha_2$ or $\alpha_1,\alpha_2$ are disjoint and mutually parallel circles in $T_2$. 
 
If $\alpha_1\neq\alpha_2$ then by (5) the component of $A\cap R_{2,3}$ which contains $\alpha_1$ is a companion annulus for $\alpha_1$ in $R_{2,3}$, and similarly 
$\alpha_1$ has a companion annulus in $R_{1,2}$, contradicting (3). Therefore $\alpha_1=\alpha_2$ and so $A=A_1\cup A_2$. 

Suppose now that $B\subset R_{1,3}$ is a spanning annulus disjoint from $A$. By the argument above we may assume that $B_1=B\cap R_{1,2}\subset R_{1,2}$ and $B_2=B\cap R_{2,3}\subset R_{2,3}$ are spanning annuli. Let $V$ be the  region in $R_{1,3}$ cobounded by $A$ and $B$, and let $C\subset V$ be the annulus cobounded by the circles $(A\sqcup B)\cap T_2$. By Lemma~\ref{annp}(3) the region $V$ is a solid torus and so $C$ separates $V$ into two solid tori $V_1=V\cap R_{1,2}$ and $V_2=V\cap R_{2,3}$. Necessarily $C$ runs once around one of the solid tori $V_1$ or $V_2$.

If the pair $(R_{1,3},J)$ has index $p\geq 2$ and $A$ runs $p$ times around $V$ then necessarily $A_1$, say, runs $p$ times around $V_1$. Therefore the pair $(R_{1,2},J)$ has index $p\geq 2$, while by 
Lemma~\ref{annp}(5) the core of $C\subset T_2$ has a companion annulus in $R_{1,2}$ and  so by (3) the core of $C$ cannot have a companion annulus in $R_{2,3}$, which implies that the pair $(R_{2,3},J)$ has index 1.

Conversely, if $(R_{1,2},J)$, say, has index $p\geq 2$ then there is a spanning annulus and $A'_1\subset R_{1,2}$ disjoint from $A_1$ that cobounds with $A_1$ a solid torus $V'\subset R_{1,2}$ around which each annulus runs $p$ times. Thus $W=N(A_1\cup V')$ is a solid torus in $R_{1,3}$ and its frontier consists of two spanning annuli in $R_{1,3}$ each of which runs $p$ times around $W$, so the pair $(R_{1,3},J)$ has index $p$. Therefore (6) holds.

\medskip
The first part of (7) follows from \S\ref{oper}(2),(3).
If the pair $(R_{2,3},J)$ is primitive then by \cite[Lemma 6.9(2)]{valdez14} 
the slope $\omega'\subset T_2$ of the spanning annulus in $R_{2,3}$ is the unique circle which is primitive in $R_{2,3}$, while $\omega\subset T_2$ is the slope of the spanning annulus in $R_{1,2}$. Therefore $R_{1,3}$ is a handlebody iff $\omega$ and $\omega'$ have the same slope in $T_2$.
\epr

\subsection{Intersections of Seifert tori}

By \cite[Lemma 5.2]{sakuma3} the minimal intersection between two nonisotopic Seifert tori in $X_K\subset\mS^3$ (which is assumed to be only atoroidal) consists of two circles which are nonseparating in each of the surfaces. 
Here we extend this result to give a more detailed picture of the minimal intersection between a Seifert torus $S$ and a simplicial collection of Seifert tori $\mT\subset X_K$. In particular we will see that a nontrivial such intersection produces an annular pair of index 1 within a complementary region of $\mT$.

The next result is the first approximation to the main classification given in Proposition~\ref{mainp}.

\begin{lem}\label{new01c}
Let $\mT=T_1\sqcup\cdots\sqcup T_N$ be a simplicial collection of Seifert tori in $X_K$ and let $S\subset X_K$ be a Seifert torus which is not isotopic in $X_K$ to any component of $\mT$, such that either 
\begin{enumerate}
\item[(i)]
$S$ intersects the collection $\mT$ minimally,

\item[(ii)]
$N\geq 2$, $S\subset R_{i,j}$ ($i=j$ allowed), and $S$ intersects the collection $\mT\cap R_{i,j}$ minimally.
\end{enumerate} 
Then 
\begin{enumerate}
\item
for each $j$, $S\cap T_j$ is either empty or consists of two circle components that are nonseparating in $S$ and $T_j$; in particular $S$ and $T_j$ intersect minimally in $X_K$,

\item
the closure of each component of $S\setminus\mT$ is either a pants $P$ or an annulus,

\item
$P\cap\mT=P\cap T_i$ for some $1\leq i\leq N$,

\item
there is a Seifert torus $T\subset X_K\setminus (P\cup\mT)$ such that if $R,R'\subset X_K$ are the regions cobounded by $T\sqcup T_i$  and $P\subset R$ then 
\begin{enumerate}
\item
$R\cap\mT=T_i$,

\item
the pair $(R,J)$ is annular of index 1 with spanning annulus $A_R\subset R$,

\item
the pair $(R',J)$ is nontrivial and $A=S\cap R'$ is a companion annulus; moreover $P$ and $A$ lie on opposite sides of $T_i$ and the annuli $A,A_R$ have the same boundary slope in $T_i$.
\end{enumerate}
\end{enumerate}
\end{lem}

\bpr
(I) Since $S$ and $\mT$ have the same boundary slopes, by conditions (i),(ii) we have that $\partial S\cap\partial\mT=\emptyset$ and so $S\cap\mT$ is a nonempty collection of circles that are nontrivial in $S$ and $\mT$.

\medskip
(II) {\it For each $j$, $S\cap T_j$ consists of a collection circles that are nonseparating in $S$ and $T_j$, hence mutually disjoint and parallel in $S$ and $T_j$. Thus (1) holds.} 
\\  
For let $c$ be a circle component of $S\cap \mT$; by (I) $c$ is nontrivial in $S$ and $\mT$. Consider the case where $c$ is parallel to $\partial S$ in $S$; the case where $c$ is parallel to $\partial\mT$ can be dealt with in a similar way. We may assume that $c$ is outermost in $S$, that is, $c$ cobounds with $\partial S$ an annulus $A_c\subset S$ with interior disjoint from $\mT$.

Let $T_j\subset\mT$ be the component containing $c$.
By Lemma~\ref{ann-powers}(5) the circle $c$ separates $T_j$ and so it cobounds an annulus $A'_c\subset T_j$ with $\partial T_j$. The annulus $A_c\cup_c A'_c$ is then properly embedded in $X_K$ with $J$ as boundary slope, and as the knot $K$ is hyperbolic this annulus must be parallel in $X_K$ into an annulus $B\subset\partial X_K$. 
Thus the annuli $A_c\cup_c A'_c$ and $B$ cobound a solid torus $V\subset X_K$ around which each annulus runs once and such that $V\cap\mT=A'_c$.
It is then possible to reduce $|S\cap\mT|$ by an isotopy of $\mT$ that exchanges the annulus $A'_c$ with $A_c$ within the solid torus $V$ and pushes the resulting surface slightly off $S$, contradicting the minimality of $|S\cap\mT|$. 

\medskip

(III) {\it  $S\cap T_j$ has an even number of components:}
by Lemma~\ref{ann-powers}(1).

\medskip

(IV) 
{\it If $S\cap T_j\neq\emptyset$ then $|S\cap T_j|=2$ and the closures of the components of $S\setminus T_j$ are a pants $P_j$ and a companion annulus $A_j$ that locally lie on opposite sides of $T_j$, with $P_j\cap T_j=A_j\cap T_j=\partial A_j$.}\\
By (II)-(III) the closures of the components of $S\setminus T_j$ consist of a pants component $P_j$ and an odd number of annuli. By Lemma~\ref{ann-powers}(3), each such annulus component is a companion annulus for the slope of the circles $S\cap T_j\subset S$, and all such annular components lie on the same side of $T_j$. Therefore there can be only one such annular component $A_j$, so $|S\cap T_j|=2$ and the rest of the properties of $P_j$ and $A_j$ follow.

\medskip

(V) Similarly, by (II)-(III) $S\cap\mT$ consists of an even number of circle components which are nonseparating in $S$ and so the closures of the components of $S\setminus\mT$ consist of a pants  component $P$ and an odd number of annuli. If $\partial P=\partial S\sqcup\alpha_1\sqcup\alpha_2$ then by (IV) $P\cap T_i=S\cap T_i=\alpha_1\sqcup\alpha_2$ for some $T_i\subset\mT$ and the annulus $A=\cl[S\setminus P]$ is a companion annulus of the slope of the circles $S\cap T_i$, with $P$ and $A$ lying on opposite sides of $T_i$ and $P\cap\mT=P\cap T_j$. Thus (2) and (3) hold.

\medskip

(VI) Let $P,A,T_i$ be as in (V) so that $S=P\cup A$. Since $P$ and $\mT\setminus T_i$ are disjoint, there is  regular neighborhood $N(P\cup T_i)\subset X_K$ which is disjoint from $\mT\setminus T_i$. The frontier of $N(P\cup T_i)\subset X_K$ contains two $J$-tori $T_P,T_P'$, with $T_P$ on the same side of $T_i$ as $P$ and $T_P'$ on the opposite side and parallel to $T_i$.

Let $R,R'\subset X_K$ be the two regions cobounded by $T_P$ and $T_i$, with $P=S\cap R$ and $A=S\cap R'$. If $T_P$ and $T_i$ are parallel in $R$ or $R'$ then by \cite[Corollary 3.2]{wald1} $P$ or $A$ is parallel in $R$ or $R'$ into $T_i$, respectively, and so $S$ can be isotoped by pushing $P$ or $A$ across and to the other side of $T_i$, thus reducing $|S\cap\mT|$, which is not possible. 
Therefore $T_P$ and $T_i$ are not parallel in $X_K$ and so the pairs $(R,J)$ and $(R',J)$ are nontrivial.

Let $B\subset T_i$ be the annulus cobounded by the circles $\alpha_1\sqcup\alpha_2=P\cap T_i$. Then the $J$-torus $P\cup B\subset R$ is parallel in $R$ to $T_P$, that is, the region in $R$ between $T_P$ and $P\cup B$ is a product of the form $(P\cup B)\times[0,1]$, with $T_P=(P\cup B)\times\{0\}$ and $P\cup B=(P\cup B)\times\{1\}$. It follows that $A_R=\alpha_1\times[0,1]$ is a spanning annulus for the region $R$. Since the annulus $A=S\cap R'$ is a companion for the slope $\alpha_1\subset T_i$ outside of $R$, $A_R$ has index 1 by Lemma~\ref{ann-powers}(4) and so the pair $(R,J)$ is annular of index 1. Therefore (4) holds.
\epr

\section{Minimality of index 1 annular pairs in $X_K$.}\label{min5}

Suitably gluing together two annular pairs of index 1 results in a new annular pair of index 1 which is not minimal. In this section we show that an annular pair of index 1 produced by two Seifert tori in the exterior of a hyperbolic knot in $\mS^3$ must be minimal. 

\subsection{Annular pairs in $X_K$.}

Let $K\subset\mS^3$ be a genus one hyperbolic knot and let  $(R,J)$, $(R',J)$ be pairs cobounded by two mutually disjoint and nonparallel Seifert tori in $X_K$, so that $X_K=R\cup R'$.

If $(R',J)$ is an annular pair and the region $R'\subset X_K$ is boundary irreducible then by Lemma~\ref{genprop} (P2) the region $R\subset\mS^3$ is a genus two handlebody, an example of a {\it nontrivial handlebody knot in $\mS^3$.} In \cite[Lemma 3.8]{ozawa01} Y.\ Koda and M.\ Ozawa classify $R$ as a a certain type of handlebody knot using a result of C.\ Gordon in \cite[Lemma 3.6]{ozawa01}, along with  that any 4-punctured sphere with integral boundary slope in a knot exterior in $\mS^3$ is compressible. We remark that the compressibility of many-punctured spheres with nonintegral and nonmeridional boundary slope follows from the results in \cite[\S2.5--2.6]{cgls}, in particular Proposition 2.5.6.

We use a similar strategy to impose restrictions on the pairs $(R,J)$ or $(R',J)$ in $X_K$ whenever one of them is an annular pair. A classification of handlebody annular pairs is obtained which will be extended and refined in Proposition~\ref{annmin} to arbitrary annular pairs in a knot exterior $X_K$. We will see in Section~\ref{maxcoll} that the properties of this type of pair are the key to bound the number of maximal simplicial collections of Seifert tori in $X_K$.

\begin{lem}\label{koz1-b}
Let $K\subset\mS^3$ be a genus one hyperbolic knot and  $T_1\sqcup T_2\subset X_K$ a simplicial collection of Seifert tori such that the pair $(R_{1,2},J)$ is annular with spanning annulus $A\subset R_{1,2}$. Then one of the following holds:
\ben
\item
the region $R_{1,2}$ is a genus two handlebody and the pair $(R_{1,2},J)$ is either primitive, simple, or splits along some $J$-torus in $R_{1,2}$ into a simple and a primitive pair,

\item
the region $R_{2,1}$ is a genus two handlebody, the circles $\partial_1 A\subset T_1$ and $\partial_2A\subset T_2$ are separated in $R_{2,1}$
and one of the following holds: 
\ben
\item
the pair $(R_{2,1},J)$ is basic, with the components of $\partial A$ as basic circles in $R_{2,1}$, 

\item
one of the two components of $\partial A$, say $\partial_2A\subset T_2$, is a power circle in $R_{2,1}$ which induces a $J$-torus $T_3\subset R_{2,1}$,  such that the pair $(R_{2,3},J)$ is simple with spanning annulus $B$ of index $p\geq 2$
and the pair $(R_{3,1},J)$ is basic with $A\cap T_1$ and $B\cap T$ basic circles in $R_{3,1}$; in particular,
any $J$-torus in $R_{2,1}$ is isotopic in $R_{2,1}$ to $T_1,T_2$ or $T_3$.
\een
\een
\end{lem}

\bpr
Let the knot $L\subset\mS^3$ be the core of $A$. Then $N(A)\subset R_{1,2}$ is a solid torus neighborhood of $L$ and so $X_L$ can be identified with $\mS^3\setminus \intr N(A)$. Extending $T_1$ and $T_2$ radially in $N(K)$ so that $\partial T_1=K=\partial T_2$ yields the genus two surface $F=T_1\cup T_2$ such that $P=\cl[F\setminus N(A)]\subset F$ is a 4-punctured 2-sphere in $X_L$. 
Notice that $P$ has integral boundary slope in $\partial X_L$ and separates $X_L$ into two components $W_1=\cl[R_{1,2}\setminus N(A)]\subset R_{1,2}$ and $W_2=R_{2,1}\cup N(A)\supset R_{2,1}$.

By \cite[Lemma 3.6]{ozawa01} the 4-punctured sphere $P$ compresses in $X_L$ along some disk $E$. We consider two cases.

\medskip
{\bf Case 1:} $E\subset W_1$. 

Then $E\subset R_{1,2}$, so $\partial R_{1,2}$ compresses in $R_{1,2}$ and so the region $R_{1,2}$ is a genus two handlebody by Lemma~\ref{genprop} (P1); hence the pair $(R_{1,2},J)$ is primitive if $A$ has index 1.

Suppose now that $A$ has index $p\geq 2$, so that $\partial_1A=A\cap T_1$ is a $p$-power circle in $R_{1,2}$. 
If the pair $(R_{1,2},J)$ is minimal then it is simple by Lemma~\ref{annp}(5).
If the pair $(R_{1,2},J)$ is not minimal and $T_a\subset R_{1,2}$ is some $J$-torus not parallel to $T_1,T_2$ then by 
\cite[Lemma 3.7(2)(3)]{valdez14}
each region $R_{1,a},R_{a,2}\subset R_{1,2}$ is a handlebody and we may assume that $(R_{1,a},J)$ is simple, hence annular of index $p$, and hence by Lemma~\ref{ann-powers}(6) that the pair $(R_{a,2},J)$ is annular of index 1, hence primitive.
Therefore (1) holds.

\medskip

{\bf Case 2:} $E\subset W_2$.

The disk $E$ is disjoint from $N(A)$ and so it is properly embedded in  $R_{2,1}$; therefore $R_{2,1}$ is a genus two handlebody by Lemma~\ref{genprop} (P1). We also have that the circles $\partial A,\partial E$ are mutually disjoint and, as $T_1,T_2$ are incompressible in $R_{2,1}$, $\partial E$ is not parallel to any component of $\partial A$.

If the disk $E\subset R_{2,1}$ is nonseparating then by \cite[Lemma 3.4]{valdez14} the circles $\partial A$ are coannular in $R_{2,1}$, contradicting Lemma~\ref{ann-powers}(4).
Therefore $E$ must be a separating disk in $R_{2,1}$ and so by \cite[Lemma 3.4]{valdez14} each circle $\partial_1A,\partial_2A$ is a primitive or power circle in $R_{2,1}$, and by Lemma~\ref{ann-powers}(4) not both can be power circles. 

In $R_{2,1}$, if the separated circles $\omega_1=\partial_1 A\subset T_1$ and $\omega_2=\partial_2A\subset T_2$ are both primitive then by \S\ref{basic} they are basic circles in $R_{2,1}$ and so the pair $(R_{2,1},J)$ is basic. 

Suppose now for definiteness that, in $R_{2,1}$, $\omega_1\subset T_1$ is a primitive circle and $\omega_2\subset T_2$ is a $p\geq 2$ power circle. Since $\omega_2$ is disjoint from the separating disk $E\subset R_{2,1}$ a companion solid torus of $V_2\subset R_{2,1}$ of $\omega_2$ can be isotoped so as to be disjoint from $E$. 
Therefore we may assume that the $J$-torus $T_3\subset R_{2,1}$ induced by $\omega_2$ as in \S\ref{induced} is disjoint from $E$, the pair $(R_{2,3},J)$ is simple of index $p$, and $E\subset R_{3,1}$.

Let $\omega_3\subset T_3$ be the power circle of the simple pair $(R_{2,3},J)$. As $R_{2,1}$ is a handlebody, by \S\ref{oper}(3)
the circle $\omega_3\subset T_3$ is primitive in $R_{3,1}$. The circle $\omega_1\subset T_1$ is primitive in $R_{2,1}$ and hence  it must be primitive in $R_{3,1}$. Since $E\subset R_{3,1}$ separates $\omega_3$ and $\omega_1$ it follows from \S\ref{basic} that the circles $\omega_3,\omega_1$ are basic in $R_{3,1}$ and hence that the pair $(R_{3,1},J)$ is basic.

Suppose now that $S$ is a $J$-torus in $R_{2,1}$. If $S$ is not isotopic to $T_1,T_2$ or $T_3$ in $R_{2,1}$ then by Lemma~\ref{new01c}(3) applied to $S$ and the collection $T_3\subset R_{2,1}$ we have that one of the pairs $(R_{2,3},J)$ or $(R_{3,1},J)$ must be primitive, which is not the case by Lemma~\ref{prim-2}(2). Therefore $S$ is isotopic in $R_{2,1}$ to $T_1,T_2$ or $T_3$ and hence (2) holds.
\epr

\begin{rem}\label{remkoz}
\begin{enumerate}
\item
The handlebody region $R_{1,2}$ in conclusion (1) of Lemma~\ref{koz1-b} which splits into a primitive and a simple pair is an example of an {\rm exchange region}. These regions are classified in general in Section~\ref{trick} and their properties will be used in Sections \ref{classif} and \ref{maxcoll} to obtain the restricted structure of of the complex $MS(K)$ in Theorem~\ref{main1}.

\item
The 4-punctured sphere $P\subset X_L$ constructed in the proof of Lemma~\ref{koz1-b} compresses in $X_L$ on one of its sides along a disk which is also a compression disk of either $\partial R_{1,2}$ in $R_{1,2}$ or $\partial R_{2,1}$ in $R_{2,1}$, corresponding to conclusions (1) and (2) of the lemma. If $P$ compresses on both sides then conclusions (1) and (2) hold simultaneously and hence a maximal simplicial collection of Seifert tori in $X_K$ has at most 4 components.

An example where the surface $P\subset X_L$ compresses only on its side contained in the annular pair $(R_{1,2},J)$ is provided by the family of knots constructed in Proposition~\ref{Khyper}(1) and represented in Fig.~\ref{k05}, right. In these examples the pair $(R_{2,3},J)$ is simple, hence annular, but $R_{3,2}$ does not satisfy conclusion (2) of Lemma~\ref{koz1-b}.
\end{enumerate}

\end{rem}

\subsection{Index 1 annular pairs in $X_K$.}
\label{52}
By Lemma~\ref{genprop}, for an index 1 annular pair $(R_{i,j},J)$ in $X_K$ the region $R_{i,j}$ may be a handlebody or a boundary irreducible manifold, and in the latter case the complementary region $R_{j,i}$ must be a handlebody. In this section we will see that this relationship between the regions $R_{i,j}$ and $R_{j,i}$, whose union is the exterior $X_K$ of the knot $K\subset\mS^3$, greatly limits the topology of the annular pair $(R_{i,j},J)$.

We first establish a technical result that applies to manifolds like the regions  $R_{i,j}\subset X_K$.

\begin{lem}\label{hhbdy}
Let $H$ be an irreducible manifold with $\partial H$ a surface of genus two, and let $\alpha,\beta,\gamma\subset\partial H$ be nonseparating circles such that 
\begin{enumerate}
\item $\alpha$ is disjoint from $\beta\cup\gamma$, 

\item $\beta$ and $\gamma$ intersect minimally in one point,

\item
$\alpha$ and $\beta$ cobound an annulus $A\subset H$,

\item $H(\alpha)$ and $H(\gamma)$ are solid tori.
\end{enumerate}
Then $H(\alpha\sqcup\gamma)=\mS^3$ and $H$ is either a genus two handlebody or a toroidal irreducible manifold with irreducible  boundary.
\end{lem}

\begin{proof}
The conditions (1)--(4) imply that the circles $\alpha,\beta,\gamma$ are nontrivial in $H$ and the circle $\alpha$ is not parallel in $\partial H$ to $\beta$ or $\gamma$.
Also the nonseparating annulus $A\subset H$ in (3) turns into the meridian disk $\wh{A}$ of the solid torus $H(\alpha)$ with $\beta=\partial \wh{A}$ intersecting $\gamma\subset \partial H(\alpha)$ minimally in one point. Therefore $\gamma$ is a longitude of $H(\alpha)$ and so 
\[
\mS^3=H(\alpha)(\gamma)=
H(\alpha\sqcup\gamma)=H(\gamma)(\alpha)
\] which implies that $\alpha$ is a longitude of the solid torus $H(\gamma)$.

\medskip

Let $\tau$ be the core of the 2-handle $\mD^2\times I$ used in the construction of $H(\gamma)$. Then $\tau$ is an arc properly embedded in the solid torus $H(\gamma)$ with regular neighborhood $N(\tau)=\mD^2\times I\subset H(\gamma)$, such that 

\begin{itemize}
\item[(i)]
$H\subset H(\gamma)$ is the closure of $H(\gamma)\setminus N(\tau)$,

\item[(i)]
$N(\gamma)=(\partial\mD^2)\times I\subset\partial H$ is an annular neighborhood of $\gamma$ in $\partial H$,

\item[(iii)]
$S_0=\cl[\partial H(\gamma)\setminus N(\tau) ] = \cl[\partial H\setminus N(\gamma)]$ is a twice punctured torus such that $\partial H=S_0\cup\partial N(\gamma)$,

\item[(iv)]
$\beta=\beta_1\cup\beta_2$, where
\begin{itemize}
\item
$\beta_1=\beta\cap S_0$ is an arc properly embedded in $S_0$ connecting the boundary components of $S_0$,
\item
$\beta_2=\beta\cap N(\gamma)$ is a spanning  arc of the annulus $N(\gamma)$ which intersects $\gamma$ minimally in one point.
\end{itemize}
\end{itemize}
The situation is represented in Fig.~\ref{m01}, top.

\medskip

Since the circle $\alpha$ and the arc $\beta_1$ are disjoint and properly embedded in the twice punctured torus $S_0\subset\partial H(\gamma)$ and the circle $\alpha$ is a longitude of the solid torus $H(\gamma)$, there is a meridian disk $E\subset H(\gamma)$ such that $\partial E$ lies in $S_0\subset\partial H(\gamma)$, intersects $\alpha$ minimally in one point, is disjoint from $\beta_1$, and intersects the arc $\tau\subset H(\gamma)$ minimally among all meridian disks of $H(\gamma)$ satisfying the previous conditions.

\begin{figure}
\Figw{3.5in}{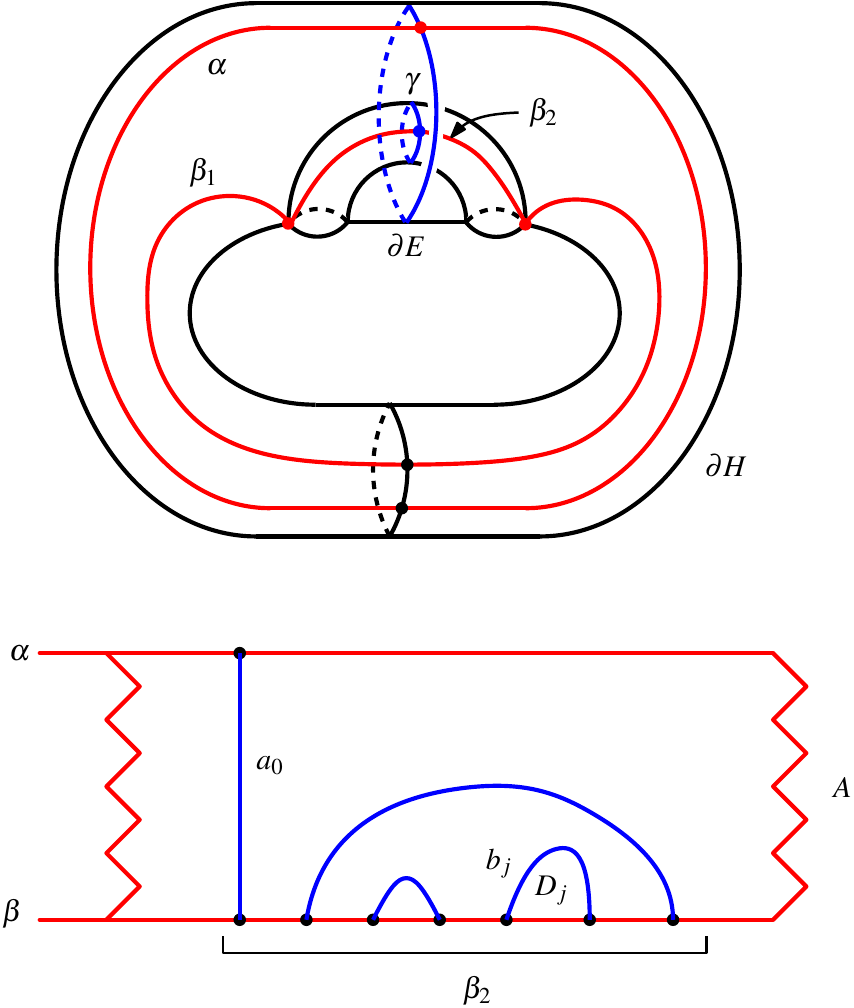}{The circles $\alpha,\beta=\beta_1\cup\beta_2,\gamma,\partial E$ in $\partial H$ (top) and the graph $G_A=A\cap E\subset A$ (bottom).}{m01}
\end{figure}

We may therefore assume that $F=E\cap H$ is a punctured disk with boundary the circle $\partial E\subset S_0\subset\partial H$ and some circle components parallel to $\gamma$ in the annulus $N(\gamma)\subset\partial H$. Keeping $\partial F$ fixed, we further isotope $F$ in $H$ so as to intersect the annulus $A\subset H$ minimally.

Since $\alpha\subset\partial A$ intersects $\partial E\subset\partial F$ in one point, and each component of $\partial F$ in $N(\gamma)\subset\partial H$ intersects $\beta\subset\partial A$ in one point, it follows that the graph $G_A=A\cap F\subset A$ consists of one spanning arc $a_0\subset A$ and perhaps some arcs  $b_i$, $1\leq i\leq n$, each with both boundary points on the subarc $\beta_2$ of the circle $\beta\subset\partial A$.

The presence of the spanning arc $a_0$ implies that any circle component of $A\cap F$ is trivial in $A$. If an innermost such circle component $c$ is nontrivial in $F$ then surgerying $E$ along the disk bounded in $A$ by $c$ produces a meridian disk for $H(\gamma)$ satisfying all conditions above but having fewer intersections with $A$, contradicting the minimality of $A\cap F$. Therefore $A\cap F$ has no circle components and so the graph $G_A$ consists only of arc components, as represented in Fig.~\ref{m01}, bottom.

If the arcs $b_i$ are present then the graph $G_A$ has a disk face $D_j$ cobounded by a subarc of $\beta$ and an outermost arc $b_j$; but then the disk $D_j$ may be used to boundary compress $F$ in $H$ and reduce by 2 the number of intersections in $H(\gamma)$ between $E$ and $\tau$, contradicting the minimality of $E\cap\tau$.

\medskip

Therefore $E\cap\tau$ consists of a single point, so $F\subset H$ is an annulus and $A\cap F$ consists of the single arc $a_0$. This final situation is represented in Fig.~\ref{m01}, top.

\medskip

It follows that $W=N(A\cup F)\subset H$ is a product of the form $T_0\times I$, where $T_0=T_0\times\{0\}$ is the once-punctured torus $N(\alpha\cup\partial E)\subset\partial H$ and $T_0\times\{1\}$ the once-punctured torus $N(\beta\cup\gamma)\subset\partial H$. The circles $\partial T_0$ and $\partial T_1$ are then separating and disjoint in $\partial H$ and hence cobound an annulus $B_0\subset\partial H$; moreover $B=\fr\, W=(\partial T_0)\times I$ is a separating annulus properly embedded in $H$ with $\partial B=\partial B_0$.

\medskip

Let $V=\cl[H\setminus W]\subset H$ be the region in $H$ cobounded by $B_0$ and $B$, so that $H=W\cup_BV$. Since $\partial V=B_0\cup B$ is a torus and $V\subset H\subset\mS^3$, $V$ is either a solid torus or the exterior of a nontrivial knot in $\mS^3$. Since the circle $\partial T_0\subset\partial V$ bounds the surface $T_0$ outside $V$, if $V\subset\mS^3$ is a solid torus then 
$\partial T_0$ runs once around $V$. Therefore $\partial T_0\subset\partial B_0$ is a longitude of $V$ and so $V$ is a parallelism between the annuli $B_0$ and $B$ in $H=W\cup_BV$. It follows that $H\approx W=T_0\times I$ is a genus two handlebody.

Suppose now that $V$ is the exterior of a nontrivial knot in $\mS^3$. Then the torus $\partial V$ is incompressible in $V$ and the annulus $B=(\partial T_0)\times I$, which is incompressible in $W=T_0\times I$, is therefore incompressible in $H=W\cup_BV$. It follows that $H$ is an irreducible and boundary irreducible manifold and that $\partial V$, when pushed slightly into the interior of $H$, is an incompressible torus in $H$.
\end{proof}

With the help of Lemmas~\ref{koz1-b} and \ref{hhbdy} we now obtain more information about the topology of an index 1 annular pair $(R_{i,j},J)$ in $X_K$ and spanning annulus $A\subset R_{i,j}$ by analyzing the manifold $R_{i,j}(\partial A)$.

\begin{lem}\label{k2}
Let $T_1\sqcup T_2\subset X_K$ be a simplicial collection of Seifert tori
that cobound an annular pair $(R_{1,2},J)$ of index 1 with spanning annulus $A\subset R_{1,2}$, such that $\partial_1A\subset T_1$ and $\partial_2A\subset T_2$. Then
\begin{enumerate}
\item
there is a closed 3-manifold $M$ such that, for $\{i,j\}=\{1,2\}$, $R_{1,2}(\partial_i A)=\mS^1\times\mD^2\,\#\,M$ with $\partial_{j}A$ the slope of the meridian of the solid torus summand $\mS^1\times \mD^2$,

\item
if $R_{1,2}(\partial_1A)$ is a solid torus then $R_{1,2}$ is a handlebody and so the pair $(R_{1,2},J)$ is primitive, hence minimal.
\end{enumerate}
\end{lem}

\begin{proof}
Observe that for $\{i,j\}=\{1,2\}$ the boundary of the manifold $R_{1,2}(\partial_i A)$ is a torus. Since 
the spanning annulus $A\subset R_{1,2}$ turns into a disk in $R_{1,2}(\partial_i A)$ with boundary the circle $\partial_{j}A\subset\partial R_{1,2}(\partial_i A)$, it follows that 
$R_{1,2}(\partial_iA)=\mS^1\times\mD^2\#M_i$ for some closed 3-manifold $M_i$, where the meridian slope of the solid torus factor $\mS^1\times\mD^2$ is $\partial_{j}A$. 
Therefore
\begin{align*}
R_{1,2}(\partial A)&=R_{1,2}(\partial_1A)(\partial_2A)=\mS^1\times\mS^2\,\#\,M_1\\
&=R_{1,2}(\partial_2A)(\partial_1A)=\mS^1\times\mS^2\,\#\,M_2
\end{align*}
whence $M_1\approx M_2$ by uniqueness of prime factorization. Thus (1) holds.

\medskip
For part (2) suppose that $R_{1,2}(\partial_1A)$ is a solid torus and $R_{1,2}$ is not a handlebody. By Lemma~\ref{koz1-b} the region $R_{2,1}$ is a handlebody, the circles $\partial_1A\subset T_1$ and $\partial_2A\subset T_2$ are separated by a disk in $R_{2,1}$, and we may assume that $\partial_1A\subset T_1$ is a primitive circle in $R_{2,1}$.
Therefore there is a disk $D\subset R_{2,1}$ which intersects $\partial_1A$ transversely in one point and is disjoint from $\partial_2A$. 

Let $V\subset R_{2,1}$ be the solid torus $R_{2,1}|D$ and denote its core by $L\subset V$. The exterior of the knot $L\subset\mS^3=R_{1,2}\cup R_{2,1}$ is then the manifold $X_L=\mS^3\setminus\intr\,V\approx R_{1,2}(\partial D)$

By (1) the manifold $R_{1,2}(\partial_2A)$ is a solid torus with meridian slope the circle $\partial_1A$, and since $\Delta(\partial D,\partial_1A)=1$ and $\partial D\cap\partial_2A=\emptyset
$ it follows that  
\begin{align*}
\mS^3&=R_{1,2}(\partial_2A)(\partial D)\\
&= R_{1,2}(\partial D)(\partial_2A)\\
&= X_{L}(\partial_2A)
\end{align*}
Since $\partial_2A$ does not bound a disk in $R_{2,1}$, hence neither in the solid torus $V$, by \cite{gordonlu1} the knot $L\subset\mS^3$ is trivial and hence $R_{1,2}(\partial D)\approx X_L$ is a solid torus.
But then by Lemma~\ref{hhbdy} applied to the 4-tuple
$(H,\alpha,\beta,\gamma)=(R_{1,2},\partial_1A,\partial_2A,\partial D)$ the manifold $R_{1,2}$ must be toroidal, contradicting Lemma~\ref{genprop}(P1). Therefore $R_{1,2}$ is a handlebody and so the pair $(R_{1,2},J)$ is primitive, hence minimal by Lemma~\ref{prim-2}.
\end{proof}

The above results can now be combined to obtain the minimality of any index 1 annular pair in $X_K$. As a consequence we extend the classification of handlebody annular pairs in $X_K$ given in Lemma~\ref{koz1-b}(1) to include nonhandlebody such pairs.

For convenience we may denote by $R_{S,S'}$ and $R_{S',S}$ the regions in $X_K$ cobounded by two disjoint Seifert tori $S,S'\subset X_K$, so that $X_K=R_{S,S'}\cup R_{S',S}$.

\begin{prop}\label{annmin}
Let $\mT=T_1\sqcup\cdots\sqcup T_N$, $N\geq 1$, be a simplicial collection of  Seifert tori in $X_K$ with minimal pairs. Then
\begin{enumerate}
\item
any index 1 annular pair $(R_{i,j},J)$ ($i=j$ allowed) in $X_K$ is minimal, and if $R_{i,j}$ is not a handlebody then $|\mT|\leq 3$,

\item
if $|\mT|\geq 2$ and for some $k\geq 2$ the nonminimal pair $(R_{1,1+k},J)$ is annular of index $p\geq 2$ then $k=2$ and $R_{1,1+k}=R_{1,3}$, and if $T'_1,T'_3\subset R_{1,3}$ are the $J$-tori induced by $T_1,T_3$, respectively, then
\begin{enumerate}
\item
$T'_1$ and $T'_3$ are not isotopic in $X_K$,

\item
any $J$-torus in $R_{1,3}$ ($T_{2}$ for instance) is isotopic to $T_1$, $T_{3}$, $T'_1$ or $T'_3$.

\item
the pair $(R_{T_1,T'_1},J)$ is simple of index $p$ and $(R_{T'_1,T_3},J)$ is annular of index 1,

\item
the pair $(R_{T_1,T'_3},J)$ is annular of index 1 and $(R_{T'_3,T_3},J)$ is simple of index $p$.
\end{enumerate}
\end{enumerate}
\end{prop}

\begin{proof}
If the annular pair $(R_{i,j},J)$ is minimal and $R_{i,j}$ is not a handlebody then $R_{j,i}$ is a handlebody by Lemma 2.1.1 (P2), in which case by Lemma 5.1.1(2) either 
\begin{itemize}
\item
the pair $(R_{j,i},J)$ is basic, hence minimal by Lemma 2.3.12(1); in this case we obtain $|\mT|\leq 2$, where $|\mT|=1$ if the basic pair $(R_{j,i},J)$ is trivial,

\item
$R_{j,i}$ contains exactly one $J$-torus not parallel to $T_i$ or $T_j$, in which case $|\mT|=3$.
\end{itemize}
Therefore the second part of (1) follows from the first part. For the first part of (1) we argue by contradiction.
For definiteness suppose that the pair $(R_{1,j},J)$ is annular of index 1 for some $j\neq 1,2$, so that $R_{1,3}\subseteq R_{1,j}$ with $(R_{1,3},J)$ a nonminimal pair. By Lemma~\ref{ann-powers}(6) we may therefore assume that $j=3$. 
Let $A\subset R_{1,3}$ be a spanning annulus. 

\medskip
(I) 
By Lemma~\ref{ann-powers}(6) we may assume that 
$A_1=A\cap R_{1,2}$ and $A_2=A\cap R_{2,3}$ are spanning annuli of index 1 in $R_{1,2}$ and $R_{2,3}$, respectively. Therefore each of the pairs $(R_{1,2},J)$ and $(R_{2,3},J)$ is annular of index 1 and so the region $R_{1,3}$ is not a handlebody by \cite[Lemma 3.7(3)]{valdez14}. 

\medskip
(II) {\it $R_{1,2}$ and $R_{2,3}$ are handlebodies.}\\
For if $R_{1,2}$ is not a handlebody then $R_{2,3}\subseteq R_{2,1}$ is a handlebody by Lemma~\ref{genprop}(P2),(P3) and hence the pair $(R_{2,3},J)$ is primitive by (I). This contradicts Lemma~\ref{koz1-b}(2)(b) since a primitive pair is neither basic nor simple by Lemma~\ref{prim-2}. Therefore $R_{1,2}$ is a handlebody, and by a similar argument $R_{2,3}$ is also a handlebody.

\medskip
(III) {\it $R_{1,3}(\partial_1A)$ is a solid torus.}\\
For by (I) and (II) the pairs $(R_{1,2},J)$ and $(R_{2,3},J)$ are primitive with spanning annuli $A_1\subset R_{1,2}$ and $R_{2,3}$. 
Denote the boundary components of $A_1$ by $\omega_1=A_1\cap T_1$ and $\omega_2=A_1\cap T_2$; these are primitive circles in $R_{1,2}$, and $\omega_2=A_2\cap T_2$ is primitive in $R_{2,3}$. 
Therefore $R_{1,2}(\omega_1)$ is a solid torus with meridian disk $\wh{A}_1$ such that $\partial\wh{A}_1=\omega_2\subset\partial R_{1,2}(\omega_1)$, and $R_{2,3}(\omega_2)$ is also a solid torus.

For $i=1,2$ each manifold $R_{i,i+1}(J)$ has boundary the tori $\wh{T}_i$ and $\wh{T}_{i+1}$, and $R_{1,3}(J)=R_{1,2}(J)\cup_{\wh{T}_2}R_{2,3}(J)$.

Therefore
\begin{align*}
R_{1,3}(\partial_1A)=R_{1,3}(\omega_1)&=R_{1,3}(J)(\omega_1)=[R_{2,3}(J)\cup_{\wh{T}_2} R_{1,2}(J)](\omega_1)\\
&=R_{2,3}(J)\cup_{\wh{T}_2} [R_{1,2}(J)(\omega_1)]
\\
&=R_{2,3}(J)(\partial\wh{A}_1)=R_{2,3}(\omega_2)=\text{ solid torus}
\end{align*}
This contradicts Lemma~\ref{k2}(2) since by (I) the region $R_{1,3}$ is not a handlebody. Therefore (1) holds.

\medskip
Part (2)(a) will be established in the next section in Lemma~\ref{nonpar}.
Parts (2)(c) and (2)(d) follow from Lemma~\ref{ann-powers}(6) and the properties of induced tori in \S\ref{induced}.

\medskip
For part (2)(b) suppose that $S\subset R_{1,3}$ is a $J$-torus that is not parallel to $T_1$ or $T_3$. By Lemma~\ref{ann-powers}(6) we may assume that the pair $(R_{T_1,S},J)$ is annular of index $p\geq 2$ while $(R_{S,T_3},J)$ is annular of index 1. 

By \S\ref{index} (A4) the $J$-torus $T'_1$ induced by $T_1$ in $R_{T_1,S}$ is isotopic to the $J$-torus induced by $T_1$ in $R_{1,3}$, so it is not parallel to $T_3$ in $R_{1,3}$, and cobounds a region $R_{T_1,T'_1}\subset R_{T_1,S}$ such that the pair $(R_{T_1,T'_1},J)$ is simple of index $p\geq 2$.

By Lemma~\ref{ann-powers}(6) the pair $(R_{T'_1,T_3},J)$ is then  annular of index 1 and hence minimal by (1). As $R_{T'_1,T_3}=R_{T'_1,S}\cup_S R_{S,T_3}$ and the pair $(R_{S,T_3},J)$ is nontrivial, it follows that the pair $(R_{T'_1,S},J)$ must be trivial and hence that $S$ is parallel to $T'_1$. Therefore (2)(b) holds.
\end{proof}

Examples of genus one hyperbolic knots in $\mS^3$ realizing the conditions of Proposition~\ref{annmin}(1), with $R_{i,j}$ a handlebody or not, can be found in Sections~\ref{82}, \ref{hk23} and \ref{hk23-2}.

\subsection{Exchange regions and the exchange trick.}\label{trick}

By Proposition~\ref{annmin}(2)(c)(d), 
given an annular pair $(R_{1,3},J)$ of index $p\geq 2$ of a simplicial collection of Seifert tori $\mT\subset X_K$ with minimal pairs, the region $R_{1,3}\subset X_K$ contains two nontrivial minimal subpairs $(R_{T_1,T},J)$ and $R_{T,T_3},J)$ where the nature of each subpair alternates between being simple of index $p$ or annular of index 1 depending on the choice of splitting $J$-torus $T\subset R_{1,3}$. 

We will refer to any region in $X_K$ with properties similar to those of the region $R_{1,3}\subset X_K$ as an {\it exchange region}, to the pair $(R_{1,3},J)$ as an {\it exchange pair,}
and to the switch of type of subpair in $R_{1,3}$ adjacent to $T_1$ or $T_3$ between a simple and an index 1 pair as the {\it exchange trick}.

In this section we prove Lemma~\ref{nonpar} which states that the two induced $J$-tori in an exchange region are not isotopic in $X_K$, thus completing the proof of Proposition~\ref{annmin}(2)(a). 
We first review the construction of the induced $J$-tori $T'_1,T'_3\subset R_{1,3}$ given in \S\ref{induced} and \S\ref{index}(A4).

By hypothesis the pair $(R_{1,3},J)$ is annular of index $p\geq 2$ and so by Lemma~\ref{annp}(5) there are disjoint spanning annuli $A,A'\subset R_{1,3}$ which cobound a solid torus $V\subset R_{1,3}$ around which each spanning annulus runs $p$ times.

Push $V$ off $T_3$ and into $R_{1,3}$ to obtain a companion solid torus $V_1\subset R_{1,3}$ of the circle $\omega_1=A\cap T_1$. Similarly the circle $\omega_3\subset T_3$ has a companion solid torus $V_3\subset R_{1,3}$. For $i=1,3$ the frontier of a thin regular neighborhood $N(T_i\cup V_i)\subset R_{1,3}$ then consists of $T_i$ and the $J$-torus $T'_i$ induced by $T_i$. The situation is represented in Fig.~\ref{h11-2}.

Notice that $T'_1$ and $T'_3$ intersect transversely in two circles $T'_1\cap T'_3=\alpha\sqcup\beta$ that are nonseparating in $T'_1$ and $T'_3$.
The closures of the components of $T'_1\setminus T'_3$ consist of an annulus $B_1$ and a pants $P_1$, and those of $T'_3\setminus T'_1$ of an annulus $B_{3}$ and a pants $P_{3}$, with $\partial B_1=\alpha\sqcup\beta=\partial B_3$ as shown in Fig.~\ref{h11-2}.

By the construction of the induced tori $T'_1,T'_3$, the torus $B_1\cup B_3\subset R_{1,3}$ bounds a solid torus $W\subset R_{1,3}$ obtained by pushing the solid torus $V\subset R_{1,3}$ off from $T_1,T_3$, as represented in Fig.~\ref{h11-2}). Since the index of $(R_{1,3},J)$ is $p\geq 2$, each spanning annulus runs $p$ times around $V$ and hence each circle $\alpha,\beta$ runs $p$ times around $W$. 

\begin{figure}
\Figw{3in}{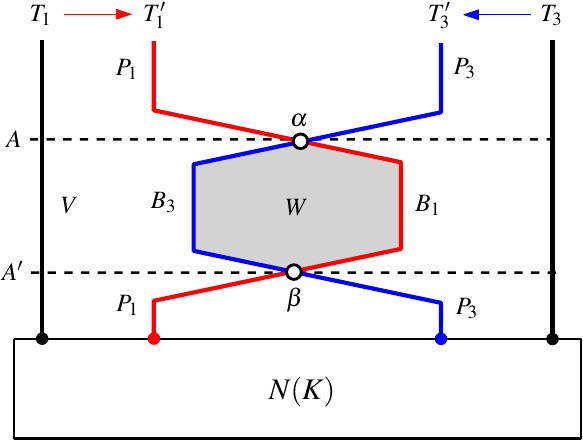}{The induced $J$-tori $T'_1,T'_3$ in the exchange region $R_{1,3}\subset X_K$.}{h11-2}
\end{figure}

In order to establish the isotopy properties of the induced $J$-tori $T'_1,T'_3\subset R_{1,3}$ we use the following result of 
\cite[Proposition 4.8]{sakuma4},
an elaboration of the results in \cite{wald1}.

Let $P,Q$ be surfaces properly embedded in a 3-manifold $M$ and which intersect transversely. A {\it product region} between $P$ and $Q$ is an embedded copy of a manifold of the form $\wt{\Sigma}=\Sigma\times I/\sim$ in $M$, where $\Sigma$ is a surface, $b$ is a compact 1-submanifold of $\partial\Sigma$, and
\begin{enumerate}
\item[(i)]
for each $x\in b$ the relation $\sim$ collapses the arc $\{x\}\times I$ to a point,

\item[(ii)]
$\Sigma\times\{0\}\subset P$, $\Sigma\times\{1\}\subset Q$, and 
$\cl[(\partial\Sigma-b)]\times I\subset\partial M$,

\item[(iii)]
$P\cap\intr\,\wt{\Sigma}=\emptyset$, and $Q\cap\intr\,\wt{\Sigma}$ may be nonempty only when $\Sigma$ is a disk and $P\cap\partial M$ is connected.
\end{enumerate}

\begin{lem}[\cite{sakuma4}]\label{sakuma}
Let $M$ be a Haken 3-manifold with incompressible boundary, and let $P,Q\subset M$ be properly embedded incompressible and boundary incompressible surfaces which intersect transversely with $\partial P\cap\partial Q=\emptyset$. If in $M$ the surfaces $P$ and $Q$ are isotopic or $P$ is isotopic to a surface disjoint from $Q$ then there is a product region between $P$ and $Q$.
\hfill\qed
\end{lem}

In the presence of a product region $\wt{\Sigma}$ between $P$ and $Q$  it is possible to reduce $|P\cap Q|$ by an isotopy that exchanges $\Sigma\times\{1\}\subset Q$ with $\Sigma\times\{0\}$ and pushes the resulting new surface $Q$ slightly off $P$.

We will apply the lemma to surfaces with disjoint boundary, in which case the 1-submanifold $b\subset\partial\Sigma$ is simply a union of components of $\partial\Sigma\subset \partial P\sqcup\partial Q\sqcup (P\cap Q)$.

\begin{lem}\label{nonpar}
The Seifert tori $T'_1,T'_3\subset X_K$ are nonisotopic in $X_K$.
\end{lem}

\begin{proof}
By Lemma~\ref{sakuma} it suffices to show that there are no product regions in $X_K$ between $T'_1$ and $T'_3$.

By Proposition~\ref{annmin}(2) for $i=1,3$ the pairs $(R_{T_1,T'_i},J)$ and $(R_{T'_i,T_3},J)$ are nontrivial  and so
the induced $J$-tori $T'_1,T'_3$ are not parallel to $T_1$ nor $T_3$ in $R_{1,3}$. Therefore $T'_1$ and $T'_3$ intersect minimally in $R_{1,3}$ and so by Lemma~\ref{sakuma} there are no product regions between $T'_1$ and $T'_3$ contained in the region $R_{1,3}$.  

%

By the above construction of the induced tori $T'_1$ and $T'_3$
any product region $\wt{\Sigma}$ in $X_K$ between $T'_1$ and $T'_3$ must run between the following subsurfaces of $T'_1$ and $T'_3$:

\medskip
{\it (a) $B_1$ and $B_3$:} Here the only possible product region $\wt{\Sigma}\subset X_K$ must be of the form
\[
\Sigma=B_1,  \quad 
b=\partial B_1=\partial B_{3},  \quad
\wt{\Sigma}\cap T'_1=
\Sigma\times\{0\}=B_1, \quad
\wt{\Sigma}\cap T'_3=
\Sigma\times\{1\}=B_{3}, \quad
\wt{\Sigma}\cap\partial X_K=\emptyset
\]
whence necessarily $\wt{\Sigma}=W\subset R_{1,3}$, contradicting the argument above that $\wt{\Sigma}\not\subset R_{1,3}$.

\medskip
{\it (b)  $P_1$ and $P_{3}$:}
The surface $P_1\cup P_{3}$ is a separating twice punctured torus properly embedded in $X_K$ and so any product region $\wt{\Sigma}$ between $P_1$ and $P_{3}$ must be constructed from 
\[
\Sigma=P_1,  \quad 
b=\partial P_1\cap\partial P_{3}=\alpha\sqcup\beta,  \quad
\wt{\Sigma}\cap T'_1=
\Sigma\times\{0\}=P_1, \quad
\wt{\Sigma}\cap T_{3}=
\Sigma\times\{1\}=P_{3}, \quad
\wt{\Sigma}\cap\partial X_K=(\partial T'_1)\times I
\]
Thus $\wt{\Sigma}\approx P_1\times I$ is a genus two handlebody such that each component of $\partial P_1\supset\alpha\sqcup\beta$ is a primitive circle in $\wt{\Sigma}$.
However, as $\wt{\Sigma}\not\subset R_{1,3}$, we must have $W\subset\wt{\Sigma}$ (see Fig.~\ref{h11-2}) which implies that $\alpha$ and $\beta$ are $p\geq 2$ power circles in $\wt{\Sigma}$. 

This last contradiction shows that there are no product regions in $X_K$ between $T'_1$ and $T'_3$.
\end{proof}

\section{No exchange regions for $|\mT|=5$.}\label{classif}

In this section we assume that $K\subset\mS^3$ is a genus one hyperbolic knot which bounds a maximal simplicial collection of five Seifert tori $\mT=T_1\sqcup\cdots\sqcup T_5\subset X_K$.
Our goal is to prove the following result:

\begin{prop}\label{maxcase2}
If $|\mT|=5$ then no pair $(R_{i,i+2},J)$ is an exchange pair and the maximal simplicial collection $\mT\subset X_K$ is unique up to isotopy.
\end{prop}

A sketch of the proof goes like this. Both an exchange region, say $R_{1,3}$, and its complementary region $R_{3,1}$ must be genus two handlebodies; thus the pair $(R_{3,1},J)$ is maximal. 
At this point we use a method developed in \cite[\S 7.3]{valdez14} to construct a Heegaard diagram from the Heegaard decomposition $R_{1,3}\cup R_{3,1}$ which applies whenever one of the pairs $(R_{1,3},J)$ or $(R_{3,1},J)$ is maximal. 
That the Heegaard decomposition $R_{1,3}\cup R_{3,1}$ cannot correspond to $\mS^3$ can then be detected from the fact that otherwise the core knot of a simple pair in $R_{1,3}$ should be a trivial or torus knot in $\mS^3$, which we show cannot be the case.

\subsection{The regions $R_{i,i+1}$.}
\label{61}
The following general result restricts the types of pairs $(R_{i,i+1},J)$ produced by the simplicial collection $\mT\subset X_K$. Its proof relies on an analysis of the essential graphs of intersection between $\mT$ and a {\it Gabai} meridional planar surface for $\mT$ from \cite{gabai03}, along with some basic results from \cite[\S2.1--2.2]{valdez14}.

\begin{lem}\label{f1}
For each $1\leq i\leq 5$ the region $R_{i,i+1}$ is a handlebody, the pair $(R_{i,i+1},J)$ is minimal, and at least one of the pairs $(R_{i,i+1},J)$ or $(R_{i+1,i+2},J)$ is simple.
\end{lem}

\bpr
Each pair $(R_{i,i+1},J)$ is minimal since the simplicial collection $\mT$ is maximal, and each region $R_{i,i+1}$ is a handlebody by \cite[Lemma 4.1(3)]{valdez14}. We show that of any two consecutive pairs, say $(R_{1,2},J)$ and $(R_{2,3},J)$, at least one of them is simple. 

Let $\mT'=T_1\sqcup T_2\sqcup T_3$. By \cite{gabai03}
there is a planar surface $Q\subset X_K$ with meridional boundary slope which intersects $\mT'$ transversely in essential graphs $G_Q=Q\cap\mT'\subset Q$ and $G'=Q\cap\mT'\subset\mT'$. Necessarily  each cycle around a face of $G_Q$ has an even number of edges.

If the graph $G_Q$ has no parallel edges then it is a reduced graph and the degree of each of its vertices is 3; therefore by \cite[Lemma 2.3(2)]{valdez14} $G_Q$ has a disk face $D_4$ with 4 edges around its boundary. Otherwise $G_Q$ has parallel edges, that is, $G_Q$ has a disk face $D_2$ with 2 edges around its boundary.

The disk face $E\in\{D_2,D_4\}$ of $G_Q$ lies in one of the regions $R_{1,2},R_{2,3}, R_{3,1}$ and by \cite[Lemma 2.1(3)]{valdez14} intersects $J$ minimally in 2 or 4 points. 

If $E\subset R_{i,j}$ then $R_{i,j}$ is a handlebody by Lemma~\ref{genprop} (P1) and so, by \cite[Lemma 6.1]{valdez14}, $(R_{i,j},J)$ is either a simple pair, which is minimal, or a nonminimal double pair.
However, by \S\ref{double}, in a double pair any $J$-torus is parallel to a boundary $J$-torus or to the $J$-torus separating the double pair into simple subpairs.
Since $R_{3,1}$ contains the two $J$-tori $T_4, T_5$, which are not parallel to the boundary nor to each other, it follows that $R_{i,j}\neq R_{3,1}$. 

Therefore $E\subset R_{1,2}$ or $E\subset R_{2,3}$, in which case, respectively, the minimal pair $(R_{1,2},J)$ or $(R_{2,3},J)$ is simple by \cite[Lemma 6.1]{valdez14}.
\epr

For the rest of Section~\ref{classif} we assume that $R_{1,3}$ is the exchange region for the simplicial collection $\mT$ with exchange $J$-tori $T_2$ and $T_{2'}$ as shown in Fig.~\ref{k10b}. In the figure the elements of each simple pair will be represented using the notation set up in Section~\ref{simple} and Fig.~\ref{k26}. Circles of distinct slope, like $\omega_1,\omega'_5\subset T_1$, are represented as nonoverlapping. 

In the exchange region $R_{1,3}$ the pairs $(R_{1,2},J)$ and $(R_{2',3},J)$ in Fig.~\ref{k10b} are simple with core knots $K_1$ and $K_{2'}$, respectively.

\subsection{The regions $R_{i,j}$.}
\label{62}
In this section we establish some of the general properties of the pairs $(R_{i,j},J)$. We will see that indeed each of the pairs $(R_{i,i+1},J)$ is simple for $i=3,4,5$ as represented in Fig.~\ref{k10b}.

We will use the following notation: A circle $\gamma$ in the boundary of a genus two handlebody $H$ is a {\it Seifert circle} if $H(\omega)=\mD^2(p,q)$ for some integers $p,q\geq 2$.

\begin{figure}
\Figw{5.5in}{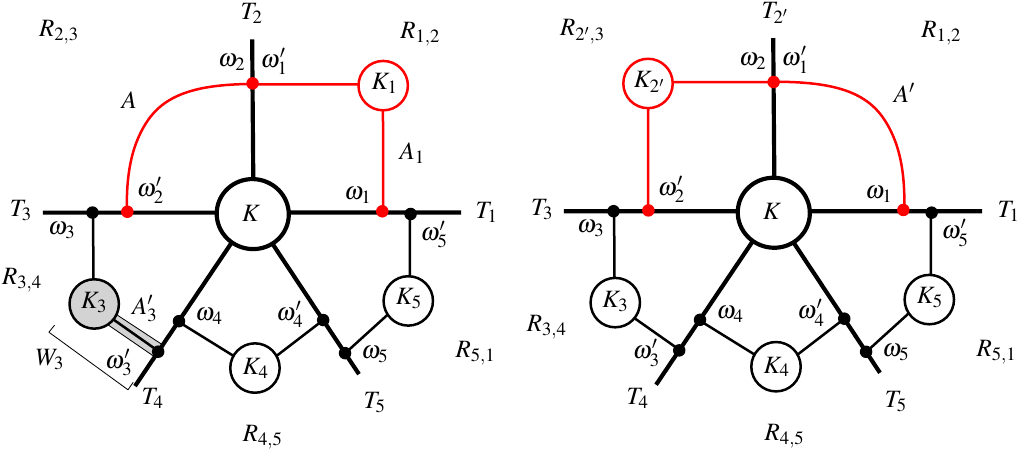}{The knot $K\subset\mS^3$ with the primitive pair $(R_{2,3},J)$.}{k10b}
\end{figure}

\begin{enumerate}
\item[(E1)]
By Lemma~\ref{f1} applied to the regions $R_{2,4}$ and $R_{5,2'}$, each of the pairs $(R_{3,4},J)$ and $(R_{5,1},J)$ is simple, with cores the knots $K_3,K_5$, respectively. 

In Fig.~\ref{k10b}, left, $\omega'_2\neq \omega_3$ by Lemma~\ref{ann-powers}(4) and so the region $R_{2,4}$ is not a handlebody by Lemma~\ref{ann-powers}(7). Therefore the region $R_{4,2}$ is a handlebody by Lemma~\ref{genprop} (P2) and so the pair $(R_{4,2},J)$ is maximal. By \cite[Lemma 6.8]{valdez14} it follows that 

\begin{enumerate}
\item[(a)]
$(R_{4,5},J)$ is a simple pair with core the knot $K_4$,

\item[(b)]
the simple pair $(R_{5,1},J)$ is a basic pair with basic circles $\omega'_4\subset T_5$ and $\omega_1\subset T_1$,

\item[(c)]
$\Delta(\omega'_4,\omega_5)=1=\Delta(\omega'_5,\omega_1)$.
\end{enumerate}

A similar argument using Fig.~\ref{k10b}, right, shows that 

\begin{enumerate}
\item[(d)]
the simple pair $(R_{3,4},J)$ is a basic pair with basic circles $\omega'_2\subset T_3$ and $\omega_4\subset T_4$,

\item[(e)]
$\Delta(\omega'_2,\omega_3)=1=\Delta(\omega'_3,\omega_4)$.
\end{enumerate}
\end{enumerate}

Let $p_1,p_2,p_3,p_4,p_5\geq 2$ be the indices of the simple pairs 
$(R_{1,2},J)$, $(R_{2',3},J)$, $(R_{3,4},J)$, $(R_{4,5},J)$, $(R_{5,1},J)$, respectively.

\medskip
\begin{enumerate}
\item[(E2)]
{\it The region $R_{3,1}$ is a handlebody and the core knots $K_3,K_5$ are hyperbolic Eudave-Mu\~noz knots of indices $p_3=2=p_5$.}
\\
Let $W_3$ be the solid torus neighborhood of $K_3\sqcup A'_3$ in $R_{3,4}$ indicated in Fig.~\ref{k10b}, left, where $A'_3$ is the annulus constructed in \S\ref{simple} (see also Fig.~\ref{k26}) such that the slope $r_3$ of the boundary circle $A'_3\cap\partial N(K_3)$ is nonintegral (relative to $N(K_3)$) of the form $*/p_3$. We identify the exterior $X_{K_3}\subset\mS^3$ of the core knot $K_3$ with $\mS^3\setminus\intr\,W_3$. 
   Back in Fig.~\ref{k10b}, left, observe that
\begin{itemize}
\item
by (E1), $(R_{3,4},J)$ is a simple pair and $R_{2,4}$ is not a handlebody,

\item
$R_{3,2}$ is not a handlebody by Lemma~\ref{genprop} and \S\ref{max}.
\end{itemize}
Therefore \cite[Lemma 7.1(1)]{valdez14} applies to the simple pair $(R_{3,4},J)$ (denoted $(R_{3,4},K)$ in \cite{valdez14}) and the $J$-tori $T_2,T_3,T_4$ to conclude that the separating two-punctured torus $F=\cl[T_2\cup T_4\setminus W_3]$ is incompressible in $X_{K_3}$ and the torus $\wh{F}$ is incompressible in $X_{K_3}(r_3)$.
   Moreover the closures $F^+,F^-$ of the components of $X_{K_3}\setminus F$ can be identified with the handlebodies $F^+=R_{4,2}$ and 
$F^-=\cl[R_{2,4}\setminus W_3]\approx R_{2,3}$ (by \S\ref{oper}(1)), and since the slope $r_3\subset\partial X_{K_3}$ is nonintegral (with denominator $p_3\geq 2$) it follows from \cite[Lemma 3.14]{ozawa01} that $K_3$ is a hyperbolic Eudave-Mu\~noz knot. 

By \cite{gordonlu6} $r_3$ is the unique nonintegral toroidal slope for $K_3$ and we must have $p_3=2$. By \cite[Theorem 2.1 and Proposition 2.2]{eudave2} the torus $\wh{F}$ is the unique essential torus in $X_{K_3}(r_3)$ and it decomposes $X_{K_3}(r_3)$ into a union of two Seifert fiber spaces of the form $\mD^2(*,*)$ for $*\geq 2$. As
\[
X_{K_3}(r_3)\approx R_{2,3}(\omega'_2) \cup_{\wh{F}} R_{4,2}(\omega'_3)
\]
it follows that $R_{4,2}(\omega'_3)$ is a space of the form $\mD^2(*,*)$ for $*\geq 2$ and hence that $\omega'_3\subset T_4$ is a Seifert circle in the handlebody $R_{4,2}$. As $\partial R_{4,2}\setminus\omega'_3$ contains the power circle $\omega'_1\subset T_2$, \cite[Lemma 6.10(2)(b)]{valdez14} applied to the pair $(R_{4,2},K)$ and the $J$-torus $T_1\subset R_{4,2}$ yields that $\omega'_3$ is a primitive circle in the handlebody $R_{4,1}$, which by \cite[Lemma 7.2(5)]{valdez14} implies that $R_{3,1}$ is a handlebody.

\medskip
Using Fig.~\ref{k10b}, right, a symmetric argument shows that $K_5$ is also a hyperbolic Eudave-Mu\~noz knot and $p_5=2$.

\medskip
\item[(E3)]
{\it The core knot $K_1$ (Fig.~\ref{k10b}, left) is a trivial or torus knot.}
\\
Since the pair $(R_{2,3},J)$ is primitive with spanning annulus $A$, $R_{2,3}(\omega'_1)$ is a solid torus with meridian disk $\wh{A}$ and hence meridian slope $\partial\wh{A}=\omega'_2\subset T_3$. As the circles $\omega'_2,\omega_4$ are basic in $R_{3,4}$, it follows that 
$V_3=R_{2,4}(\omega'_1)=R_{2,4}(J)(\omega'_1)$ is a solid torus with meridian slope that intersects $\omega_4$ in one point. Similarly, as the circles $\omega_1,\omega'_4$ are basic in $R_{5,1}$, $V_5=R_{5,1}(\omega_1)=R_{5,1}(J)(\omega_1)$ is a solid torus with meridian slope that intersects $\omega'_4$ in one point. Therefore, if $r_1\subset \partial N(K_1)$ is the nonintegral slope $*/p_1$ of the boundary circle $A_1\cap\partial N(K_1)$ (see Fig.~\ref{k10b}, left) constructed in \S\ref{simple}, then we have 
\begin{align*}
X_{K_1}(r_1)&\approx R_{2,1}(\omega_1\sqcup\omega'_1)=
R_{2,1}(J)(\omega_1\sqcup\omega'_1)\\
&=\left( R_{2,4}\cup R_{4,5}\cup R_{5,1} \right)(J)(\omega_1\sqcup\omega'_1)
\\
&\approx R_{4,5}(J)\cup V_3\cup V_5\\
&=\mS^2(p_4,1,1)
\end{align*}
and so $X_{K_1}(r_1)$ is either $\mS^3$, $\mS^1\times\mS^2$ or a lens space. In fact, as $r_1$ is a nonintegral slope, for homological reasons $X_{K_1}(r_1)$ cannot be $\mS^1\times\mS^2$.
Therefore by \cite{cgls} and \cite{gordonlu1}
$K_1$ is a trivial or torus knot.

\medskip
\item[(E4)]
{\it The circles $\omega_1,\omega'_2$ are Seifert circles in $R_{3,1}$.}
\\
Since the region $R_{4,2}$ in Fig.~\ref{k10b}, left, is a handlebody, the circle $\omega_1$ is primitive in $R_{4,1}$ by \cite[Lemma 6.8(1)(b)]{valdez14}; hence $\omega_1$ is a Seifert circle in $R_{3,1}$ by \cite[Lemma 6.8(1)(d)]{valdez14}.
\\
A similar argument applied to the handlebody region $R_{2',5}$ of Fig.~\ref{k10b}, right, shows that $\omega'_2$ is also a Seifert circle in $R_{3,1}$.
\een

\subsection{The maximal pair $(R_{3,1},K)$.}\label{RK}
For the rest of this section we assume that the regular neighborhood $N(K)\subset\mS^3$ has been retracted radially onto $K$, so that the circles $J$ and $\partial T_i$ become identified with $K$. Thus we use the notation $(R_{i,i+1},K)$ for the pairs $(R_{i,i+1},J)$

We construct a complete disk system for the maximal pair $(R_{3,1},K)$ as follows.
First observe that by Lemma~\ref{ann-powers}(6) the maximal pair $(R_{3,1},K)$ is not annular, hence by \cite[Lemma 3.4]{valdez14} there is a disk $E\subset R_{3,1}$, unique up to isotopy, which separates the power circles $\omega_3,\omega'_5\subset R_{3,1}$.
Thus $R_{3,1}|E$ consists of two solid tori $V_3,V_5$ with the power circles $\omega_3\subset\partial V_3$ and $\omega'_5\subset\partial V_3$ intersecting meridian disks $D_3\subset V_3,D_5\subset V_5$ in $p_3=2, p_5=2$ points, respectively. 
Using the method outlined in \cite[Section 7.3 and Lemma 7.6(2)]{valdez14},  the 6-tuple $(\partial E,K,\omega_1,\omega'_2,\omega_3,\omega'_5)$ is homeomorphic to the 6-tuple in Fig.~\ref{k11}, top, or to the 6-tuple obtained by reflecting $\partial R_{3,1}$ across the plane of the page. The two 6-tuples are then homeomorphic and hence we only consider the case of Fig.~\ref{k11}, top.

The meridian circles $\partial D_3,\partial D_5\subset\partial R_{3,1}$ can then be constructed as homological sums of the form
\[
\partial D_3=q_3\omega_3+a_3\beta_3
\quad\text{and}\quad
\partial D_5=q_5\omega_5'+a_5\beta_5
\]
where $\gcd(a_3,q_3)=1=\gcd(a_5,q_5)$ and $\beta_3\subset\partial V_3$ and $\beta_5\subset\partial V_5$ are the two circles indicated in Fig.~\ref{k11}, bottom, such that 
\[
\Delta(\beta_3,\omega_3)=1=\Delta(\beta_5,\omega'_5)
\]

The values of $a_3,a_5,q_3,q_5$ can be found by performing some  computations in the fundamental group of $R_{3,1}$. To this end we set $x=\partial D_3$, $y=\partial D_5$, and $\pi_1(R_{3,1})=\grp{x,y \ | \  -}$ relative to some base point. Thus a circle $c\subset\partial R_{3,1}$ which intersects $x\sqcup y$ transversely is represented by an unreduced word $c(x,y)\in \pi_1(R_{3,1})=\grp{x,y \ | \ -}$ obtained by reading the consecutive signed intersections of $c$ with $x$ and $y$ without introducing any cancellations, relative to a base point in $c\setminus(x\sqcup y)$. Notice that if the word $c(x,y)$ is cyclically reduced then $c$ intersects $x\sqcup y$ minimally, but not conversely.

For convenience we use the notation $X=x^{-1}$ and $Y=y^{-1}$ to denote the inverses of $x$ and $y$ in $\grp{x,y \ | \ -}$.

\begin{figure}
\Figw{4.7in}{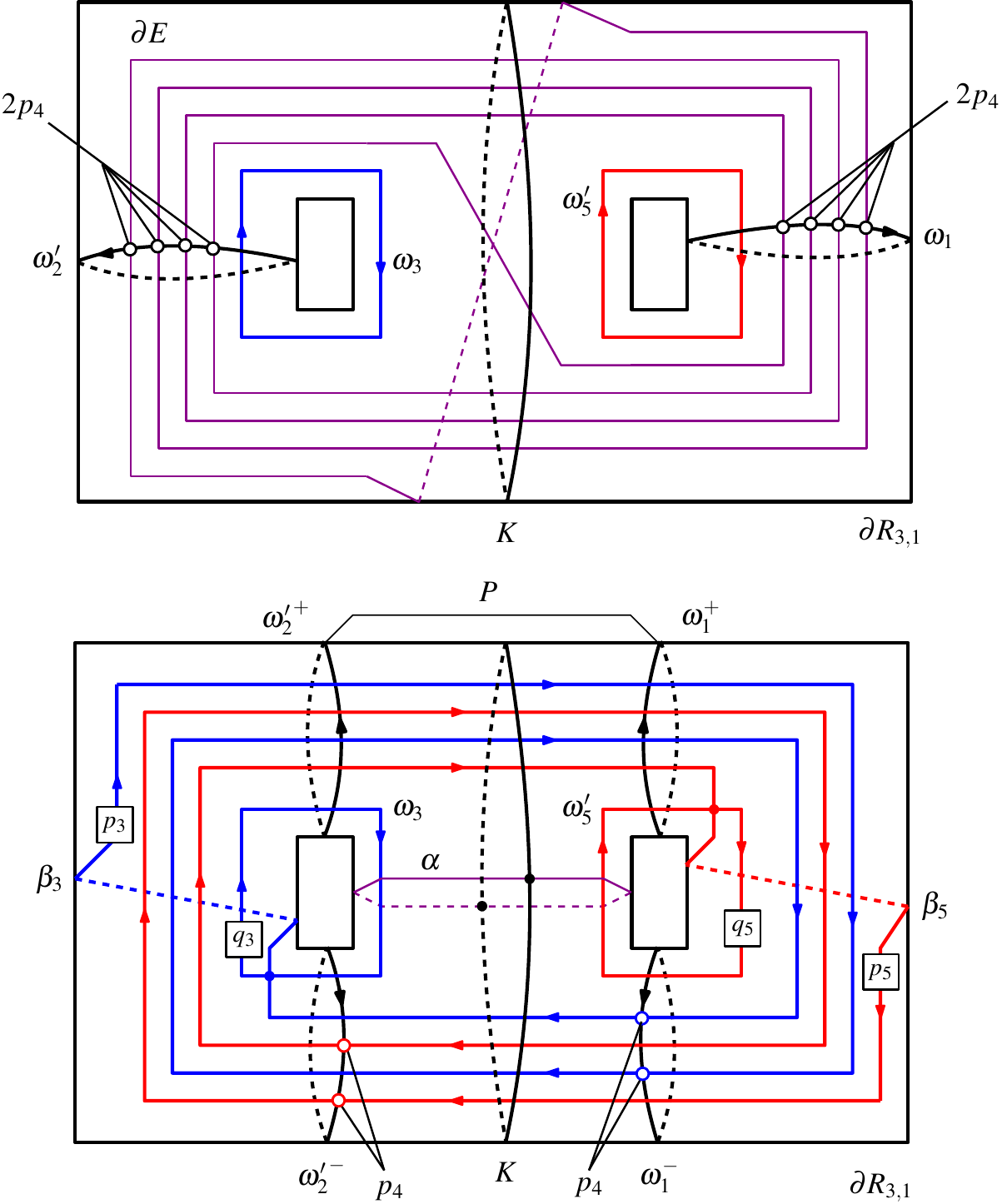}{Construction of the circles $\partial D_1=q_1\omega_1+p_1\beta_1$ and $\partial D_3=q_3\omega_3'+p_3\beta_3$ in $\partial R_{3,1}$ (with $p_4=2$).}{k11}
\end{figure}

The following relations now follow from Fig.~\ref{k11}, bottom (relative to some base point and intersection signs scheme):
\begin{itemize}
\item[(E5)(a)]
$\omega_3(x,y)=x^{a_3}$ and $\omega'_5(x,y)=y^{a_5}$: since $\omega_3$ and $\omega'_5$ are $p_3$ and $p_4$ power circles in $R_{3,1}$, respectively, we may choose $a_3=p_3=2$ and $a_5=p_5=2$;

\medskip
\item[(E5)(b)]
$\omega_1(x,y)=(y^{p_5}x^{p_3})^{p_4}y^{q_5}$ and 
$\omega'_2(x,y)=(y^{p_5}x^{p_3})^{p_4}x^{q_3}$.\\
Observe that $\omega_1(x,y)=W(x^{p_3},y)$ where $W(x,y)=(y^{p_5}x)^{p_4}y^{q_5}$. As $\omega_1$ is a Seifert circle in $R_{3,1}$,
by the argument of \cite[Lemma 7.11]{valdez14} we must have that $W(x,y)$ is a primitive word in the free group $\grp{x,y \ | \ -}$ and hence that $q_5=\pm 1$. In a similar way we must have that $q_3=\pm 1$.
\end{itemize}
For each circle $\omega_3,\omega'_5, \beta_3,\beta_5$ in Fig.~\ref{k11}, bottom, the coefficient in the box on top of the circle represents the number of copies of that circle used in the homological sum construction of a given meridian circle $x=\partial D_3$ and $y=\partial D_5$.

We summarize the above facts in the next result.

\begin{lem}
The 7-tuple $(\partial R_{3,1},\partial E,K,\omega_1,\omega'_2,\omega_3,\omega'_5)$ is homeomorphic to the 7-tuple in Fig.~\ref{k11}, top (where we use $p_4=2$ for simplicity). Moreover the circles $\partial D_3,\partial D_5\subset\partial R_{3,1}$ can be represented as the homological sums
\[
\partial D_3=q_3\omega_3+p_3\beta_3\quad\text{and}\quad\partial D_5=q_5\omega_5'+p_5\beta_5
\]
where $\beta_3,\beta_5$ are the circles indicated in Fig.~\ref{k11}, bottom, with $p_3=2=p_5$ and $q_3,q_5=\pm1$.
\hfill\qed
\end{lem}

Since the circles $\omega_1,\omega'_2\subset\partial R_{1,3}$ cobound an annulus in $R_{1,3}$, by \cite[Lemma 3.4]{valdez14} the surface $\partial R_{1,3}\setminus(\omega_1\sqcup\omega'_2)$ compresses along a  nonseparating disk $D\subset R_{1,3}$ (unique up to isotopy), and necessarily $R_{1,3}|D$ is a solid torus neighborhood of the knot $K_1$. Therefore we may set 
\[
X_{K_1}=R_{3,1}(\partial D)
\] 

By (E3) the core knot $K_1\subset\mS^3$ of the simple pair $(R_{1,2},K)$ is either trivial or a torus knot. Therefore $X_{K_1}=R_{3,1}(\partial D)$ is either a solid torus or a Seifert fiber space of the form $\mD^2(*,*)$ for $*\geq 2$, or, equivalently, 

\begin{enumerate}
\item[(E6)]
{\it the circle $\partial D$ is either primitive or Seifert in $R_{3,1}$. }
\end{enumerate}

The next two results will be useful in restricting the possible embeddings of the circle $\partial D$ in $\partial R_{3,1}$.

\begin{lem}\label{string}
Let $\omega\subset\partial R_{3,1}$ be any circle that intersects $x\sqcup y\subset\partial R_{3,1}$ minimally. If some cyclic reorderings of the word $\omega(x,y)\in\pi_1(R_{3,1})$ contain strings of the form $x^{a}$ and $y^{b}$ for some integers $|a|,|b|\geq 2$ then $\omega(x,y)$ is cyclically reduced and $\omega$ is neither a primitive nor a power circle in $R_{3,1}$.
\end{lem}

\begin{proof}
Let $Q$ be the 4-punctured 2-sphere $\partial R_{3,1}\setminus\intr\,N(x\sqcup y)$ with boundary components the circles $x^+,x^-$ and $y^+,y^-$ corresponding to the two sides of $x$ and $y$ in $\partial R_{3,1}$, respectively. Since $\omega$ intersects $x\sqcup y$ minimally, $\omega\cap Q$ consists of a collection of properly embedded arcs none of which is parallel into $\partial Q$.

\begin{figure}
\Figw{3in}{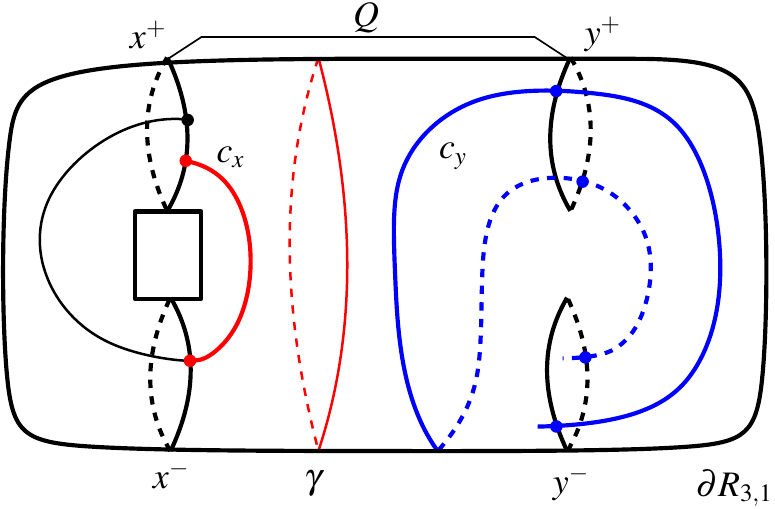}{The arc components $c_x,c_y\subset\omega\cap Q$.}{k29}
\end{figure}

By hypothesis some cyclic reordering of the word $\omega(x,y)$ contains a string of the form $x^a$ for some integer $|a|\geq 2$ and so there is an arc component $c_x\subset Q$ with one endpoint in $x^+$ and the other in $x^-$. Thus the circle $\gamma=\fr\,N(x^+\cup c_x\cup x^-)\subset Q$ separates $x^+\sqcup x^-$ from $y^+\sqcup y^-$.

Suppose that some cyclic reordering of the word $\omega(x,y)$ has a canceling string of the form $yY$ or $Yy$. Then there is an arc component $c_y\subset\omega\cap Q$ with both endpoints on, say, the boundary component $y^+\subset\partial Q$. As $c_y$ is disjoint from $c_x$, it is also disjoint from $\gamma$ and so $c_y$ separates $x^+\sqcup x^-$ from $y^-$ in $Q$ (see Fig.~\ref{k29}), which is impossible since $\omega$ is a closed circle in $\partial R_{3,1}$.

Therefore no cyclic reordering of the the word $\omega(x,y)$ has canceling strings of the form $yY, Yy$, and in a similar way neither of the form $xX,Xx$, whence it is a cyclically reduced word. 
Since the word $\omega(x,y)$ contains strings of the form $x^{a}$ and $y^{b}$ for some $|a|,|b|\geq 2$, by \S\ref{prim-pwr} the circle $\omega\subset\partial R_{3,1}$ is neither primitive nor a power in $R_{3,1}$.
\end{proof}

We now take parallel copies $\omega_1^+,\omega_1^-$ of $\omega_1$ and ${\omega'_2}^+,{\omega'_2}^-$ of $\omega' _2$ in $\partial R_{3,1}$ as shown in Fig.~\ref{k11}, bottom, and let $P$ be the 4-punctured 2-sphere in $\partial R_{3,1}$ cobounded by the 4 circles
$\omega_1^+,\omega_1^-,{\omega'_2}^+,{\omega'_2}^-$. 

For each pair of values of $q_3,q_5=\pm 1$ let $\Gamma\subset P$ denote the  graph $P\cap(\partial D_3\sqcup\partial D_5)$ and $\ov{\Gamma}\subset P$ the reduced graph obtained by amalgamating each collection of parallel edges of $\Gamma$ into a single edge.
By minimality of $|(x\sqcup y)\cap\partial D_3|$ and $|(x\sqcup y)\cap\partial D_5|$ no edge of $\Gamma$ or $\ov{\Gamma}$ is parallel into $\partial P$, that is, the graphs $\Gamma,\ov{\Gamma}$ are essential.

\begin{figure}
\Figw{4.5in}{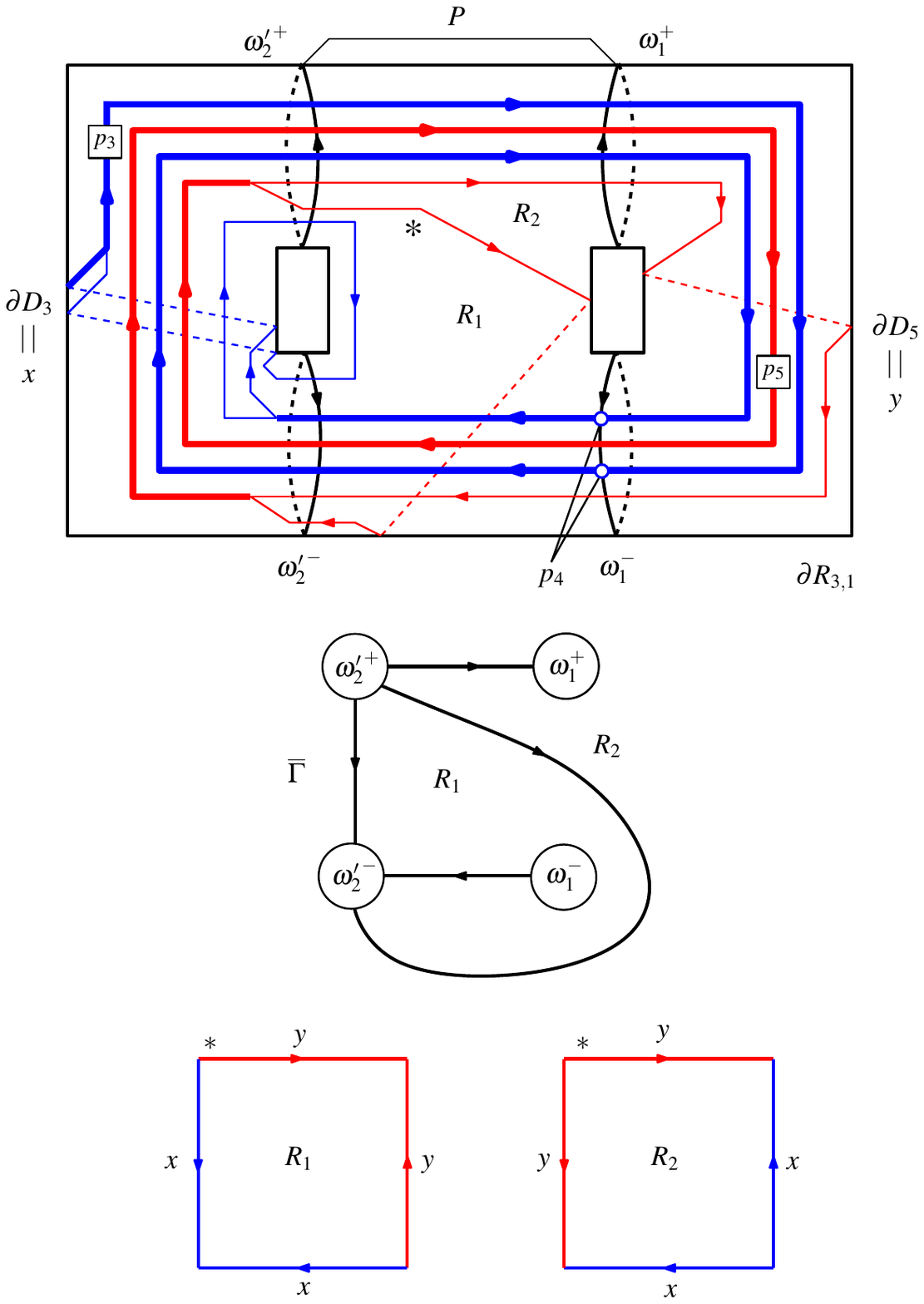}{The circles $\partial D_3,\partial D_5\subset\partial R_{3,1}$ for $q_3=+1$ and $q_5=-1$.}{k12}
\end{figure}

\begin{lem}\label{redd}
If a circle $c\subset P$ intersects the reduced graph $\ov{\Gamma}$ minimally then the word $c(x,y)$ is cyclically reduced. In particular, any circle in $P$ which is primitive in $R_{3,1}$ is isotopic to the circle $\alpha\subset P$ in Fig.~\ref{k11}, bottom.
\end{lem}

\begin{proof}
We consider the case $q_3=+1$ and $q_5=-1$, the other cases being similar. Fig.~\ref{k12}, top, shows the graph $\Gamma\subset P$ where the thicker lines represent 2 amalgamated parallel edges of one of the circles $\partial D_3,\partial D_5$ (corresponding to the values $p_3=2=p_5$), while the thinner lines represent single arcs.

The reduced graph $\ov{\Gamma}\subset P$ is shown in Fig.~\ref{k12}, middle, where each amalgamated edge shows the common orientation of its components. The set of faces of $\ov{\Gamma}$ consists of the two 4-sided disk faces $R_1,R_2$ in Fig.~\ref{k12}, bottom, where each edge of $R_i$ is labeled and oriented as the corresponding edge in the unreduced graph $\Gamma\subset P$. 

Let $c\subset P$ be any circle which intersects $\ov{\Gamma}\subset P$ minimally. Then the sink/source pattern of the oriented edges around the faces $R_1,R_2$ guarantee that the word $c(x,y)$ does not contain any of the canceling pairs $xX,Xx,yY,Yy$ and hence that it is cyclically reduced.

Notice that if $c$ intersects any of the horizontal edges of $\ov{\Gamma}$ then the word $c(x,y)$ contains one of the strings $x^2y^2$, $y^2x^2$ or their inverses and hence by \S\ref{prim-pwr} the word $c(x,y)$ cannot be primitive in the free group $\pi_1(R_{3,1})=\grp{x,y \ | \ -}$.
Therefore if $c\subset P$ is a primitive circle in $R_{3,1}$ then $c$ can be isotoped in $P$ so as to be disjoint from the horizontal edges of the graph $\ov{\Gamma}\subset P$. As the horizontal edges of $\ov{\Gamma}$ cut $P$ into an annulus with core the circle $\alpha\subset P$, it follows that $c$ must be isotopic to $\alpha$ in $P$, hence in $\partial R_{3,1}$.
\end{proof}

\medskip

We now complete the proof of Proposition~\ref{maxcase2}.

\begin{proof}[Proof of Proposition~\ref{maxcase2}]
By (E6) the circle $\partial D\subset\partial R_{3,1}$ is either a primitive or a Seifert circle in $R_{3,1}$. We consider two cases and arrive at a contradiction in each.

\medskip

{\bf Case 1:} $\partial D\subset\partial R_{3,1}$ is a primitive circle in $R_{3,1}$ ($K_1\subset\mS^3$ is a trivial knot). 
\\
As  the circle $\partial D\subset\partial R_{3,1}$ is disjoint from the circles $\omega_1\sqcup\omega'_2\subset\partial R_{3,1}$ it can be isotoped so as to lie in $P$ and so by Lemma~\ref{redd} it must be isotopic in $P$ to the circle $\alpha$ in Fig.~\ref{k11}, top. Since $|\partial D\cap K|=|\alpha\cap K|=2$ and $D\subset R_{1,3}$ by \S\ref{simple}
the pair $(R_{1,3},K)$ is simple, hence minimal, which is not the case. Therefore this case does not occur.

\medskip
{\bf Case 2:} $\partial D\subset P$ is a Seifert circle in $R_{3,1}$ ($K_1\subset\mS^3$ is a nontrivial torus knot). 
\\
By \cite[Lemma 6.7]{valdez14} there is a circle $h\subset\partial R_{3,1}\setminus\partial D$ which is a power circle in $R_{3,1}$.

By Lemma~\ref{redd} isotopying $\partial D$ in $P$ so as intersect $\ov{\Gamma}$ minimally yields a cyclically reduced word $\partial D(x,y)$. Once $\partial D$ has been isotoped, isotopying $h$ in $\partial R_{3,1}\setminus\partial D$ so as to intersect $x\sqcup y=\partial D_3\sqcup\partial D_5$ minimally produces the minimal intersection in $\partial R_{3,1}$ between $h$ and $x\sqcup y$.

Now, by \cite[Lemma 6.7]{valdez14} the circle $h\subset R_{3,1}(\partial D)=X_{K_1}$ is a fiber of the Seifert fiber space $X_{K_1}=\mD^2(*,*)$. Since by (E3)
$X_{K_1}(r_1)\approx R_{3,1}(\partial D)(\omega_1)\approx R_{3,1}(\partial D)(\omega'_2)$ is either $\mS^3$ or a lens space, it follows that 
$\Delta(h,\omega_1)=1=\Delta(h,\omega'_2)$ in $\partial X_{K_1}$ and hence that $h$ intersects each circle $\omega_1,\omega'_2$ nontrivially in $\partial R_{3,1}$.

Therefore there is an arc component $h'$ of $h\cap P$ with one endpoint in 
$\omega_1^+\sqcup\omega_1^-$ and the other endpoint in ${\omega'_2}^+\sqcup{\omega'_2}^-$. 

If $h'$ intersects transversely at least one of the horizontal edges in the reduced graph $\ov{\Gamma}\subset P$ then the word $h'(x,y)$, and hence $h(x,y)$, contains one of the strings $x^2y^2,y^2x^2$ or their inverses, contradicting Lemma~\ref{string} since $h$ is a power circle in $R_{3,1}$. 
So if $h'$ has endpoints on $\omega_1^+\sqcup{\omega'_2}^+$ or $\omega_1^-\sqcup{\omega'_2}^-$ then  $h'$ can be isotoped so as to be parallel to one of the horizontal edges of $\ov{\Gamma}\subset P$, which implies that $\partial D$, being disjoint from $h'$, is isotopic in $P$ to the primitive circle $\alpha\subset P$, contradicting the hypothesis that $\partial D$ is a Seifert circle in $R_{3,1}$. 

Therefore the arc $h'\subset P$ must have endpoints on, say, $\omega_1^+$ and ${\omega'_2}^-$. As $\partial D$ and $h'$ are disjoint, in the first integral  homology group 
\[
H_1(R_{3,1})=H_1(R_{3,1};\mZ)=x\mZ\oplus y\mZ
\]
the circle $\partial D$ is the homological sum 
$\omega_1^+ 
+_{h'} 
{\omega'_2}^-$ 
of $\omega_1^+$ and ${\omega'_2}^-$ along the arc $h'\subset P$. Using the orientations for $\omega_1^+$ and ${\omega'_2}^-$ in Fig.~\ref{k11}, top,
and the relations $\omega_1^+(x,y)=(y^2x^2)^{p_4}y^{q_5}$, ${\omega'_2}^-(x,y)=(x^2y^2)^{p_4}x^{q_3}$ (up to conjugation) found in (E5)(b) we obtain, in $H_1(R_{3,1})$,
\begin{align*}
\omega_1^+&=2p_4x+(2p_4+q_5)y
\quad\text{and}\quad 
{\omega'_2}^-=(2p_4+q_3)x+2p_4y\\
\implies 
\partial D&
=\omega_1^+ + {\omega'_2}^-
= (4p_4+q_3)x+(4p_4+q_5)y
\end{align*}
On the other hand, as $R_{3,1}(\partial D)$ is a knot exterior in $\mS^3$ and hence a homology solid torus, the circle $\partial D$ must be primitive in the abelian group $H_1(R_{3,1})=x\mZ\oplus y\mZ$, so we must have 
\[
1=\gcd(4p_4+q_3,4p_4+q_5)=\gcd(4p_4+q_3,q_3-q_5)=\gcd(4p_4+q_3,1-q_3q_5)
\]
As $q_3,q_5\in\{\pm 1\}$ and $p_4\geq 2$ this implies that 
\[
q_3q_5=-1
\]
and hence that
\begin{align*}
H_1(X_{K_1}(r_1))&=H_1\left(R_{3,1}(\omega_1^+\sqcup{\omega'}_2^+)\right)\\
&=x\mZ\oplus y\mZ/\grp{2p_4x+(2p_4+q_5)y, (2p_4+q_3)x+2p_4y}\\
&=\{0\}\leftarrow\text{since }
\det\begin{bmatrix}
2p_4 & 2p_4+q_5\\
2p_4+q_3 & 2p_4
\end{bmatrix}=1
\end{align*}
Since by (E3) the manifold $X_{K_1}(r_1)$ is $\mS^3$ or a lens space, it follows that $X_{K_1}(r_1)=\mS^3$. But then, as $r_1$ is a nonintegral slope of the form $a_1/p_1$, $p_1\geq 2$, by \cite{cgls} $K_1$ is a trivial knot, contradicting the fact that $K_1$ is a nontrivial torus knot. 

Therefore Case 2 does not occur and so the simplicial collection $\mT\subset X_K$ does not produce any exchange regions. By Lemma~\ref{new01c}(4) any Seifert torus in $X_K$ is then isotopic to some component of $\mT$ (see the proof of Proposition~\ref{mainp} for more details), hence the collection $\mT\subset X_K$ is unique up to isotopy.
\end{proof}

\section{Simplicial collections $\mT\subset X_K$ with minimal pairs
$(R_{i,i+1},J)$}\label{maxcoll}

In this section we show that for a genus one hyperbolic knot $K\subset\mS^3$ there are, up to isotopy, at most two  maximal simplicial collections of Seifert tori in $X_K$. This restricted number of such collections is the result of the interplay between the restrictions on the complementary regions of a maximal simplicial collection $\mT\subset X_K$ in Lemma~\ref{ann-powers}, the small size of of an annular pair $(R_{i,j},J)$ of index $\geq 2$ found in Lemma~\ref{annmin}, and the bound $|\mT|\leq 5$.

\begin{lem}\label{onex}
Any simplicial collection $\mT$ of Seifert tori in $X_K$ with minimal pairs $(R_{i,i+1},J)$ has at most one exchange region, and if so then $2\leq|\mT|\leq 4$.
\end{lem}

\begin{proof}
Set $|\mT|=N$ where $1\leq N\leq 5$ by \cite{valdez14}. Clearly there are no exchange regions when $N=1$.
If $N=2$ then by Lemma~\ref{ann-powers}(4) only one of $R_{1,1}$ or $R_{2,2}$ can be an exchange region, while in the case $N=5$ there are no exchange regions by Proposition~\ref{maxcase2}.

Therefore we may assume that $N=3,4$ in which case any two exchange regions must intersect. Arguing by contradiction, we only need to consider the following two cases.

\begin{figure}
\Figw{5.5in}{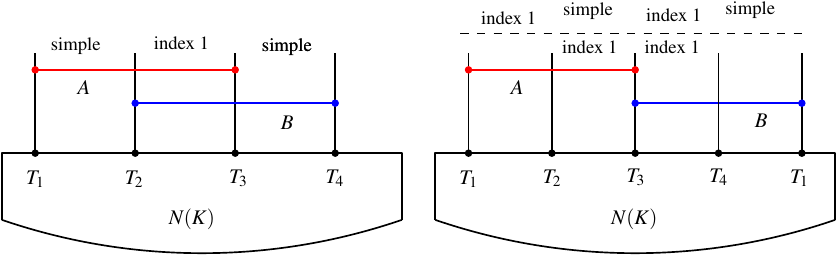}{Intersecting exchange regions in $X_K$.}{k27}
\end{figure}

\medskip
{\bf Case 1:} {\it $R_{1,3}$ and $R_{2,4}$ are exchange regions (with $T_1=T_4$ allowed).}\\
The situation is represented in Fig.~\ref{k27}, left.
By the exchange trick of \S\ref{trick} we may assume that the pair $(R_{2,3},J)$ is annular of index 1 while the pairs $(R_{1,2},J)$ and $(R_{3,4},J)$ are simple of indices $\geq 2$. So if $A$ and $B$ are spanning annuli for $R_{1,3}$ and $R_{2,4}$, respectively, then by Lemma~\ref{ann-powers}(6) we may assume that
$A\cap R_{2,3}$ and $B\cap R_{2,3}$ are spanning annuli of index 1 in $R_{2,3}$, hence isotopic by Lemma~\ref{annp}. 
Thus the circles $A\cap T_2$ and $B\cap T_2$ have the same slope on $T_2$, and similarly the circles $A\cap T_3$ and $B\cap T_3$ have the same slope on $T_3$. This implies that the boundary components of the spanning annulus $A\cap R_{2,3}\subset R_{2,3}$ have companion annuli in $R_{1,2}$ and $R_{3,4}$, contradicting Lemma~\ref{ann-powers}(4).

\medskip
{\bf Case 2:} {\it $N=4$ and $R_{1,3}$ and $R_{3,1}$ are exchange regions.}\\
Let $A$ and $B$ be spanning annuli for $R_{1,3}$ and $R_{3,1}$, respectively, and let $\Delta_3=\Delta(A\cap T_3,B\cap T_3)$.

Suppose that $\Delta_3=0$. By the exchange trick and Lemma~\ref{ann-powers}(6) we may assume that $A\cap R_{2,3}$ and $B\cap R_{3,4}$ are spanning annuli of index 1 in $R_{2,3}$ and $R_{3,4}$, respectively, as shown in Fig.~\ref{k27}, right, below the dashed line.

Therefore the annulus $A\cap R_{2,3}$ can be isotoped in $R_{2,3}$ so that $A\cap R_{2,3}=B\cap R_{3,4}$, in which case their union becomes an index 1 spanning annulus for the region $R_{2,4}$, contradicting Lemma~\ref{annmin}(1).

Therefore $\Delta_3\neq 0$ and, by the exchange trick, this time we may assume that the pairs 
$(R_{1,2},J)$ and $(R_{3,4},J)$ are annular of index 1 while the pairs 
$(R_{2,3},J)$ and $(R_{4,5},J)$ are simple of index $\geq 2$, as shown in Fig.~\ref{k27}, right, above the dashed line.

If $R_{2,4}$ is a handlebody then $R_{3,4}$ is a handlebody by Lemma~\ref{genprop}(3) and so, being of index 1, $(R_{3,4},J)$ is a primitive pair with spanning annulus $B\cap R_{3,4}$. 
On the other hand $A\cap R_{2,3}$ is a spanning annulus for the simple pair $(R_{2,3},J)$. As $\Delta_3\neq 0$, the slopes of the spanning annuli $A\cap R_{2,3}$ and $B\cap R_{3,4}$ disagree on $T_3$, contradicting Lemma~\ref{ann-powers}(7). Therefore $R_{2,4}$ is not a handlebody and hence the region $R_{4,2}$ is a handlebody by Lemma~\ref{genprop} (P2).

However, as $N=4$, by the argument above we also have that 
$\Delta_1=\Delta(A\cap T_1,B\cap T_1)\neq 0$ and hence that the region $R_{4,2}$ is not a handlebody, a contradiction.

Therefore there cannot be two exchange regions for the collection $\mT$ in $X_K$.
\end{proof}

\begin{prop}\label{mainp}
Let $K$ be a hyperbolic knot in $\mS^3$.
\begin{enumerate}
\item
A simplicial collection $\mT=\sqcup_i T_i\subset X_K$ of Seifert tori is maximal iff its complementary pairs $(R_{i,i+1},J)$ are all minimal. In particular, any simplicial collection of Seifert tori in $X_K$ can be extended to a maximal such collection by suitably adding $J$-tori to each nonminimal pair of the collection.

\item
Up to isotopy, there are at most two maximal simplicial collections of  Seifert tori in $X_K$. 
Specifically, if $\mT\subset X_K$ is a maximal such collection then either
\begin{enumerate}
\item
$\mT$ has no exchange region and $\mT$ is the unique maximal simplicial collection of Seifert tori in $X_K$; any Seifert torus in $X_K$ is isotopic to some component of $\mT$,

\item
$2\leq|\mT|\leq 4$ and 
$\mT$ has a unique exchange region $R_{i-1,i+1}$ with induced tori $T_{i+1}, T'_{i+1}\subset R_{i-1,i+1}$, and $\mT$ and $(\mT\setminus T_i)\sqcup T'_{i+1}$ are the unique maximal simplicial collections of Seifert tori in $X_K$; any Seifert torus in $X_K$ is isotopic to some component of $\mT$ or to $T'_{i+1}$.
\end{enumerate}
\end{enumerate}
\end{prop}

\begin{proof}
Let $\mT=T_1\sqcup\cdots\sqcup T_N\subset X_K$ be a simplicial collection of  Seifert tori such that each pair $(R_{j,j+1},J)$ is minimal, and let $\mS\subset X_K$ be any simplicial collection of  Seifert tori. 

Isotope $\mS$ in $X_K$ so as to intersect $\mT$ minimally with 
$\partial\mS\cap\partial\mT=\emptyset$. By the argument in Lemma~\ref{new01c}(1) it follows that each component of $\mS$ is either disjoint from $\mT$ and hence parallel to some component of $\mT$, or intersects $\mT$ minimally in $X_K$.

Suppose that $S\subset\mS$ is a Seifert torus which is not isotopic to any component of $\mT$. By Lemma~\ref{new01c} there is a component $T_j\subset\mT$ such that 
\begin{enumerate}
\item[(i)]
$|S\cap T_j|=2$, 

\item[(ii)]
the closures of the components of $S\setminus \mT$ consist of a pants $P$ and a companion annulus $A$ with $P\cap T_j=P\cap\mT=A\cap T_j$ that lie on opposite sides of $T_j$,

\item[(iii)]
there is a Seifert torus $T\subset X_K\setminus(P\cup\mT)$ which is not parallel to $T_j$ in $X_K$, 

\item[(iv)]
if $R\subset X_K$ is the region cobounded by $T$ and $T_j$ that contains $P$ then the pair $(R,J)$ is annular of index 1 with spanning annulus $A_R\subset R$ having the same boundary slope on $T_j$ as $A$, 
and $R\subset R_{j-1,j}$ or $R\subset R_{j,j+1}$.
\end{enumerate}
Since each pair $(R_{k,k+1},J)$ is minimal, by (iii) we must have $N\geq 2$ and by (iv) and Proposition~\ref{annmin}(1) we may assume that $R=R_{j-1,j}$, in which case by (ii) we have $A\subset R_{j,j-1}$. 

If $A\cap T_{j+1}\neq\emptyset$ then some component of $A\cap R_{j,j+1}$ is a spanning annulus of $R_{j,j+1}$, of the same boundary slope on $T_j$ as $A_R$ by (iv), and some component of $A\cap R_{j+1,j-1}$ is a companion annulus. Therefore by Lemma~\ref{ann-powers}(6) the pair $(R_{j,j+1},J)$ is annular of index  1 and so the pair $(R_{j-1,j+1},J)$ is also annular of index 1, contradicting Proposition~\ref{annmin}(1).

It follows that the companion annulus $A$ lies in $R_{j,j+1}$ and hence that the minimal pair $(R_{j,j+1},J)$ is simple by Lemma~\ref{annp}(5)(b). Therefore the pair $(R_{j-1,j+1},J)$ is annular of index $\geq 2$ and hence that $R_{j-1,j+1}$ is the unique exchange region of $\mT$. 
Moreover, since $S\subset R_{j-1,j+1}$ and
$T_{j}\subset R_{j-1,j+1}$ is the $J$-torus induced by $T_{j+1}$, by Proposition~\ref{annmin}(2) $S$ is isotopic in $R_{j-1,j+1}$ to the $J$-torus $T'_j\subset R_{j-1,j+1}$ induced by $T_{j-1}$ and $2\leq |\mT|\leq 4$.

Therefore the collection $\mS$ is isotopic to some subset of one of the collections $\mT$ or $(\mT\setminus T_j)\sqcup T'_j$, both of which have size $|\mT|$, and so the collection $\mT$ is maximal.

That a maximal simplicial collection produces minimal pairs follows by definition of maximality. Therefore (1) holds, and now (2) holds by the above argument.
\end{proof}

\begin{proof}[Proof of Theorem~\ref{main1} parts (1) and (2):]
Set $d=\dim\,MS(K)$.
That $0\leq d\leq 4$ follows from the bound $|\mT|\leq 5$ given in \cite{valdez14} for any maximal simplicial collection of Seifert tori $\mT\subset X_K$. Hence part (1) holds.

Each $d$-dimensional simplex of $MS(K)$ corresponds to the isotopy class of some such maximal collection $\mT\subset X_K$, and by Proposition~\ref{mainp}(2) any two such maximal collections differ up to isotopy by at most one component. Therefore $MS(K)$ consists of at most two $d$-simplices, and two $d$-dimensional simplices in $MS(K)$ intersect in a common $(d-1)$-face. Hence part (2) holds.
\end{proof}

\section{Examples of hyperbolic knots in $\mS^3$.}\label{examples}

By Propositions~\ref{maxcase2} and \ref{mainp}, a maximal simplicial collection $\mT\subset X_K$ of size $|\mT|=5$ produces no exchange regions and is therefore unique up to isotopy.
In this section we construct examples of hyperbolic knots $K\subset\mS^3$ with maximal simplicial collections of Seifert tori $\mT\subset X_K$ of sizes $2\leq|\mT|\leq 4$ that produce exchange regions and hence $MS(K)$ consists of two top-dimensional simplices.

One example of such a knot $K$ with a collection $\mT\subset X_K$ having an exchange region was constructed in \cite[\S6]{sakuma3}. In that example it is proved that there are nonisotopic Seifert tori in $X_K$ that intersect nontrivially and hence the diameter of $MS(K)$ must be 2; thus the presence of an exchange region for $\mT$ is inferred from Proposition~\ref{mainp}(2). 
We follow a different strategy in the construction of examples along with the results obtained so far which allows us to determine both the size of their maximal simplicial collection of Seifert tori and the Kakimizu complex of the constructed knots.

\subsection{Detecting primitive pairs and exchange pairs}

In the case of handlebody pairs by \S\ref{oper}(3)(4) an exchange pair can be thought of as an extension of a primitive pair by a simple pair. Both simple and exchange pairs are annular pairs of index $\geq 2$, and the next result will be useful in distinguishing these types of pairs form each other.

\begin{lem}\label{dis}
Let $(H,J)$ be a handlebody pair with $\partial H=T_1\cup_J T_2$, $\omega_1\subset T_1$ and $\omega_2\subset T_2$ coannular circles in $H$, and $\gamma\subset T_1$ a circle with $\Delta(\omega_1,\gamma)=1$. Then the surface $\partial H\setminus(\omega_1\sqcup\omega_2)$ compresses in $H$ along a nonseparating disk $D\subset H$, unique up to isotopy, and the following hold:
\begin{enumerate}
\item
$\omega_1$ and $\omega_2$ are both primitive in $H$ or both $p$-power circles for some $p\geq 2$,

\item
$\gamma\cdot\partial D=\pm1$,

\item
$|\gamma\cap\partial D|_{\text{min}}=1$
iff $(H,J)$ is a trivial or simple pair,

\item
if $|\gamma\cap\partial D|_{\text{min}}>1$
then $(H,J)$ is a primitive or an exchange pair if $\omega_1$ is a primitive or a power circle in $H$, respectively.
\end{enumerate}
\end{lem}

\begin{proof}
That the surface $\partial H\setminus(\omega_1\sqcup\omega_2)$ compresses in $H$ along a nonseparating disk $D\subset H$ which unique up to isotopy and part (1) follow from \cite[Lemma 3.4]{valdez14}.

Isotope $\partial D$ in $\partial H\setminus(\omega_1\sqcup\omega_2)$ so as to intersect $\gamma$ minimally.
As the circles $\omega_1\sqcup\omega_2\sqcup\partial D$ separate $\partial H$ into two pants the circle $\partial D$ is homologous in $\partial H$ to $\omega_1\sqcup\omega_2$, and since $\partial D\cap\omega_2=\emptyset$ we must have (up to some orientation scheme)
\[
\gamma\cdot\partial D=
\gamma\cdot\omega_1+\gamma\cdot\omega_2=\gamma\cdot\omega_1=\pm 1
\]
so (2) holds.

\medskip
Since $|\omega_1\cap\gamma|_{\text{min}}=1$ and $\omega_1\cup\gamma\subset T_1$, it follows that the circles $J$ and $\partial N(\omega_1\cup\gamma)\subset T_1$ are parallel in $T_1$ and hence that 
\[
|J\cap\partial D|_{\text{min}}=
2\cdot|\gamma\cap\partial D|_{\text{min}}
\]
By \S\ref{simple} the pair $(H,J)$ is trivial or simple iff  the disk $D$ intersects $J$ minimally in two points, hence iff $D$ intersects $\gamma$ minimally in one point so (3) holds.

For the case $|\gamma\cap\partial D|_{\text{min}}>1$ by (3) the pair $(H,J)$ is nontrivial and not simple, so if $\omega_1$ is a primitive circle in $H$ then $(H,J)$ is a primitive pair.
Otherwise by (1) $\omega_1$ is a $p\geq 2$ power circle in $H$ and so $(H,J)$ is an exchange pair by Lemma~\ref{koz1-b}(1) and Remark~\ref{remkoz}(1), hence (4) holds.
\end{proof}

In the construction of examples of knots $K\subset\mS^3$ we will make use of Lemma~\ref{Rpair} to justify that the regions $R_{i,j}$ involved form pairs $(R_{i,j},J)$ or $(R_{i,j},K)$ before knowing that the knot $K\subset\mS^3$ is hyperbolic. As this is automatically the case whenever the region $R_{i,j}$ is a handlebody, we will only invoke Lemma~\ref{Rpair} when $R_{i,j}$ is not a handlebody.

\subsection{Hyperbolic knots with $|\mT|=4$ and one exchange region.}
\label{82}
Suppose that $\mT\subset X_K$ is a maximal simplicial collection of size $|\mT|=4$ such that $R_{1,3}$ is an exchange region with $(R_{1,2},J)$ an index one annular pair. By Proposition~\ref{annmin}(1) the region $R_{1,2}$ must then be a handlebody, in which case the pair $(R_{1,2},J)$ is primitive and $R_{1,3}$ is a handlebody by \S\ref{oper}(3).
Arguments similar to those in \S\ref{61}--\ref{62} can be used to prove that the region $R_{3,1}$ must also be a handlebody and at least one of the pairs $(R_{3,4},J)$ or $(R_{4,1},J)$ be simple.

In this section we construct a family of knots $K=K(q,k,\ve)\subset\mS^3$ for integers $q\geq 1$, $k\in\mZ$, and $\ve=\pm1$, each of which bounds a maximal simplicial collection of 4 Seifert tori $\mT=T_1\sqcup T_2\sqcup T_3\sqcup T_4$ with an exchange region $R_{1,3}$ such that the regions $R_{1,3}$ and $R_{3,1}$ are handlebodies and both pairs $(R_{3,4},J)$ and $(R_{4,1},J)$ are simple. The collection $\mT\subset X_K$ is represented in Fig.~\ref{k15b},

\begin{figure}
\Figw{6in}{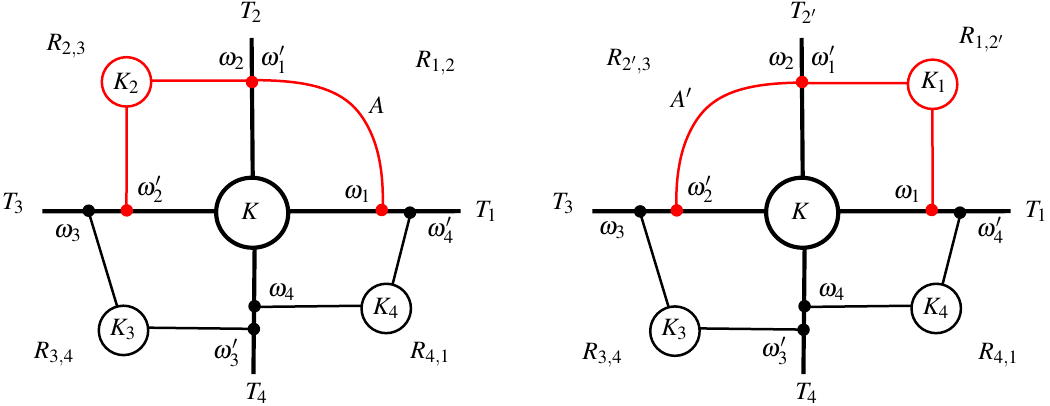}{The knot $K=K(q,k,\ve)\subset\mS^3$ with exchange pair $(R_{1,3},J)$.}{k15b}
\end{figure}

\medskip
(I) {\bf Construction of the circles $\omega_1,\omega'_4\subset T_1$ and $\omega'_2,\omega_3\subset T_3$ relative to the handlebody $R_{3,1}$.}
\\
Fig.~\ref{k18} shows the genus two handlebody $R_{3,1}$ with complete meridian system given by the disks $x\sqcup y\subset R_{3,1}$ and a disk $E\subset R_{3,1}$ separating $x$ and $y$. We identify $\pi_1(R_{3,1})$ with the free group
\[
\pi_1(R_{3,1})=\grp{x,y \ | \ -}
\]
relative to some base point. For $w_1,w_2\in \grp{x,y \ | \ -}$ we write $w_1\equiv w_2$ to indicate that the words differ by a cyclic permutation.

The circles $\omega_1,\omega'_4,\omega'_2$ in $\partial R_{3,1}$ are constructed as indicated in Fig.~\ref{k18}, along with an extra circle $u$. Notice that to the left of the separating disk $E$ all arcs in the figure are mutually parallel and intersect the disk $x$ minimally in one point. To the right of $E$ there are 3 disjoint arcs which intersect $y$ minimally in $q$, $q+\ve$ and $2q+\ve$ points, as well as the circle $\omega'_4$ which intersects $y$ minimally in $2q+\ve$ points (since it is disjoint from the arc that intersects $y$ in $2q+\ve$ points). 
Fig.~\ref{k18} corresponds to the case $(q,\ve)=(1,-1)$.

\begin{figure}
\Figw{5in}{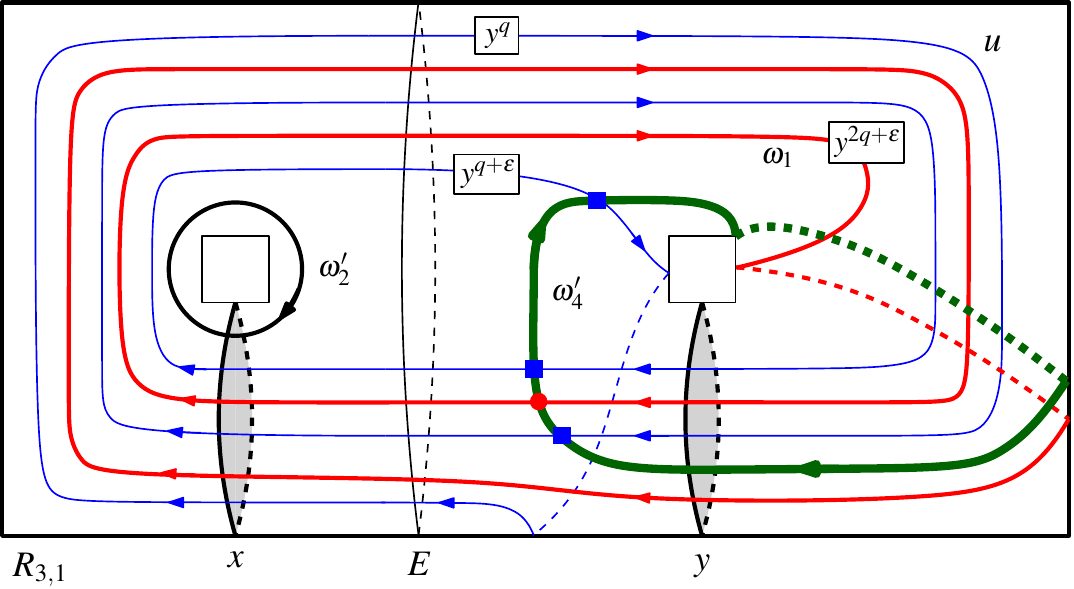}{
The circles $\omega_1,\omega'_2,\omega'_4$ and $u$ in $\partial R_{3,1}$ for $(q,\ve)=(1,-1)$.
}{k18}
\end{figure}

\begin{figure}
\Figw{5in}{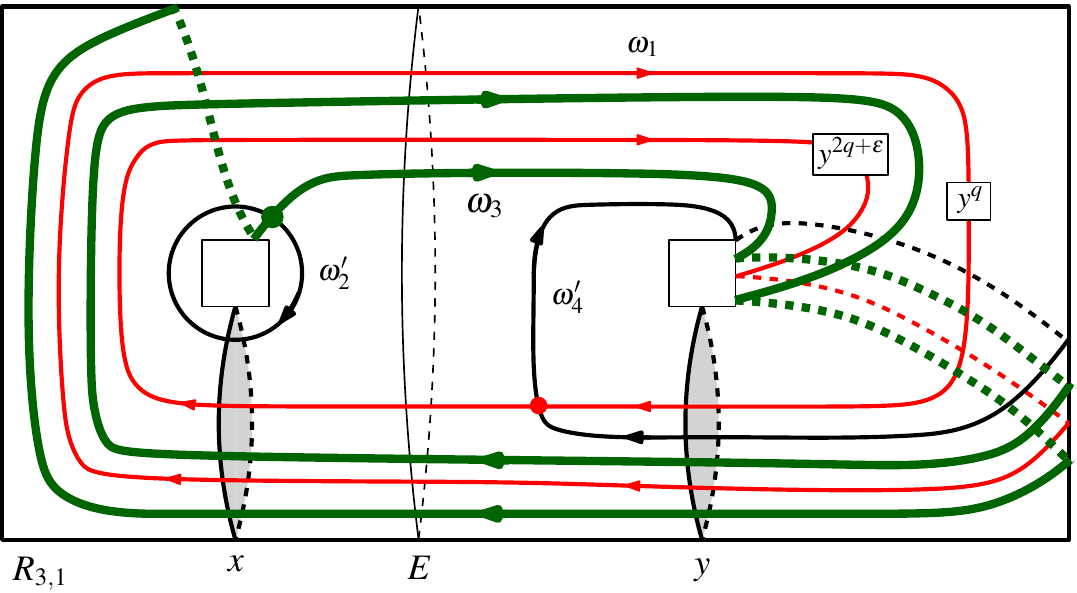}{The circles $\omega_1,\omega'_2,\omega_3,$ and $\omega'_4$ in $\partial R_{3,1}$.}{k18c}
\end{figure}

\begin{figure}
\Figw{5in}{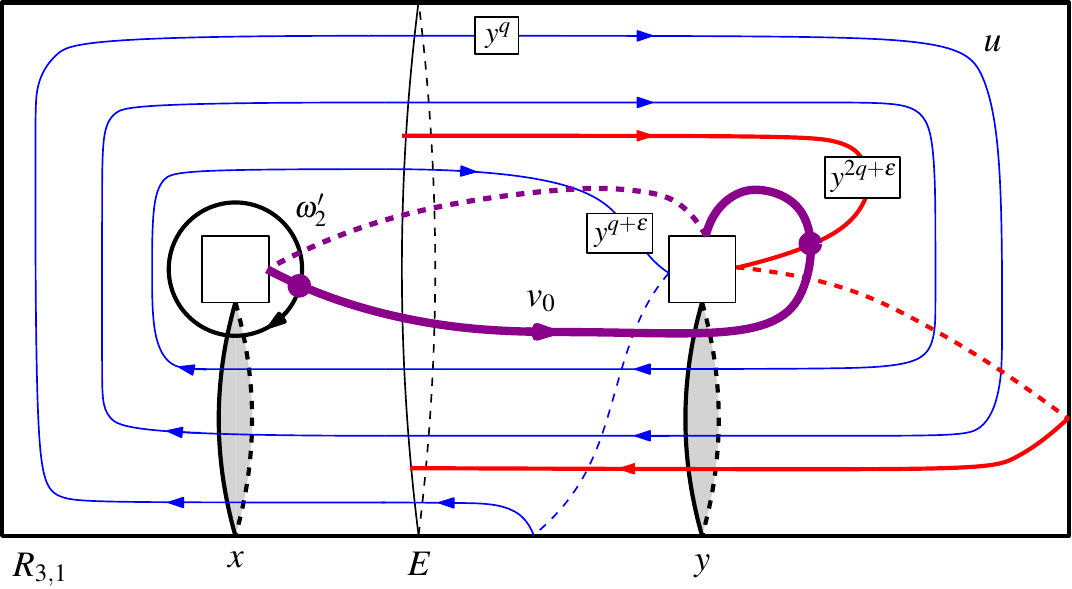}{The circles $\omega'_2$, $u$ and $v_0$ in $\partial R_{3,1}$.}{k18d}
\end{figure}

The circle $\omega_3\subset T_2$ is constructed in Fig.~\ref{k18c}.
With their given orientations these circles satisfy the relations
\begin{align*}
(\omega_1\cup\omega'_4)\cap(\omega'_2\cup\omega_3)=\emptyset
\quad &\text{and} \quad
\Delta(\omega_1,\omega'_4)=1=\Delta(\omega'_2,\omega_3)
\\
\omega_1(x,y)\equiv xy^qxy^{2q+\ve}
\quad &\text{and} \quad
\omega'_4(x,y)\equiv y^{2q+\ve}
\\
\omega'_2(x,y)=x
\quad &\text{and} \quad
\omega_3(x,y)\equiv (xy^{2q+\ve})^2 
\end{align*}

Therefore we define
\[
K=\partial N(\omega_1\sqcup\omega'_4)\subset\partial R_{3,1}
\]

Notice that for $(q,\ve)=(1,-1)$ we have $\omega_1(x,y)\equiv(xy)^2\equiv\omega_3(x,y)$, and in fact from Fig.~\ref{k18c} it follows directly that in this case the power circles $\omega_1$ and $\omega'_3$ are coannular in $R_{3,1}$.

\medskip
(II) {\bf Construction of the handlebody $R_{1,3}$.}

We construct two disjoint and nonseparating circles $u,v\subset \partial R_{3,1}=\partial R_{1,3}$ representing the boundary of the complete system of disks for $R_{1,3}$.

The circle $u=\partial D\subset\partial R_{3,1}=\partial R_{1,3}$ is given in Fig.~\ref{k18} and satisfies the relations
\begin{align*}
u\cap(\omega_1\sqcup\omega'_2)=\emptyset
\quad &\text{and} \quad
u(x,y)=(xy^q)^3y^{\ve}=\text{ primitive in }R_{3,1}
\end{align*}

To obtain the circle $v\subset\partial R_{1,3}\setminus u$ we first construct the auxiliary circle $v_0\subset\partial R_{1,3}$ in Fig.~\ref{k18d} such that
\[
v_0\cap u=\emptyset,\quad
|v_0\cap\omega'_2|=1 \quad\text{and}\quad v_0(x,y)=y^{\ve}\text{ in }\pi_1(R_{3,1})
\]
The circle $v\subset\partial R_{1,3}\setminus u$ is then constructed in $\partial R_{1,3}$ as the homological sum 
\[
v=(1+3k)\cdot\omega'_2+[q+k(3q+\ve)]\ve\cdot v_0
\]
where $k\in\mZ$.

\medskip
(III) {\bf The Heegaard decomposition $R_{1,3}\cup_{\partial}R_{3,1}\approx\mS^3$.}
\\
In the first integral homology group $H_1(R_{3,1})=H_1(R_{3,1};\mZ)=x\mZ\oplus y\mZ$ we have
\[
u=3x+(3q+\ve)y \quad\text{and}\quad 
v=(1+3k)\,x+[q+k(3q+\ve)]\,y
\]
where
\begin{align*}
\det\begin{bmatrix}
3 &3q+\ve\\
1+3k &q+k(3q+\ve)
\end{bmatrix}=-\ve =\pm 1
\end{align*}
Therefore
\begin{align*}
H_1(R_{3,1}(u\sqcup v))&= 
x\mZ\oplus y\mZ/
\grp{3x+(3q+\ve)y, \ (1+3k)\,x+[q+k(3q+\ve)]\,y}
=0
\end{align*}
and since
\[
R_{1,3}\cup_{\partial}R_{3,1}\approx R_{3,1}(u\sqcup v)=R_{3,1}(u)(v)
\]
and $R_{3,1}(u)$ is a solid torus it follows that 
$
R_{1,3}\cup_{\partial}R_{3,1}\approx\mS^3
$.

\medskip
(IV) {\bf The exchange region $R_{1,3}$.}

By construction, in $\pi_1(R_{1,3})=\grp{u,v \ | \ -}$ we have
$\omega'_2\equiv v^{p}$ for $p=|q+k(3q+\ve)|$. 

Since $u\cap(\omega_1\sqcup\omega'_2)=\emptyset$ by (II),
the nonseparating disk $u\subset R_{1,3}$ is a compression disk  for $\partial R_{1,3}\setminus(\omega_1\sqcup\omega'_2)$ and so the circles $\omega_1$ and $\omega'_2$ are coannular in $R_{1,3}$
by \cite[Lemma 3.4(2)]{valdez14}. From Fig.~\ref{k18} we can see that $|u\cap\omega'_4|_{\text{min}}=3$ in $\partial R_{1,3}$ and so, by Lemma~\ref{dis}, $(R_{1,3},J)$ is an exchange pair iff $p=|q+k(3q+\ve)|\geq 2$.

\medskip

We summarize the information above in the next result.

\begin{prop}\label{K4}
For integers $q\geq 1$, $k$, and $\ve=\pm1$, except for $(q,k)=(1,0)$ and $(q,\ve)=(1,-1)$,
each of the knots $K$ in the family $K(q,k,\ve)\subset\mS^3$ is hyperbolic and bounds a maximal simplicial collection $\mT=T_1\sqcup T_2\sqcup T_3\sqcup T_4\subset X_K$ of 4 Seifert tori such that
\begin{enumerate}
\item
the regions $R_{1,3}$ and $R_{3,1}$ are handlebodies,

\item
$(R_{1,3},J)$ is an exchange pair of index $p=|q+k(3q+\ve)|\geq 2$,

\item
$(R_{3,4},J)$ is a simple pair of index 2 and 
$(R_{4,1},J)$ is a simple pair of index $2q+\ve\geq 3$,

\item
$\Delta(\omega'_2,\omega_3)=\Delta(\omega'_3,\omega_4)=\Delta(\omega'_4,\omega_1)=1$,

\item
the Kakimizu complex $MS(K)$ is a union of two 3-simplices intersecting in a common 2-face, and each surgery manifold $X_K(r)$ is hyperbolic whenever $\Delta(r,J)\geq 2$.
\end{enumerate}

In the two exceptional cases $(q,k,\ve)=(1,0,-1), (1,-1,-1)$ the knot $K$ is hyperbolic and bounds a maximal collection of 2 Seifert tori $T_1\sqcup T_3$ with $(R_{3,3},J)$ an exchange pair 
(where $R_{3,3}=\cl[X_K\setminus T_3\times I\,]$), and $MS(K)$ the union of two 1-simplices along a common vertex (Fig.~\ref{k15b-3}, top left).

In the exceptional case $(q,k,\ve)=(1,0,1)$ the knot $K$ is hyperbolic and bounds a unique maximal collection of 3 Seifert tori $T_1\sqcup T_3\sqcup T_4$ with no exchange region
(Fig.~\ref{k15b-3}, bottom).

In the remaining exceptional cases $(q,k,\ve)=(1,k,-1)$ with $k\neq0,-1$ the knot $K$ is a satellite of a $(2,2k+1)$ torus knot 
(Fig.~\ref{k15b-3}, top right). 
\end{prop}

\begin{proof}
Observe that 
\begin{align*}
q\geq 1\implies 0<\frac{q}{3q+\ve}<1
\quad\text{and}\quad
q\geq 2\implies 0<\frac{q\pm1}{3q+\ve}\leq \frac{q+1}{3q-1}
<1
\end{align*}
and so, for $q\geq 1$,
\begin{align*}
|q+k(3q+\ve)|\leq 1 &\iff |q+k(3q+\ve)|= 1\iff -k=\frac{q\pm 1}{3q+\ve}\in\mZ\\
&\iff
q=1\text{ and } 
\begin{cases}
k=0\\
\text{or}\\
k=-1\text{ and } q+\ve=0
\end{cases}
\end{align*}
Therefore for integers $q\geq 1$, $k$, and $\ve=\pm1$ with $(q,k)\neq (1,0)$ and $(q,\ve)\neq(1,-1)$ we have
\[
p=|q+k(3q+\ve)|\geq 2
\quad\text{and}\quad
2q+\ve\geq 3
\]
and so by (IV) $(R_{1,3},J)$ is an exchange pair. Moreover the circles $\omega_1\subset T_1$ and $\omega'_2\subset T_3$ are coannular in $R_{1,3}$ and $\omega_1\equiv v^p$ in $\pi_1(R_{1,3})$, therefore the index of each core knot $K_1,K_2\subset R_{1,3}$ in Fig.~\ref{k15b} is $p=|q+k(3q+\ve)|\geq 2$.

By (I) 
$\omega_3(x,y)\equiv (xy^{2q+\ve})^2$ and $\omega'_4(x,y)\equiv y^{2q+\ve}$ in $\pi_1(R_{3,1})$ and so $\omega_3\subset T_3$ and $\omega'_4\subset T_4$ are noncoannular power circles in $R_{3,1}$. By the argument in the proof of \cite[Lemma 3.4(3)]{valdez14} it follows that there is a properly embedded disk $E\subset R_{3,1}$ that separates $\omega_3$ and $\omega'_4$. So if $A_3$ and $A_4$ are companion annuli in $R_{3,1}$ for the circles $\omega_3$ and $\omega'_4$, respectively, then $A_3$ and $A_4$ can be isotoped to be disjoint from $E$ and hence from each other. Therefore the $J$-tori $T'_3$ and $T_4$ induced by $\omega_3$ and $\omega'_4$ in $R_{3,1}$, respectively, are disjoint in $R_{3,1}$.

As we also have $\omega'_2(x,y)=x$ in $\pi_1(R_{3,1})$, the circle $\omega'_2\subset T_3$ is primitive in $R_{3,1}$ and so the pair $(R_{3,1},J)$ is not maximal by \S\ref{max}. This implies that the $J$-tori $T'_3$ and $T_4$ are mutually parallel in $R_{3,1}$, and since by construction they cobound simple pairs with $T_3$ and $T_1$, respectively, by \cite[Lemma 6.8(2)]{valdez14} $(R_{3,1},J)$ is a double pair with $T_4\subset R_{3,1}$ the unique $J$-torus not parallel into $T_3$ nor $T_1$ and $\Delta(\omega'_3,\omega_4)=1$ in $T_4$.

It now follows that the indices of the core knots $K_3$ and $K_4$ of the simple pairs $(R_{3,4},J)$ and $(R_{4,1},J)$ are $2$ and $2q+\ve\geq 3$, respectively. 

Finally, in the case of $T_1$, the circle $\omega_1\subset T_1$ has a companion annulus in $R_{1,3}$ since it is the boundary of a spanning annulus in $R_{1,3}$ of index $p\geq 2$, while $\omega'_4\subset T_1$ has a companion annulus in $R_{3,1}$ since it is a power circle in $R_{3,1}$.
As $\Delta(\omega_1,\omega'_4)=1$ and by construction each region $R_{1,3}$ and $R_{3,1}$ is a handlebody, hence atoroidal, it follows from Lemma~\ref{compa2} that no circle on $T_1$ has a companion annulus on either side of $T_1$. A similar conclusion holds for the Seifert torus $T_3\subset X_K$ using the circles $\omega'_2\sqcup\omega_3$.

Therefore, by \cite[Lemma 8.1]{valdez14} applied to the simplicial collection $T_1\sqcup T_3\subset X_K$, the knot $K=K(q,k,\ve)\subset\mS^3$ is hyperbolic and each surgery manifold $X_K(r)$ is hyperbolic for any slope $r\subset\partial X_K$ with $\Delta(r,J)\geq 2$. 

Since each of the pairs $(R_{i,i+1},J)$ is minimal and $R_{1,3}$ is an exchange region, by Proposition~\ref{mainp} the collection $\mT=T_1\sqcup T_2\sqcup T_3\sqcup T_4$ is maximal and $MS(K)$ consists of two 3 simplices intersecting in a common 2-face. Therefore (1)--(5) hold.

\begin{figure}
\Figw{5.5in}{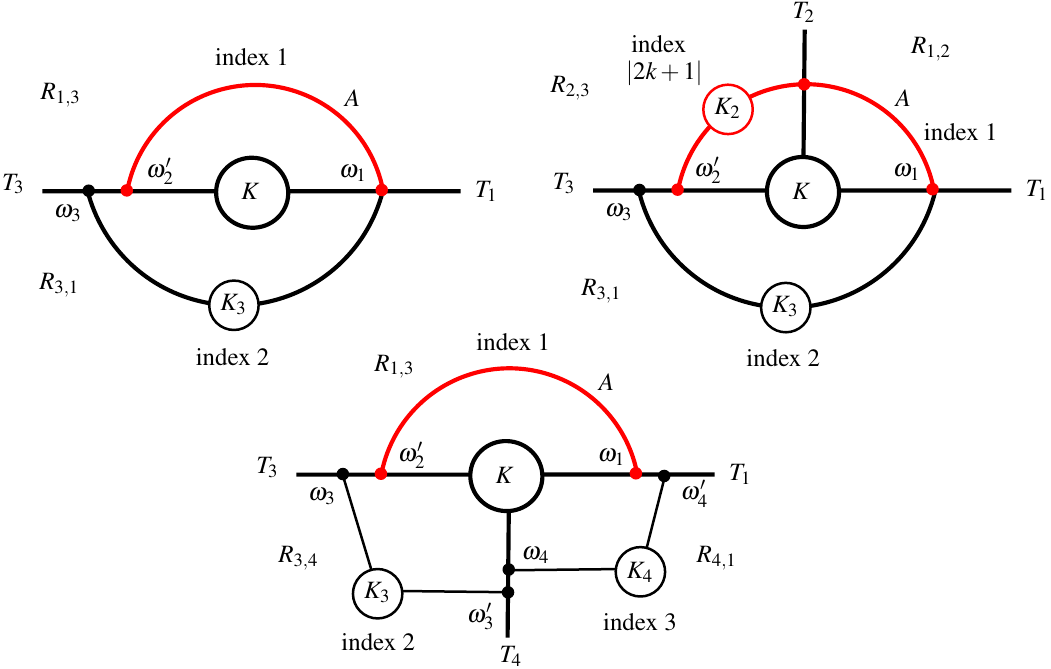}{
The Seifert tori bounded by the knot $K(q,k,\ve)$ in the exceptional cases.}{k15b-3}
\end{figure}

In the cases $(q,\ve)=(1,-1)$ we have  $2q+\ve=1$ and by definition the simple pair $(R_{4,1},J)$ degenerates into a trivial pair. Moreover, by (I) the 2-power circles $\omega_3\subset T_3$ and $\omega_1\subset T_1$ become coannular in $R_{3,1}$, so $(R_{3,1},J)$ becomes a simple pair of index 2. 

If $k=0,-1$ then $p=1$ and so $(R_{2,3},J)$ becomes a trivial pair, whence $(R_{1,3},J)$ becomes a primitive pair with primitive circles $\omega_1$ and $\omega'_2$. It follows that the region $R_{3,3}$ is an exchange region for the collection $T_1\sqcup T_3$. Since $\Delta(\omega'_2,\omega_3)=1$, by the above general argument the knot $K$ is hyperbolic and bounds the maximal simplicial collection $T_1\sqcup T_3$ with exchange region $R_{3,3}$, hence $MS(K)$ is the union of two 1-simplices along a vertex (see Fig.~\ref{k15b-3}, top right).

If $k\neq 0,-1$ then $p=|2k+1|\geq 3$ and so the spanning annulus $A$ in $R_{1,2}$ has companion annuli on either boundary circle, so $K$ is not a hyperbolic knot by Lemma~\ref{ann-powers}(4); more precisely, by \cite[Lemma 5.1]{valdez14} the knot $K$ is a satellite of a $(2,2k+1)$ torus knot (see Fig.~\ref{k15b-3}, top left).

In the last exceptional case $(q,k,\ve)=(1,0,1)$ we have $p=1$ and $2q+\ve=3$. By a similar argument it follows that $K$ is a hyperbolic knot that bounds the unique maximal simplicial collection $T_1\sqcup T_3\sqcup T_4$ of 3 Seifert tori with no exchange region (see Fig.~\ref{k15b-3}, bottom) and so $MS(K)$ is a single 2-simplex.
\end{proof}

\subsection{Hyperbolic knots with $|\mT|=2,3$ and one exchange handlebody region.}\label{hk23}

We construct a family of hyperbolic knots $K=K(-1,n,2)\subset\mS^3$, $n\in\mZ$, each of which bounds a maximal simplicial collection $\mT\subset X_K$ of $|\mT|=2,3$ Seifert tori with a handlebody exchange region $R_{1,3}$. A projection of the knot $K(-1,n,2)$ with at most $14+6|n|$ crossings is shown in Fig.~\ref{m19}.

Unlike the case $|\mT|=4$ of \S\ref{82}, for $|\mT|=2,3$ the exchange region $R_{1,3}$ need not be a handlebody; examples where $R_{1,3}$ is not a handlebody will be constructed in Section~\ref{hk23-2}.

\begin{figure}
\Figw{4in}{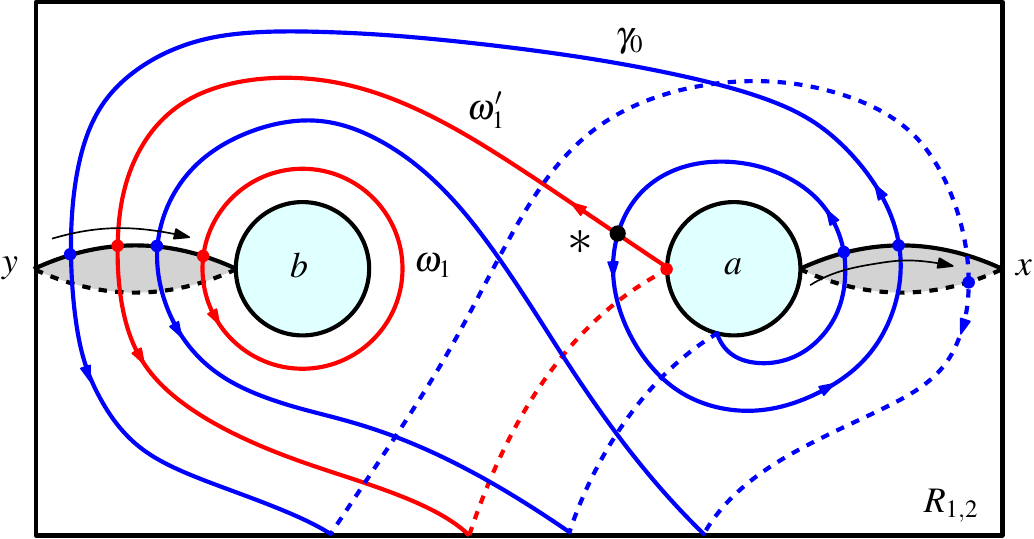}{
The complete disk systems $x\sqcup y\subset R_{1,2}$, $a\sqcup b\subset R_{2,1}$ and the circles $\omega_1,\omega'_1,\gamma_0\subset\partial R_{1,2}$.}{k28}
\end{figure}

\medskip
{\bf (I) Construction of the primitive pair $(R_{1,2},J)$.}

Fig.~\ref{k28} shows a genus two handlebody $R_{1,2}$ standardly embedded in $\mS^3$ with the following features:
\begin{enumerate}
\item
The disks $x\sqcup y\subset R_{1,2}$ form a complete disk system.

\item
The complementary handlebody $R_{2,1}=\mS^3\setminus\intr\,R_{1,2}$ has complete disk system $a\sqcup b\subset R_{2,1}$.

\item
The circles $\omega_1,\omega'_1\subset\partial R_{1,2}$ are disjoint from each other and from the disk $x\subset R_{1,2}$, and each intersects the disk $y$ minimally in one point. Thus $\omega_1$ and $\omega'_1$ cobound an annulus $A\subset R_{1,2}\setminus x$, and by \cite[Lemma 3.4]{valdez14} $x$ is the unique compression disk of the surface $\partial R_{1,2}\setminus(\omega_1\sqcup\omega'_1)$.

\item
The circles $\omega'_1$ and $\gamma_0$ intersect minimally in one point labeled $*$ in the figure.

\item
$\gamma_0$ and  $\partial x$ intersect minimally in 3 points.
\end{enumerate}

For any circle $\delta\subset\partial R_{1,2}$ and integers $k,n\in\mZ$ we denote by $\delta(k,n)\subset\partial R_{1,2}$ the circle obtained by performing $k$ full Dehn twists on $\delta$ around $\partial x$ and $n$ full Dehn twists  around $\partial y$, where the Dehn twists are performed by cutting $\partial R_{1,2}$ along the circles $\partial x\sqcup\partial y$ and twisting on the side of these circles in the direction indicated by the arrows for positive twists.

For integers $k,n,p\in\mZ$ we construct the following circles in $\partial R_{1,2}$:
\begin{enumerate}\setcounter{enumi}{5}
\item
The homological sum $\gamma_p=\gamma_0+p\omega'_1\subset\partial R_{1,2}$, constructed so that it intersects $\omega'_1$ minimally in one point.

\item
The circles $\omega_1(k,n)$, $\omega'_1(k,n)$ and 
$\gamma_p(k,n)\subset\partial R_{1,2}$.

By (4) the circles $\omega'_1(k,n)$ and $\gamma_p(k,n)$ intersect minimally in one point, and $\gamma_p(k,n))$ and $\partial x$ intersect minimally in 3 points  by (5).

\item
The separating circle $J=J(k,n,p)=\partial N(\omega'_1(k,n)\cup\gamma_p(k,n))\subset\partial R_{1,2}$.
\end{enumerate}

\medskip
(II) {\bf Fundamental groups}\\
The fundamental groups of $R_{1,2}$ and $R_{2,1}$ have the presentations
\[
\pi_1(R_{1,2})=\grp{x,y \ | \ -}
\quad\text{and}\quad
\pi_1(R_{2,1})=\grp{a,b \ | \ -}
\]
relative to the base point $*=\omega'_1\cap\gamma_0$ in Fig.~\ref{k28}.
Therefore in $\pi_1(R_{1,2})$ we have
\[
\omega_1(k,n)(x,y)=\omega'_1(k,n)(x,y)=y
\quad\text{and}\quad
\gamma_p(k,n)(x,y)=xyXyxy^p 
\]
while in $\pi_1(R_{2,1})$ we compute
\[
\omega_1(k,n)(a,b)=b^n,
\quad
\omega'_1(k,n)(a,b)=b^n a
\quad\text{and}\quad
\gamma_p(k,n)(a,b)=a^k b^n A^k b^n a^{k+1} (b^n a)^p
\]
where $X=x^{-1}$ and $A=a^{-1}$ as usual. 
Thus $\omega'_1(k,n)$ is a primitive circle in $R_{2,1}$.

\medskip
(III) {\bf The knot $K=K(k,n,p)\subset\mS^3$.}\\
The circle $J=J(k,n,p)$ separates $\partial R_{1,2}$ into two once-punctured tori $T_1,T_2\subset\partial R_{1,2}$ with $\omega_1(k,n)\subset T_1$ and $\omega'_1(k,n)\cup\gamma_p(k,n)\subset T_2$. We let $K=K(k,n,p)\subset\mS^3$ be the knot represented by $J(k,n,p)$ and consider $T_1,T_2\subset X_K$ as Seifert tori for $K$.

By (I)(3), (I)(7) and Lemma~\ref{dis} it follows that the pair $(R_{1,2},J)$ is primitive. 

Since by (II) the circle $\omega'_1(k,n)\subset T_2$ is primitive in $R_{2,1}$, the pair $(R_{2,1},J)$ is not maximal by \S\ref{max}.

\medskip
(IV) {\bf The power circles $\omega_1(k,n),\gamma_p(k,n)$ in $R_{2,1}$.}\\
By \cite{cassongor}, $\gamma_p(k,n)\subset T_2$ is a power circle in $R_{2,1}$ iff the word $\gamma_p(k,n)(a,b)$ is a power of some primitive word $w(a,b)$ in $\pi_1(R_{2,1})$, where by \S\ref{prim-pwr} in the cyclic reduction of $w(a,b)$ all exponents of $a$ ($b$) are $1$ or all $-1$ while all the exponents of $b$ ($a$, respectively) are of the form $\ell,\ell+1$ for some integer $\ell$.

Now, the following words are powers in $\pi_1(R_{2,1})$:
\begin{align*}
\omega_1(k,n)&=b^n \quad\text{for}\quad |n|\geq 2,\\
\gamma_{2}(-1,n)&\equiv (b^{2n}a)^2 \quad\text{for all}\quad n,\\
\gamma_{-1}(0,n)&=b^n \quad\text{for}\quad |n|\geq 2
\end{align*}
and we claim that these are the only cases when both $\omega_1(k,n)$ and $\gamma_p(k,n)$ are power circles in $R_{2,1}$. Indeed, suppose that $|n|\geq 2$ so that $\omega_1(k,n)$ is a power circle. 

If $k\neq 0,-1$ then the cyclic reduction of the  word $\gamma_p(k,n)(a,b)$ contains both $a$ and $A$ factors and hence it is not a power.
If $k=0$ then $\gamma_p(0,n)=b^n (b^n a)^{p+1}$ is a power iff $p=-1$, while if $k=-1$ then 
\[
\gamma_p(-1,n)(a,b)=A b^n a b^n (b^n a)^p = 
AB^n\cdot (b^{2n}a)^2 \cdot(b^n a)^{p-2}\cdot b^n a\equiv
(b^{2n}a)^2 \cdot(b^n a)^{p-2}
\]
is a power iff $p=2$.

\medskip
The next result summarizes the information above.

\begin{prop}\label{K3} For $|n|\geq 2$,
\begin{enumerate}
\item
the knot $K=K(-1,n,2)\subset\mS^3$ is hyperbolic and bounds a maximal  collection of 3 Seifert tori $\mT=T_1\sqcup T_2\sqcup T_3\subset X_K$ with one exchange pair $(R_{3,2},J)$ of index $|n|$ and a simple pair $(R_{2,3},J)$ of index $2$, and $MS(K)$ is the union of two 2-simplices along a common 1-subsimplex (see Fig.~\ref{m19} and Fig.~\ref{k15-2}, left)

\item
the knot $K=K(0,n,-1)\subset\mS^3$ is hyperbolic and bounds a maximal  collection of 2 Seifert tori $\mT=T_1\sqcup T_2\subset X_K$ with one exchange pair $(R_{1,1},J)$ of index $|n|$, and $MS(K)$ is the union of two 1-simplices along a common vertex (see Fig.~\ref{k15-2}, right).
\end{enumerate}
\end{prop}

\begin{figure}
\Figw{5.5in}{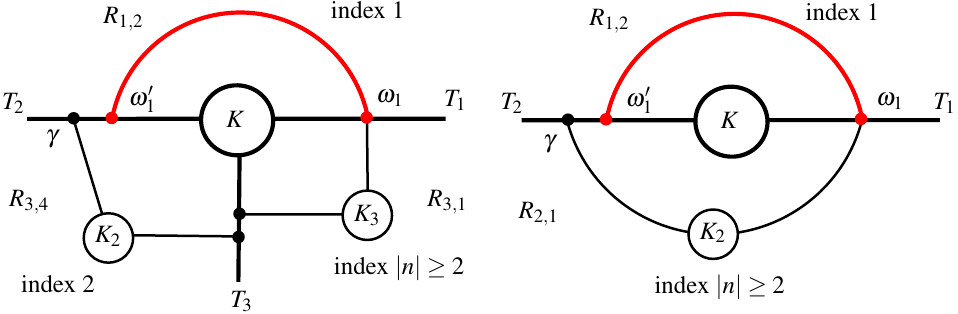}{
The Seifert tori bounded by the knots $K(-1,n,2)$ and $K(0,n,-1)$.}{k15-2}
\end{figure}

\begin{proof}
We sketch the proof following the argument in Proposition~\ref{K4} closely. 

By (III), for all $k,n,p$ the pair $(R_{1,2},J)$ with $J=J(k,n,p)$ is primitive  with primitive circles $\omega_1(k,n)\subset T_1$ and $\omega_1'(k,n)\subset T_2$. Let $|n|\geq 2$.

For the knot $K=K(-1,n,2)\subset\mS^3$, by (IV) the circles $\omega_1(-1,n)\subset T_1$ and $\gamma_2(-1,n)\subset T_2$ are power circles in $R_{2,1}$ with words of the form $b^n$ and $(b^{2n}a)^2$, respectively, and hence are not coannular in $R_{2,1}$. Since by (III) the pair $(R_{2,1},J)$ is not maximal, 
either power circle $\omega_1(-1,n)$ or $\gamma_2(-1,n)$ induces a
$J$-torus $T_3\subset R_{2,1}$ which
splits the pair $(R_{2,1},J)$ into two simple pairs 
$(R_{2,3},J)$ and $(R_{3,1},J)$ of indices $2$ and $|n|$, respectively.
Therefore $R_{3,2}=R_{3,1}\cup_{T_1} R_{1,2}$ is an exchange region of index $|n|$.

\medskip
For the knot $K=K(0,n,-1)\subset\mS^3$, by (IV) we have $\omega_1(0,n)=b^n\equiv\gamma_{-1}(0,n)$ in $\pi_1(R_{2,1})$, and it can be seen directly from the corresponding diagram in Fig.~\ref{k28} that $\omega_1(0,n)$ and $\gamma_{-1}(0,n)$ are coannular $|n|$-power circles in $R_{2,1}$. It is also not hard to see that $a\subset R_{2,1}$ is the compression disk of the surface $\partial R_{2,1}\setminus[\omega_1(0,n)\sqcup\gamma_{-1}(0,n)]$.  Since the disk $a\subset R_{2,1}$ intersects $\omega'_1(0,n)\subset T_2$ minimally in one point, by the definition of $J$ and Lemma~\ref{dis} it follows that $(R_{2,1},J)$ is a simple pair of index $|n|$. Therefore 
$R_{1,1}=\cl[X_K\setminus N(T_1)]$ is an exchange region of index $|n|$.

As in the proof of Proposition~\ref{K4}, that the knots $K(-1,n,2)$ and $K(0,n,-1)$ are hyperbolic now follows from \cite[Lemma 8.1]{valdez14}. Therefore (1) and (2) hold.
\end{proof}

\subsection{Hyperbolic knots with $|\mT|=4$, two hyperbolic pairs, and no exchange region.}\label{ht4}
We will use the notation set up in \S\ref{FF} in the classification of basic pairs.

Fig.~\ref{k04}, top, shows a basic pair $(H,J)$ constructed as in \S\ref{FF} using basic circles $\alpha\sqcup\beta\subset\partial H$ separated by the disk $D\subset H$, with parameters set for this example as $m=4$ and $n=3$.  By Lemma~\ref{baspa2} the pair $(H,J)$ is therefore hyperbolic. 

Cutting $H$ along $D$ produces two solid tori $V_1,V_2$ with $H=V_1\cup(D\times[-1,1])\cup V_2$ and $\alpha\subset\partial V_1$, $\beta\subset\partial V_2$, as shown in Fig.~\ref{k04}.

We assume that the handlebody $H$ is standardly embedded in $\mS^3$ with complete disk system $D_1\sqcup D_2$.
Its complementary handlebody $H'=\mS^3\setminus\intr\,H$ has complete disk system $D'_1\sqcup D'_2$ such that 
$|\partial D_i\cap\partial D'_j|=0$ for $i\neq j$ and $1$ for $i=j$. 

Specifically the following circles are constructed on $\partial H=\partial H'$.

\begin{figure}
\Figw{5in}{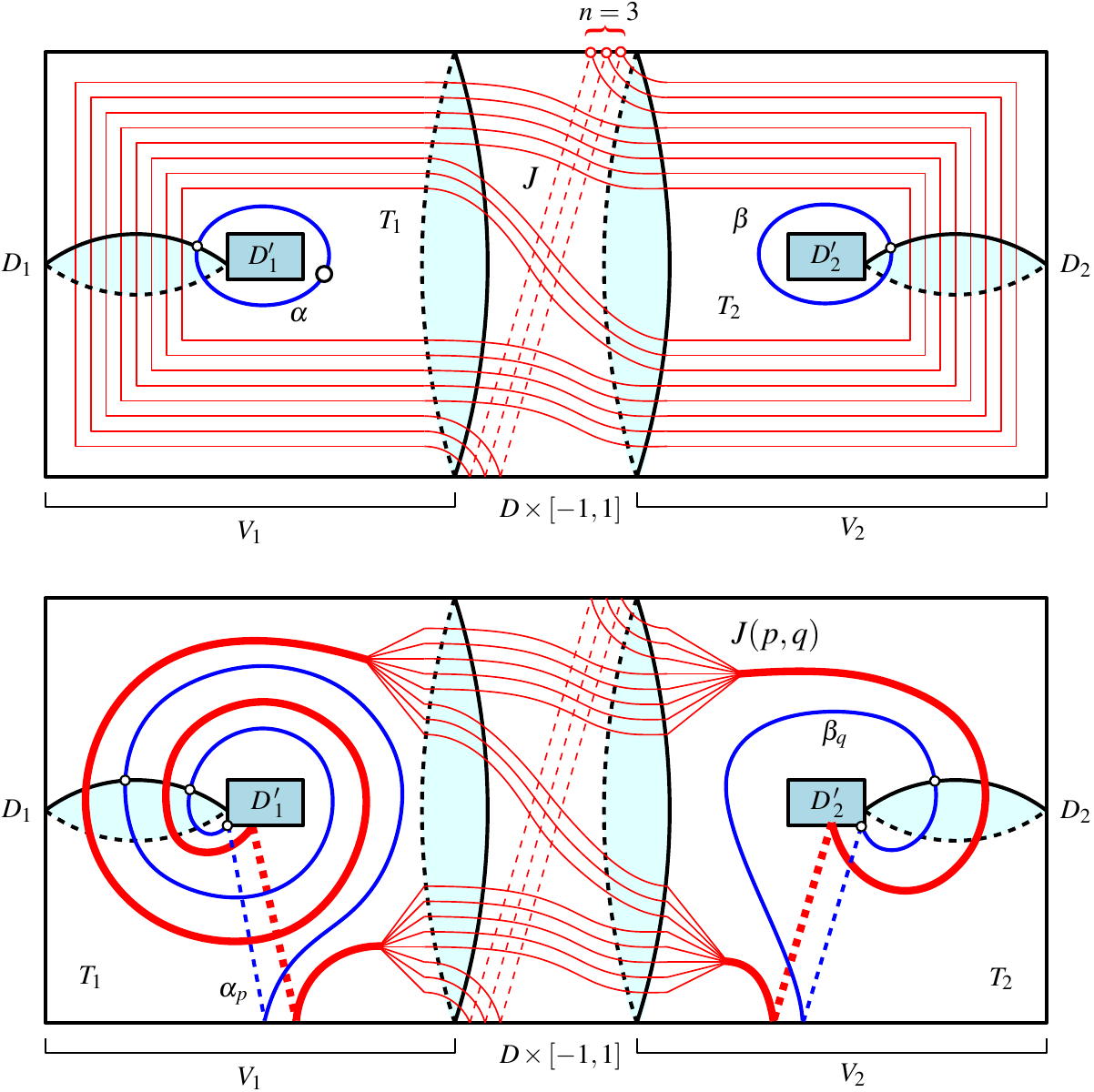}{The hyperbolic basic pair $(H,J)$ (top) and the associated hyperbolic pair $(H,J(p,q))$ for $p=2$, $q=1$ (bottom).}{k04}
\end{figure}

\begin{enumerate}
\item[(H1)]
Basic circles $\alpha\sqcup\beta\subset\partial H$ parallel to $\partial D'_1$ and $\partial D'_2$, respectively.

\item[(H2)]
The circle $\partial D\subset\partial H=\partial H'$ bounds a nontrivial separating disk $D'\subset H'$.

\item[(H3)]
We write $\partial H=T_1\cup_JT_2$ with $\alpha\subset T_1$
and $\beta\subset T_2$.

\item[(H4)]
For integers $p,q\in\mZ$ (and some orientation scheme) let $\alpha_p\subset T_1$  be a circle homologous to $\partial D_1+p\alpha$, $\beta_q\subset T_2$ a circle homologous to $\partial D_2+q\alpha$, and construct the separating circle $J(p,q)\subset\partial H$ by matching the endpoints of 8 arcs in $\partial V_1\setminus\alpha_p$ with those of  8 arcs in $\partial V_2\setminus\beta_q$ using the same pattern as for $J$.

The circle $J(p,q)$ is represented in Fig.~\ref{k04}, bottom, for $|p|=2$ and $|q|=1$.
\end{enumerate}

\medskip
We denote by $K=K(p,q)$ the knot corresponding in $\mS^3$ corresponding to $J(p,q)$. Thus $K$ bounds two simplicial Seifert tori $T_1\sqcup T_2\subset X_K$ with $R_{1,2}=H$ and $R_{2,1}=H'$.

Equivalently $K(p,q)$ is the knot obtained by performing $p$ and $q$ full twists on the indicated strands of the trivial knot $K(0,0)$ shown in Fig.~\ref{k04-3}. 

\begin{figure}
\Figw{3.7in}{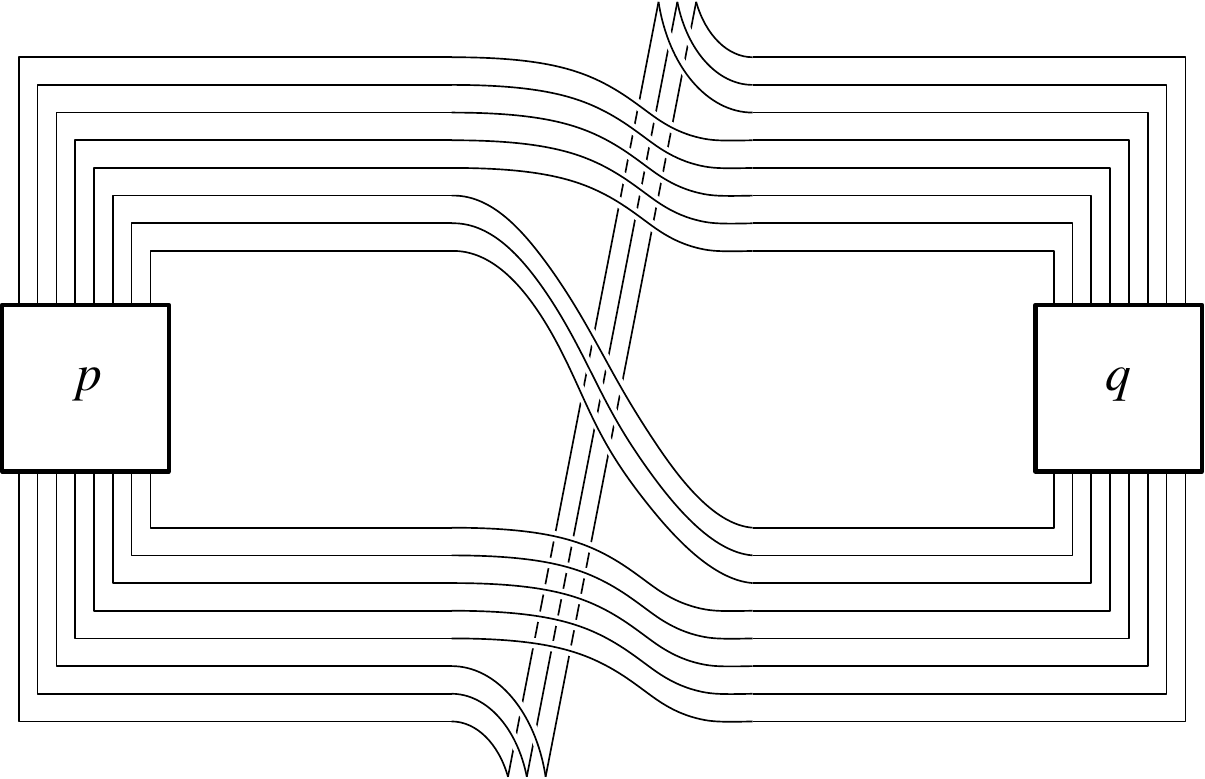}{The knot $K(p,q)\subset\mS^3$.}{k04-3}
\end{figure}

If $p=0$ then $T_1$ compresses in $H'$ along the disk $D'_1$ an so the knots $K(0,q)$ are trivial. Similarly the knots $K(p,0)$ are trivial.
The knot $K(2,2)$ is represented in Fig.~\ref{k04-2e}.

\medskip
\begin{enumerate}
\item[(H5)]
By (H2) and (H4) the circles $\alpha_p,\beta_q\subset\partial H'$ are basic circles in $H'$ separated by the disk $D'\subset H'$ with $\partial D'=\partial D$. 

Recall from \S\ref{FF} that the construction parameters $m,n$ of a circle like $J(p,q)\subset\partial H'$ depend only on the distribution of parallel arcs in the intersection of $J(p,q)$ with the annulus $(\partial D')\times I\subset\partial H'$. As $\partial D'=\partial D$ the circles $J$ and $J(p,q)$ share the same parameters $m=4,n=3$ and so  the basic pair $(H',J(p,q))$ is hyperbolic and homeomorphic to $(H,J)$.

Similarly, for $|p|=1=|q|$ the pair $(H,J(p,q))$ is hyperbolic and homeomorphic to $(H,J)$.
\end{enumerate}

\medskip

\begin{prop}\label{Khyper}
For integers 
$p,q\neq 0$ the knot $K=K(p,q)\subset\mS^3$ is hyperbolic. Specifically, setting $J^*=J(p,q)$,
\begin{enumerate}
\item 
For $|p|,|q|\geq 2$ the knot $K$ bounds a unique maximal simplicial collection of 4 Seifert tori $\mT=T_1\sqcup T_2\sqcup T_3\sqcup T_4\subset X_K$, such that each pair $(R_{1,2},J^*)$ and $(R_{3,4},J^*)$ is homeomorphic to the hyperbolic pair $(H,J)$, while each pair $(R_{2,3},J^*)$ and $(R_{4,1},J^*)$ is simple of index $|q|$ and $|p|$, respectively.

\medskip
\item
If $|p|\geq 2$ and $|q|=1$ then in (1) the pair $(R_{4,1},J^*)$ becomes a trivial pair and so $K$ bounds the unique maximal simplicial collection of 3 Seifert tori $T_1\sqcup T_2\sqcup T_3$ with the pairs $(R_{1,2},J^*),(R_{3,1},J^*)$ homeomorphic to the hyperbolic pair $(H,J)$ and $(R_{2,3},J^*)$ a pair of index $|p|$. A similar conclusion holds when $|p|=1$ and $|q|\geq 2$. 

\medskip
\item 
If $|p|=1=|q|$ then $K$ bounds the unique maximal simplicial collection of two Seifert tori $T_1\sqcup T_2\subset X_K$ with
$(R_{1,2},J^*),(R_{2,1},J^*)$ homeomorphic to the hyperbolic pair $(H,J)$.
\end{enumerate}
\end{prop}

\begin{proof}
Set $R_{1,2}=H'$, $R_{2,1}=H$ and $J^*=J(p,q)$. 

Suppose that $|p|,|q|\geq 2$.
Then in $R_{2,1}=H$ the circle $\alpha_p\subset T_1$ is a $|p|$-power circle and $\beta_q\subset T_2$ is a $|q|$-power circle. Let $T_3,T_4\subset R_{2,1}$ be the $J^*$-tori induced by $\beta_q$ and $\alpha_p$, respectively, so that $(R_{2,3},J^*)$ and $(R_{4,1},J^*)$ are simple pairs of indices $|q|$ and $|p|$, respectively (see \S\ref{induced}).

Let $B,V\subset R_{2,3}$ be the companion annulus and solid torus of the power circle $\beta_q\subset T_2$, with $V\cap T_2$ an annular regular neighborhood of $\beta_q$ in $T_2$. Similarly let
$A,W\subset R_{2,3}$ be the companion annulus and solid torus of the power circle $\alpha_p\subset T_1$, with $W\cap T_1$ an annular regular neighborhood of $\alpha_p$. The situation is represented in Fig.~\ref{k05c}.

\begin{figure}
\Figw{3.4in}{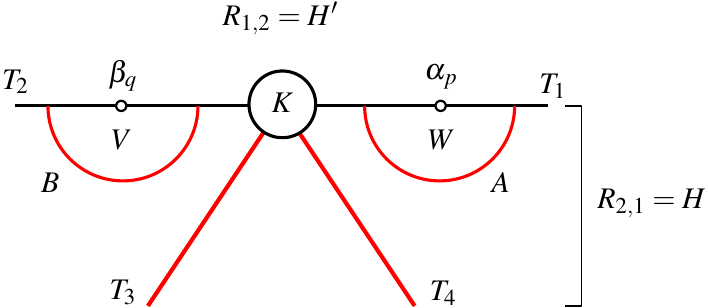}{Construction of the Seifert tori $T_3,T_4\subset X_K$.}{k05c}
\end{figure}

Then $J^*$ is disjoint in $\partial R_{1,2}$ of the annuli $W\cap T_1$ and $V\cap T_2$, and we may assume that $V$ and $W$ are disjoint from the separating disk $D\subset H$.

Let $H^*\subset R_{2,1}=H$ be the genus two handlebody component of $R_{2,1}$ cut along the companion annuli $A\sqcup B$. By \S\ref{oper}(1) the pair $(H^*,J^*)$ is homeomorphic to $(R_{3,4},J^*)$,
and by \cite[Lemma 6.8]{valdez14}
the core circles $\alpha'_p$ and $\beta'_q$ of the companion annuli $A$ and $B$ are basic circles in $H^*$ disjoint from $J^*$. 

The disk $D\subset H$ separates the solid tori $V$ and $W$ and hence it lies in $H^*$, which by the argument in (H5) implies that the pairs $(H^*,J^*)\approx(R_{3,4},J^*)$ are homeomorphic to $(H,J)$, hence hyperbolic.
Thus (1) holds, and (2) and (3) follow by a similar argument.
\end{proof}

\begin{figure}
\Figw{4.5in}{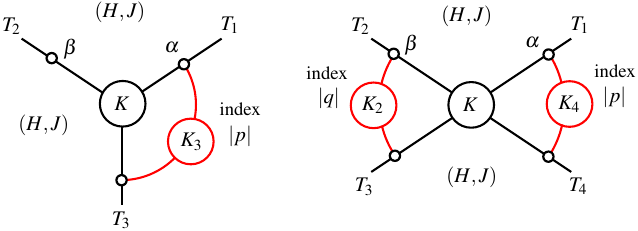}{The Seifert tori in $X_K$ for $|p|\geq 2,|q|=1$ (left) and $|p|,|q|\geq 2$ (right).}{k05}
\end{figure}

\begin{rem}\label{rem-hyper}
\begin{enumerate}
\item
In Fig.~\ref{k05}, right, 
$K_2\sqcup K_4\subset\mS^3$ is the unlink formed by the cores of the solid tori components of $H|D_0$ in Fig.~\ref{k04}, bottom. This is one of the reasons why it is easy to render a projection of the knots $K(p,q)$ as in Fig.~\ref{k04-3}.
Moreover the regions $R_{1,3}$ and $R_{3,1}$ are handlebodies but $R_{2,1}$ and $R_{4,3}$ are not, making the latter the exteriors of nontrivial handlebody knots in $\mS^3$ with an interesting internal structure. 

\item
Similarly in Fig.~\ref{k05}, left, the knot $K_3\subset\mS^3$ is trivial and
the region $R_{1,3}$ is not a handlebody.

\item
We have used the values $m=4,n=3$ for the parameters defined in \S\ref{FF}; conclusions similar to those reached in Proposition~\ref{Khyper} can be reached for any other suitable values of $m,n$.
\end{enumerate}
\end{rem}

\subsection{Hyperbolic knots with $|\mT|=2,3$ and one exchange nonhandlebody region.}\label{hk23-2}

In this section we construct hyperbolic knots that bound a maximal simplicial collection $\mT\subset X_K$ of $|\mT|=2,3$ Seifert tori with a nonhandlebody exchange region $R_{1,3}$.

The pair $(R_{3,1},J)$ is necessarily basic by Lemma~\ref{koz1-b}(2)
and in the context of \cite{ozawa01} represents a handlebody knot in $\mS^3$ whose exterior $R_{1,3}$ contains an essential annulus. Such handlebody knots are completely classified in \cite{ozawa01} and the spanning annulus in $R_{1,3}$ corresponds to a Type 3-3 annulus as constructed in \cite[\S3, Fig.~8]{ozawa01}.

As in \S\ref{RK} we retract the regular neighborhood $N(K)\subset\mS^3$ radially onto $K$ so that the circles $J$ and $\partial T_i$ become identified with $K$ and use the notation $(R_{i,i+1},K)$ for the pairs $(R_{i,i+1},J)$
 
\medskip
{\bf (I) Construction of handlebody knots $V\subset\mS^3$ with an incompressible  Type 3-3 annulus in their exterior.}

(i) Let $L\subset\mS^3$ be a fixed knot with exterior $X_L=\mS^3\setminus\intr\,N(L)\subset\mS^3$ and let $\alpha\subset\partial X_L=\partial N(L)$ be a circle of nonintegral slope of the form $r=a/p$ for some $p\geq 2$. If $L$ is the trivial knot we choose $\alpha$ to be a nontrivial torus knot.

Let $L'\subset \intr\,X_L$ be a knot which cobounds an annulus with $\alpha$; thus $L'\subset\mS^3$ is a nontrivial knot which is a cable of $L$. The exterior $X=\mS^3\setminus\intr[N(L)\cup N(L')]$ of the link $L\sqcup L'\subset\mS^3$ then contains a properly embedded annulus $A$ with the circle $\alpha=\partial A\cap\partial N(L)$ of nonintegral slope in $N(L)$, and the circle $\partial A\cap\partial N(L')$ of integral slope $r'$ in $\partial N(L')$. The situation is represented in Fig.~\ref{m03}, top.

Since at least one component of the link $L\sqcup L'$ is a nontrivial knot, the boundary tori $\partial X$ and the annulus $A$ are incompressible in $X$.

\medskip
(ii) Let $N(A)$ be a thin regular neighborhood of $A$ in $X$ and $M=\cl[X\setminus N(A)]$ the exterior of $A$ in $X$; thus $\partial M$ is a torus.
By \cite[Theorem 1.1]{myers1} there is a properly embedded arc $e\subset M$ such that
\begin{enumerate}
\item[(M1)]
$e$ has one endpoint on each of the annuli $\partial N(L)\cap\partial M$ and $\partial N(L')\cap\partial M$,

\item[(M2)]
$e$ is an excellent 1-submanifold of $M$; that is, 
the exterior $M_e=\cl[M\setminus N(e)]$ of $e$ in $M$ is irreducible, boundary irreducible, atoroidal
and {\it anannular} (any incompressible annulus in $M_e$ is boundary parallel). 
\end{enumerate}

\medskip
(iii) We set 
\[
V=N(L)\cup N(L')\cup N(e)\subset\mS^3
\quad\text{and}\quad
W=\mS^3\setminus\intr\, V=\cl[X\setminus N(e)]=M_e\cup N(A)
\]
Thus $V$ is a genus two handlebody knot in $\mS^3$ whose exterior $W$ contains the nonseparating annulus $A$.

Being a subset of $X\subset\mS^3$, the manifold $W$ is irreducible and the annulus $A\subset W$ is incompressible in $W$. It is then not hard to see from (M2) that $W$ is boundary irreducible and atoroidal, and that any annulus in $W$ is isotopic to $A$.

\medskip
(iv) The cocore disk $D$ of the 1-handle $N(e)$ attached to $N(L)\sqcup N(L')$ is a nontrivial disk in $V$ which separates the components of $\partial A\subset V$. 
By (i), in $V$ the circle $\alpha=\partial A\cap\partial N(L)$ is a $p\geq 2$ power circle and $\partial A\cap\partial N(L')$ is a primitive circle.
The situation is represented in Fig.~\ref{m03}, bottom.

\begin{figure}

\Figw{2.8in}{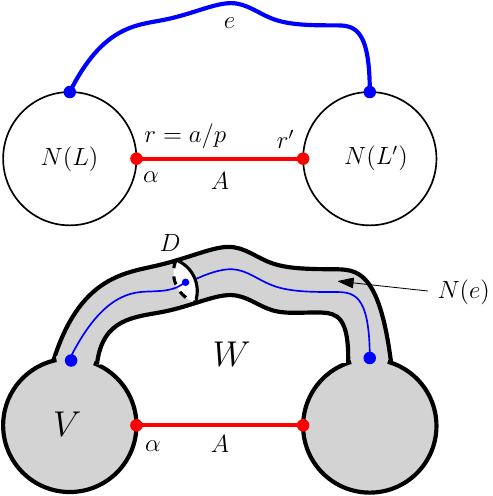}{Construction of the genus two handlebody knot $V\subset\mS^3$ with exterior $W$.}{m03}
\end{figure}

\medskip
{\bf (II) The genus one hyperbolic knots $K\subset\mS^3$.}
\\
We consider the family of knots $K\subset\mS^3$ consisting of separating circles in $\partial V=\partial W$ that are disjoint from $\partial A$ and nontrivial in $V$, or, equivalently, such that $K$ and $\partial D\subset V$  intersect  minimally in $|K\cap\partial D|\geq 4$ points. 

Since the circle $K\subset\partial V=\partial W$ is nontrivial in $V$ and $W$, $K$ is a genus one knot in $\mS^3$ that bounds two Seifert tori from the decomposition $\partial V=T_1\cup_K T_2$. We set $R_{1,2}=W$ and $R_{2,1}=V$, so that $N(L)\sqcup D\subset R_{2,1}$, and choose the notation 
$\partial_1A\subset T_1$ and $\partial_2A=\alpha\subset T_2$.

We summarize the above conclusions in the next result.

\begin{prop}\label{K3-nh}
There exist infinitely many genus one hyperbolic knots $K\subset\mS^3$ which bound a maximal simplicial collection of 2 or 3 Seifert tori with one nonhandlebody exchange region.
\hfill
\qed
\end{prop}
\begin{proof}
By (iii) and Lemma~\ref{Rpair}(2) $(R_{1,2},K)$ and $(R_{2,1},K)$ are pairs. Since the region $R_{1,2}$ contains the spanning annulus $A$, which is unique by (M2), the pair $(R_{1,2},K)$ is annular of index 1.

By (I)(iv) the circle $\partial_2A=\alpha\subset T_2$ is a power $p$-circle in $R_{2,1}$ with companion solid torus constructed out of $N(L)\subset R_{2,1}$, and we set  $T_3\subset R_{2,1}$ as the $K$-torus induced by $\alpha$ so that $(R_{2,3},K)$ is a simple pair of index $p\geq 2$.
Therefore we have that 
\begin{itemize}
\item 
the pair $(R_{1,3},K)$ is an exchange pair, 

\item
$R_{1,3}$ is not a handlebody by \S\ref{oper}(3), 

\item
the knot $K\subset\mS^3$ is hyperbolic by the argument used in Proposition~\ref{K4},

\item
$(R_{3,1},K)$ is a basic pair by
Lemma~\ref{koz1-b}.
\end{itemize}

If the pair $(R_{3,1},K)$ is trivial then $T_1$ and $T_3$ are parallel in $X_K$ and so by Proposition~\ref{mainp}(1) $\mT=T_1\sqcup T_2\subset X_K$ is a maximal simplicial collection of Seifert tori bounded by $K$. Otherwise $\mT=T_1\sqcup T_2\sqcup T_3\subset X_K$ is such a maximal simplicial collection. 

Knots with a maximal simplicial collection of size $|\mT|=3$ can be constructed by choosing $K\subset\partial V$ to follow the pattern of the knot $J(p,q)\subset\partial H$ with $q=1$ in Fig.~\ref{k04}, bottom, so that
$(V,K,\partial_1A,\partial_2A,\partial D)=(H,J(p,q),\beta_q,\alpha_p,D)$.
By the argument in the proof of Proposition~\ref{Khyper}
the pair $(R_{3,1},K)$ is then homeomorphic to the basic hyperbolic pair $(H,J(1,1))$. 

By Lemma~\ref{baspa2} replacing $J(p,q)\subset H$ in \S\ref{ht4} with a similar circle constructed with parameters $m\geq 2$ and $n=1$ then turns $(H,J(1,1))$ into a simple pair of index $m$. 

Finally, to obtain a maximal simplicial collection of size $|\mT|=2$ it suffices to choose $K\subset\partial V$ with $|K\cap\partial D|=4$ in which case the pair $R_{2,1}$ is simple of index $p$ as in Fig.~\ref{m02}, top left.
\end{proof}

\begin{proof}[Proof of Theorem~\ref{main1} part (3):]
The claim follows from Propositions~\ref{K4}, \ref{K3} and \ref{K3-nh}.
\end{proof}

\bibliographystyle{amsplain}

\providecommand{\bysame}{\leavevmode\hbox to3em{\hrulefill}\thinspace}
\providecommand{\MR}{\relax\ifhmode\unskip\space\fi MR }
\providecommand{\MRhref}[2]{%
  \href{http://www.ams.org/mathscinet-getitem?mr=#1}{#2}
}
\providecommand{\href}[2]{#2}


\end{document}